\documentclass[french, 11pt, a4paper, oneside,bibliography=totocnumbered]{scrartcl}
\usepackage[T1]{fontenc} %for font (Umlaute)
\usepackage[latin1]{inputenc} %for font
\usepackage{mathptmx} %Font Times
\usepackage{lmodern} %Font Latin Modern (better resolution of original one)
\usepackage{amsthm} %Defininig Theorems
\usepackage{amsmath} %Mathematical Extention for LaTeX
\usepackage{amsfonts} %Mathematical Fonts
\usepackage{amssymb} %Mathematical Symbols
\usepackage{mathrsfs}

\usepackage{bbm}
\usepackage[left=2.2cm, right=2.2cm, top=2.3cm, bottom=2.3cm]{geometry} %Preferences for Page
\usepackage[ngerman,german,english]{babel} %Language (Umlaute)
\usepackage{paralist} %Preferences for Enumeration, especially in Theorems: \begin{enumerate}[(i)]
\usepackage[colorlinks=true,citecolor=blue]{hyperref}
\hypersetup{linktocpage}
\usepackage{etoolbox}

\usepackage[nottoc,notlot,notlof]{tocbibind}
\usepackage{graphicx}
\usepackage{xcolor}

\setkomafont{sectionentry}{\normalsize}
\RedeclareSectionCommand[tocbeforeskip=1pt]{section}

\setkomafont{section}{\normalfont\Large\textbf}
\setkomafont{subsection}{\normalfont\large\textbf}
\setkomafont{paragraph}{\normalfont\textbf}

\renewenvironment{abstract}
{\begin{center}
		\textbf{Abstract}
	\end{center}
	\list{}{ 
		\setlength{\leftmargin}{0.05\textwidth}
		\setlength{\rightmargin}{\leftmargin}
	}
	\item\relax} 
{\endlist}

\newenvironment{keywords}
{\begin{trivlist}\item[]{\bfseries Keywords.}}
	{\end{trivlist}}

\numberwithin{equation}{section}
\theoremstyle{plain} %all theorms are numbered together as part of Theorem; [thm]
\newtheorem{theorem}{Theorem}[section] %for short documents no [section] require for undernumbering

\newtheorem{lemma}[theorem]{Lemma}
\newtheorem{proposition}[theorem]{Proposition}

\newtheorem{definition}[theorem]{Definition}
\newtheorem{remark}[theorem]{Remark}

\newcommand{\N}{\mathbb{N}}

\newcommand{\R}{\mathbb{R}}

\newcommand\restr[2]{{% we make the whole thing an ordinary symbol
		\left.\kern-\nulldelimiterspace % automatically resize the bar with \right
		#1 % the function
		\vphantom{\big|} % pretend it's a little taller at normal size
		\right|_{#2} % this is the delimiter
}}

\newcommand{\norm}[2][]{\left\|#2\right\|_{#1}}

\newcommand{\Lpnorm}[2][]{\ifthenelse{\equal{#1}{}}{\norm{#2}_{L^p}}{\norm{#2}_{L^p(#1)}}}
\newcommand{\Hknorm}[2][]{\ifthenelse{\equal{#1}{}}{\norm{#2}_{H^k}}{\norm{#2}_{H^k(#1)}}}

\newcommand{\setdef}[2]{\left\lbrace #1 \ : \ #2 \right\rbrace}

\renewcommand{\Re}{\text{\normalfont Re}}
\renewcommand{\Im}{\text{\normalfont Im}}

\newcommand{\QV}[2][]{\ifthenelse{\equal{#1}{}}{\langle #1 \rangle}{\langle #1,#2 \rangle}}

\makeatletter
\newcounter{author}
\renewcommand*\author[1]{%
	\stepcounter{author}%
	\ifnum\c@author=1
	\gdef\@author{#1}%
	\else
	\xdef\@author{\unexpanded\expandafter{\@author\and#1}}%
	\fi
	\csgdef{author@\the\c@author}{#1}}
\newcommand*\email[1]{%
	\csgdef{email@\the\c@author}{#1}}
\newcommand*\address[1]{%
	\csgdef{address@\the\c@author}{#1}}
\AtEndDocument{%
	\xdef\author@count{\the\c@author}%
	\c@author=1
	\print@authors}
\newcommand*\print@authors{%
	\ifnum\c@author>\author@count
	\else
	\print@author{\the\c@author}%
	\advance\c@author by 1
	\expandafter\print@authors
	\fi}
\newcommand*\print@author[1]{%
	\par\medskip
	\begin{tabular}{@{}l@{}}%
		\textsc{\csuse{author@#1}}\\
		\csuse{address@#1}\\
		\textit{E-mail address}: \csuse{email@#1}
\end{tabular}}
\makeatother

\title{\textrm{\textbf{\Large Description of Chemical Systems by means of Response Functions}}}

\author{\large Eugenia Franco}
\address{University of Bonn, Institute for Applied Mathematics \\ Endenicher Allee 60 \\ D-53115 Bonn \\ GERMANY}
\email{franco@iam.uni-bonn.de}

\author{\large Bernhard Kepka}
\address{University of Bonn, Institute for Applied Mathematics \\ Endenicher Allee 60 \\ D-53115 Bonn \\ GERMANY}
\email{kepka@iam.uni-bonn.de}

\author{\large Juan J. L. Vel\'{a}zquez}
\address{University of Bonn, Institute for Applied Mathematics \\ Endenicher Allee 60 \\ D-53115 Bonn \\ GERMANY}
\email{velazquez@iam.uni-bonn.de}

\date{\normalsize \today}

\begin{document}
	\maketitle
	
	\begin{abstract}
 In this paper we introduce a formalism that allows to describe the response of a part of a biochemical system in terms of renewal equations.
 In particular, we examine under which conditions the interactions between the different parts of a chemical system, described by means of linear ODEs, can be represented in terms of renewal equations.
 We show also how to apply the formalism developed in this paper to some particular types of linear and non-linear ODEs, modelling some biochemical systems of interest in biology (for instance, some time-dependent versions of the classical Hopfield model of kinetic proofreading).
We also analyse some of the properties of the renewal equations that we are interested in, as the long-time behaviour of their solution. 
Furthermore, we prove that the kernels characterising the renewal equations derived by biochemical system with reactions that satisfy the detail balance condition belong to the class of completely monotone functions.
			\end{abstract}
 		
	\tableofcontents
	
\section{Introduction}
A basic problem in biology is to determine the response of a system (that might be a cell, a cell organelle, a specific biochemical network or a tissue) to a chemical signal. The response typically might be a chemical, electrical or mechanical output. 
A formalism relating the input and the output using the so-called response-time distribution has been proposed in \cite{thurley2018modeling}. 
In particular, this formalism has been applied there to study signaling mechanisms between different immune cells. 
The goal of this paper is to formulate precise mathematical conditions which allow to use the formalism of \cite{thurley2018modeling} to model general biochemical systems. 

In the simplest model the relation between the input signal $ I(t) $ and the output $ R(t) $ of a system is given by
\begin{align}\label{A1}
	\begin{split}
		\frac{dN(t) }{dt}&= R(t)  \\
		R\left( t\right) &=R_{0}\left( t\right) +\int_{0}^{t}\psi\left( t-s, N(s) \right)
		I\left( s\right) ds. 
	\end{split}
\end{align}
In this formula $N(t)$ is the density of elements in the system, $R_{0}\left( t\right) $ is a transient response, associated
to the initial state of the system and the integral term describes the
response to the input $I\left( t\right) .$ 
The function $R_0$ is also called forcing function, see for instance \cite{diekmann1998formulation}. 
The input/output functions $I\left(
t\right) $ and $R\left( t\right) $ might be vectors, if the system
under consideration has several inputs and several outputs. Therefore, in general, $\psi
$ would be a matrix.

Notice that the formalism of response functions, using equations like \eqref{A1},  is particularly suited to study
biological systems, specifically biochemical systems. Indeed, due to the
large number of substances involved in these processes, it is often
difficult to determine all the reactions, as well as the relevant chemical coefficients, that
would be needed to model the system in detail. On the
other hand, response equations with the form (\ref{A1}) (or non-linear
versions of it) require only the knowledge of the function $\psi $, which, in principle, can be determined experimentally from
measurements of the behaviour of the system.

It is worth to mention that systems with the form \eqref{A1} have been extensively used in the modeling of biological systems. The earliest example appears in population dynamics, specifically in demography, see \cite{sharpe1911problem}. 
 Similar approaches to the one in  \cite{thurley2018modeling} can be found in models of immune systems, see for instance  \cite{busse2010competing}, in models that describe the production of $Ca^{2+}$, see \cite{moenke2012hierarchic,thurley2011derivation}, in models of kinetic proofreading \cite{bel2009simplicity} and in models of the circadian rhythms \cite{thurley2017principles}. 

In this paper, motivated by the work in \cite{thurley2018modeling}, we analyze under which conditions it is possible to study the interactions of different parts of a biochemical system by means of a set of response functions that generalizes (\ref{A1}). 
In the case of a linear system, our approach consists in considering a (large) subset of reactions as an unique object. In this paper, we call this object \textit{compartment}. Our interest is to understand the interactions between different compartments. In fact, these interactions are described in detail by response functions $\psi$, that can be derived from the reactions taking place inside each compartment. Once the response functions have been derived, one can study the resulting response function equation and ignore the detailed information on the processes taking place inside each compartment.

We anticipate that the evolution of the number of elements $N_\alpha $ in the compartment $\alpha $ will be described, in the linear case, by the following system of equations 
\begin{align}\label{eq:IntroConcentrationOneEntrance}
	\frac{d N_\alpha (t) }{dt } = B_\alpha(t) - D_\alpha(t), \quad N_\alpha (0)=N_\alpha^0
\end{align}
where $B_\alpha$ satisfies the renewal equation
\begin{equation} \label{eq:IntroRE one entrance} 
B_\alpha (t) = B_\alpha^0(t) +  \sum_{\beta \neq \alpha}  \int_0^t B_\beta (s ) \Phi_{ \beta \alpha }(t-s) ds 
\end{equation}
and where $D_\alpha  $ is a function of $B_\alpha $, namely 
\begin{equation} \label{eq:IntroConcentrationOneOutgoing}
    D_\alpha(t)=D^0_\alpha (t) + \int_0^t k_\alpha(t-s) B_\alpha(s). 
\end{equation}
In our study, the variable $ \alpha\in X $ is usually a compartment of a larger system $ \Omega $. More precisely, $ \Omega $ is the set of all possible states of elements in the system. Then, $ X $ is a partition of $ \Omega $, i.e. $ X $ is the set of all compartments $ \alpha $. We assume that $ \Omega $, and hence also $ X $, is finite. In particular, \eqref{eq:IntroConcentrationOneEntrance},  \eqref{eq:IntroRE one entrance}, \eqref{eq:IntroConcentrationOneOutgoing} is a finite system of equations.

We will refer from now on to the equations \eqref{eq:IntroConcentrationOneEntrance},  \eqref{eq:IntroRE one entrance}, \eqref{eq:IntroConcentrationOneOutgoing} as Response Function Equations (RFEs).
Notice that we are assuming that the chemicals in the system are in different states and the changes in the density of elements in a certain state is only due to jumps from one state to another. Hence, in this paper we restrict ourselves to conservative system, i.e. systems for which the total number of elements is constant in time. However, it would be possible to study also non-conservative systems, similar to the ones which appear naturally in population dynamics, see e.g. \cite{diekmann1998formulation}.

The functions $B_{\alpha },\ D_{\alpha }$ yield the total
fluxes of elements from any compartment $\beta \in X $
towards $ \alpha \in X, \ \alpha \neq \beta $ and from the compartment $\alpha $ towards any other compartment $\beta \in X,\ \beta \neq \alpha $ respectively. 
Notice that the set of
functions $\left\{ B_{\alpha }\right\},\ \left\{ D_{\alpha
}\right\} $ are related by means of input-response equations
of the form (\ref{A1}). More precisely, a flux $B_{\beta }(s)$ arriving to
the compartment $\beta \in X$ at time $s$ yields a response $B_{\beta }(s)\Phi_{\beta \alpha
}(t-s)$ at the compartment $\alpha $ at time $t.$ The formula yielding $%
B_{\alpha }(t)$ in \eqref{eq:IntroRE one entrance} takes into account the sum of the fluxes
arriving to the compartment $\beta $ for all times $s\in \left( 0,t\right) .$

The function $D^0_\alpha $ is the total flux of elements, that were already in the compartment $\alpha$ at time $0$,  to any compartment $\beta \in X $. 
Similarly $B_\alpha^0$ is the flux of elements arriving in $\alpha$ from any compartment $\beta$, given that they where in $\beta $ already at time equal zero. 

Since we assume that $\sum_{\alpha \in X }N_{\alpha }$ is conserved, we assume $\sum_{\alpha\in X}D_\alpha^0 = \sum_{\alpha\in X}B_\alpha^0 $ as well as 
\begin{equation} \label{kernel condition for conservation of numbers} 
	k_\alpha= \sum_{\beta \in X \setminus \{ \alpha \} } \Phi_{\alpha \beta }.
\end{equation}
Furthermore, for consistency, we need to assume
\begin{equation} \label{consistency ID}
    N_\alpha (0) \geq \int_0^\infty D^0_\alpha (t) dt
\end{equation}
for every $\alpha \in X $. This is a natural assumption that guarantees that the out-flux from the compartment $\alpha $ of the elements, which where in $\alpha $ already at time $t=0$, must be less than the total number of elements in $\alpha $ at time $0$. 

Finally, we assume
\begin{align}\label{eq:IntodKernelConservation}
    \sum_{\beta\in X\setminus\{\alpha\}} \int_0^\infty \Phi_{\alpha \beta }(t)\, dt=1,
\end{align}
if it is possible to go from compartment $\alpha$ to any other compartment $\beta$. Otherwise, if $\alpha $ is not connected to any $\beta \in X $ then 
\begin{equation} \label{eq:ApproxKernelsOneEntranceConservationzeroIntro}
 \sum_{\beta\in X\setminus \{\alpha\}} \int_0^\infty \Phi_{ \alpha\beta }(t)\, dt = 0. 
\end{equation}
Assumption \eqref{eq:IntodKernelConservation} guarantees that an element moves away from a certain state in finite time with probability one. 

Equations of the form \eqref{eq:IntroConcentrationOneEntrance},  \eqref{eq:IntroRE one entrance}, \eqref{eq:IntroConcentrationOneOutgoing} are one of the main objects of our study in this paper. Although, in this paper we are mainly concerned with linear models, we also introduce some non-linear variants.

Let us recall that equations of the form \eqref{eq:IntroRE one entrance} are commonly referred to as \textit{renewal equations}, (REs). Such equations, as well as some non-linear versions of them, have been extensively studied in the mathematical literature. For instance, they have been used to analyze structured populations, see e.g. \cite{diekmann1998formulation,feller1967introduction,franco2023modelling,franco2021one,Gripenberg1990VolterraIntEq}, or epidemiological models, see e.g. \cite{diekmann2013mathematical,inaba2017age,kermack1927contribution}. Similar equations, in which $\Phi_{\alpha \beta } $ contain in addition a dependence on the concentrations $\left\{ N_{\gamma }\right\} _{\gamma \in \Omega}$, have been considered in \cite{thurley2018modeling}. 

It is known that linear renewal equations modelling the evolution in time of a structured population can be reformulated as a partial differential equation with a transport term and with birth-death terms. 
See for instance \cite{calsina2016structured}. 
Similarly, this is possible for the RFEs \eqref{eq:IntroConcentrationOneEntrance},  \eqref{eq:IntroRE one entrance}, \eqref{eq:IntroConcentrationOneOutgoing}. The corresponding PDEs contain a transport term and birth-death terms. We refer to these types of models as structured population equations (SPEs). They have the form
\begin{align} \label{eq:IntrodModelPDE} 
	\begin{split}
	\partial _{t}f_{\alpha }(t,x)+\partial _{x}f_{\alpha }(t,x) & = -\Lambda_{\alpha }(x)f_{\alpha }(t,x),\ \ x\geq 0,\ \ t\geq 0,\ \ \alpha \in X
	\\
	f_{\alpha }\left( t,0\right)  &=\sum_{\beta \in X \setminus \{ \alpha \} } \int_{0}^{\infty }\lambda _{\beta \alpha}(x) f_{\beta }(t,x)dx
	\\
	f_{\alpha }(0,x) &=f_{\alpha ,0}\left( x\right)
	\end{split}
\end{align}
where $\Lambda _{\alpha }(x):=\sum_{X\setminus \{\alpha \} } \lambda_{\alpha \beta }(x) \geq 0.$

In the context of this paper, the family of solutions  $ \{ f_\alpha(t,x) \} $ of the equations \eqref{eq:IntrodModelPDE} is the density of elements with state in the compartment $\alpha \in X$ and with age $x$ at time $t$.
The particular type of SPEs considered here is conservative, i.e. $\partial
_{t}\left( \sum_{\alpha \in \Omega}\int_{0}^{\infty }f_{\alpha }(t,x)dx\right) =0
$. Notice that in these equations we assume that the elements of the population with trait in $ \alpha $ and age $x$ are removed with rate $\Lambda _{\alpha }(x)$ and they re-appear in the population as elements with trait in $\beta \neq \alpha $ and age $x=0$ at rate $\lambda_{\alpha\beta}$.

Usually in classical structured population models the parameter $x$ could be the age, the size or the immunity level against a certain pathogen, see \cite{diekmann1998formulation,franco2021one}. 
In this model $x$ is the time for which an element has been in a certain compartment. 
Introducing the age structure allows to reduce the non-Markovian system of RFEs, with forcing function $B_\alpha^0$ satisfying a suitable consistency condition (cf. \ref{PDE RE:forcing functions}), to a Markovian system of PDEs. Here with Markovian system of PDEs we mean that the transition from a certain state to another one does not depend on the history of the function $\{ f_\alpha\} $.

Let us remark that SPEs of the form \eqref{eq:IntrodModelPDE} have been extensively used in Mathematical Biology, in particular in population dynamics, see for instance \cite{perthame2006transport}. Solutions to SPEs describe the evolution of individuals (e.g. cells, humans, animals, etc.) in a population structured via a certain variable (e.g. age, size, immunity against a certain pathogen). 
In this paper, we study the equivalence of a generalization of the SPE \eqref{eq:IntrodModelPDE} with a generalization of the RFE \eqref{eq:IntroConcentrationOneEntrance}, \eqref{eq:IntroRE one entrance}, \eqref{eq:IntroConcentrationOneOutgoing}. 

Chemical systems in Systems Biology are often formulated as ODEs. 
Here, we denote by $\Omega $ the finite set of possible states of the elements in the system. The elements are for instance cells in different states or, alternatively, chemical substances which can be transformed in another one by means of chemical reactions.  

The system is described by a finite number of
concentrations $n(t)=(n_{i}(t))_{i\in \Omega }\in \mathbb{R}_{+}^{|\Omega |}$. Here, $|\Omega |\in 
\mathbb{N}$ denotes the total number of states. The
concentrations evolve according to ODEs of the form 
\begin{equation}
	\dfrac{dn}{dt}(t)=An(t),\ n\left( 0\right) =n_{0},  \label{eq:IntrodModelODE}
\end{equation}
where the matrix $A\in \mathbb{R}^{|\Omega |\times |\Omega |}$ satisfies 
\begin{equation*}
	A_{ii}=-\sum_{k\in \Omega \backslash \{i\}}A_{ki}\ \ \text{for all }i\in \Omega \ \ ,\ \ A_{ki}\geq 0\ \text{for }k\neq i.
\end{equation*}
As we will see in this paper, one can reduce a system of ODEs \eqref{eq:IntrodModelODE} to the RFEs \eqref{eq:IntroConcentrationOneEntrance},  \eqref{eq:IntroRE one entrance}, \eqref{eq:IntroConcentrationOneOutgoing}. To this end, we partition the set of states $\Omega $ into compartments, i.e. in a
family of disjoint sets $ \alpha\in X$. Then, our study concerns merely the concentrations $ N_{\alpha }= \sum_{k\in\alpha }n_{k}$ within each compartment $ \alpha\in X $ as well as the inward and outward fluxes. Let us mention that for general choices of matrices $ A $ these fluxes solve more general RFEs than those in \eqref{eq:IntroConcentrationOneEntrance},  \eqref{eq:IntroRE one entrance}, \eqref{eq:IntroConcentrationOneOutgoing}. 

It is important to remark that, while the evolution of the solution $n\in \mathbb{R}_{+}^{|\Omega |}$ of \eqref{eq:IntrodModelODE} is Markovian,  the evolution of $ \{N_{\alpha }\}_{\alpha \in X}$ is non-Markovian. Therefore, we say that reformulating the ODEs \eqref{eq:IntrodModelODE} as a RFE is a \textit{demarkovianization} process. 
Notice that the ODEs \eqref{eq:IntrodModelODE} are in general not equivalent to the RFEs \eqref{eq:IntroConcentrationOneEntrance},  \eqref{eq:IntroRE one entrance}, \eqref{eq:IntroConcentrationOneOutgoing}, unless some information on the internal states of each compartment is available.
However, the information on the evolution of internal states before time $ 0 $ is contained in the functions $\{D_\alpha^0\} $ and $\{ B_\alpha^0\}$. Similarly, the response functions $ \Phi_{ \alpha \beta }(t) $ are given by the evolution of internal states at time $ t $.

Once a RFE has been derived from a ODEs system, we can think of each compartment as a \textit{black box}. The RFE model describes interactions between these black boxes. On the other hand, the concentrations $n_{k}$ can be interpreted as a set of internal variables, which characterize each compartment completely. The demarkovianization procedure yields a system with a smaller number of variables. Thus, one replaces a large Markovian system by a smaller non-Markovian system.

This procedure is reminiscent of the construction of so-called hidden Markov models (cf. \cite{baum1966statistical}). The goal in that case is to study the evolution of a Markov-process X, with unobservable (hidden) variables, by analysing only the evolution of the observable variables.
We refer to \cite{bishop1986maximum} where hidden Markov processes have been applied to analyse DNA sequences for the first time. 

It is relevant to mention that the demarkovianization process in this paper is different from a procedure called lumpling, which has been extensively studied in chemical engineering, see \cite{atay2017lumpability}. In lumpling a large system of ODEs is reduced to a smaller system of ODEs, i.e. a Markovian process. This is only possible if the initial system of ODEs has a particular structure. In the procedure studied here the systems obtained are in general non-Markovian.

Furthermore, let us mention that ODEs of the from \eqref{eq:IntrodModelODE} describe pure jump Markov processes. Similarly, RFEs of the form \eqref{eq:IntroRE one entrance}  describe the so-called semi-Markov processes, see \cite{gyllenberg2008quasi}. These semi-Markov processes are, as well, pure jump processes. However, while in the case of Markov processes jump times are always exponentially distributed, semi-Markov processes allow for more general distributions of the jump times. In fact, the response functions $ \Phi_{ \alpha \beta } $ are exactly the probability densities of these jump time distributions.

Let us now give an overview of the problems studied in this paper using the response function formalism.
First of all, we study the relation between the ODEs and the RFEs in the linear case. We formulate in a precise
manner the response function $\Phi_{\alpha \beta} $ corresponding to the decomposition in compartments of the systems of  ODEs \eqref{eq:IntrodModelODE}. 
In addition, we prove that for any set of response functions $\Phi_{\alpha \beta }$ satisfying \eqref{eq:IntodKernelConservation} it is possible to find a sequence of ODEs with the form \eqref{eq:IntrodModelODE}, such that the corresponding response functions converge to $\Phi_{\alpha \beta }$. In other words, we prove that the set of response functions corresponding to the class of ODEs \eqref{eq:IntrodModelODE} is dense in the set of probability measures (endowed with the weak topology).
We refer to Section \ref{sec:ApproximationIntegralKernel} for more details. 

In this paper we consider also systems of the form \eqref{eq:IntrodModelODE} satisfying the detailed balance condition, i.e. we assume that at the equilibrium each individual reaction is balanced (cf. \cite{liggett1985interacting}). 
It turns out that the class of response functions that can be derived from these systems is much smaller than the one obtained for the general systems with the form \eqref{eq:IntrodModelODE}. 
More precisely, the response functions $\Phi_{\alpha \beta }$ obtained from a general class of ODEs satisfying the detailed balance condition belong to the family of completely monotone functions, i.e. they are Laplace transforms of non-negative measures 
(see Theorem \ref{thm:DetailedBalanceKernel}). 
We remark that, to have completely monotone response functions is a necessary condition for systems originating from ODEs with detailed balance, but it is not sufficient. There are also systems of ODEs, that do not satisfy the detailed balance condition, for which the response function is not completely monotone. Theorem \ref{thm:DetailedBalanceKernel} provides a way to discriminate these systems from the ones satisfying the detailed balance condition. 

As far as we know, this is the first property that has been derived for response functions that at the microscopic level are described by reactions satisfying the detailed balance property.
It is relevant to mention that the characterization of biochemical systems for which the detailed balance condition holds or fails is an active research area (see for instance \cite{battle2016broken, li2019quantifying, martinez2019inferring}). 
For instance in \cite{martinez2019inferring} a measure of the degree of irreversibility of general semi-Markov processes has been obtained. However, notice that the question of determining if the semi-Markov process under consideration has been obtained from a chemical system satisfying the detailed balance condition, is not addressed in \cite{martinez2019inferring}.

Recall that the demarkovianization procedure yields a system of RFEs, which are in general non-Markovian. It is then relevant to classify the response functions for which the corresponding RFEs are actually Markovian. In fact, we prove, for a specific class of forcing function, that such response functions are exactly given by exponentials. A similar result is known for REs arising in population dynamics and in epidemiology, see \cite{diekmann2020finite,diekmann2023systematic}. Furthermore, in terms of semi-Markov processes, this result is related to the fact that the only jump time distribution which yields a Markov processes is the exponential distribution.

Next, we study the long-time behaviour of solutions to \eqref{eq:IntroConcentrationOneEntrance},  \eqref{eq:IntroRE one entrance}, \eqref{eq:IntroConcentrationOneOutgoing}. Specifically, we give conditions on the response functions that guarantee that the solution $ \{N_\alpha\} $ to \eqref{eq:IntroConcentrationOneEntrance}, \eqref{eq:IntroRE one entrance},\eqref{eq:IntroConcentrationOneOutgoing} converge to a unique stable distribution. For this we rely on Laplace transform methods similar to the ones in \cite{diekmann2012delay}.

Another issue that we discuss in this paper is the reformulation of RFEs of the form \eqref{eq:IntroConcentrationOneEntrance},  \eqref{eq:IntroRE one entrance}, \eqref{eq:IntroConcentrationOneOutgoing} in terms of SPEs of the form \eqref{eq:IntrodModelPDE}.  
This reformulation relies on the introduction of a canonical age structure. We stress that the age structure will not be an intrinsic property of the elements of the system, but a set of auxiliary variables that allows to \textit{Markovianize} the evolution of the densities of elements in the compartments. 

In order to illustrate the use of the formalism of RFEs developed in this paper, we provide several examples of linear models, which are well-established in Biochemistry and in Systems Biology. In particular, we study the evolution in time of two variations of  the classical Hopfield model \cite{Hopfield} of kinetic proofreading using the machinery of RFEs. Furthermore, we exhibit a model of polymerization combined with proofreading (see \cite{PigolottiSartori}) that results in a non-Markovian polymerization rate. Then, we also formulate in terms of RFEs a linear version of the classical Barkai-Leibler model for robust adaptation \cite{barkai1997robustness}. 

Finally, we consider examples of non-linear models. In this situation the class of RFEs must be much more general than the one in \eqref{A1}.
Specifically, it would be relevant to determine under which conditions a class of RFEs is equivalent to non-linear chemical reactions with non-linearities of the kind of the mass action or more complicated ones like Michaelis-Menten or Hill's law.
In particular, the response function contains in general, multiple integrals of the form 
\begin{align} \label{nonlinearityintroduction}
	\int_{-\infty }^t \dots \int_{-\infty}^t I(s_1) \dots   I(s_n)  \psi(t-s_{1}, \dots t-s_{n} )  ds_{1} \dots  ds_n.
\end{align}
Notice that operators with the form \eqref{nonlinearityintroduction} can provide information on the time correlations of the signal $I(s)$. In contrast, such correlations can not be described by the integral operators in \eqref{A1}. This is not surprising since linear ODE systems cannot yield information about time correlations of the incoming signal, in contrast with non-linear systems.

The non-linear examples discussed here include a model of non-Markovian polymerization as well as a system of ODEs describing the standard Coherent Type $1$ Feed Forward Loop. The latter has been extensively considered in Systems Biology (cf.  \cite{alon2019introduction}).

Some of the results of this paper have been derived in different contexts in the literature, as for instance the characterization of response functions associated Markovian dynamics or the density of the response functions associated to ODEs systems in the space of probability measures.
However, in order to unify the notation and to present the results in the context of chemical reactions considered in this paper, we decided to include the proofs of these results here. 

\subsection{Notation and plan of the paper}
\label{sec:notation}
Before explaining the plan of this paper let us collect here some notation used later on. 
First of all we define $\mathbb R_+:= [0, \infty)$ and $\mathbb R_*:= (0, \infty)$. We denote by $C_c(\mathbb R_+)$ the space of continuous functions with compact support endowed with the supremum norm. 
Moreover we write $\mathcal M (\mathbb R_+) $ for
the space of Radon measures on $\mathbb R_+$. We denote with $\mathcal M_{+} (\mathbb R_+)$ the cone of non-negative Radon measures and with $\mathcal M_{ b} (\mathbb R_+)$ the space of bounded Radon measures. Furthermore, we write 
$\mathcal M_{ +, b} (\mathbb R_+)$ to indicate the cone of non-negative bounded Radon measures. Let us recall that endowing $\mathcal M_{b}(\mathbb R_+)$ with the total variation norm $\|\cdot \|_{TV}$ yields a Banach space which can be identified with $C_0(\mathbb R_+)^*$. Here, $ C_0(\mathbb R_+) $ denotes the space of continuous functions vanishing at infinity endowed with the supremum norm. 

Furthermore, let us recall that the weak convergence in the sense of measures is defined by duality with bounded continuous functions $ C_b(\R_+) $, i.e. $ \mu_n\rightharpoonup \mu $ if and only if
\begin{align*}
	\int_{\R_+}\phi(x)\, \mu_n(dx) \to \int_{\R_+}\phi(x)\, \mu(dx)
\end{align*}
as $ n\to \infty $ for any $ \phi\in C_b(\R_+)  $. Let us denote by $d_w $ a metric inducing weak convergence, e.g. the L\'{e}vy-Prokhorov metric.

The space $C([0, T] ; \mathcal M_{+, b}(\mathbb R_+) ) $ contains all continuous functions from $[0, T] $ to $\mathcal M_{+,b}(\mathbb R_+)$. Here, we endow the space $\mathcal M_{+, b} (\mathbb R_+)$ with the Wasserstein distance $W_1 $ defined by 
\[ 
W_1(\mu, \nu ):= \sup_{\{ \| \varphi \|_{Lip} \leq 1 \}} \int_{\mathbb R_+} \varphi(x) (\mu - \nu )(dx) 
\]
where the supremum is taken over the Lipschitz functions and where
\begin{align*}
    \| \varphi \|_{Lip }= \| \varphi \|_\infty + \sup_{ \{  x,y \in \mathbb R_+,  x \neq y \} }  \frac{|\varphi(x) - \varphi (y) |}{|x-y |}.
\end{align*}
Furthermore, we denote with $L^1_{loc}(\mathbb R , \mathbb R^{n \times m })$ the space of measurable functions from $\mathbb R$ to $\mathbb R^{n \times m }$ that are locally integrable and with $W^{1,1}_{loc}(\mathbb R_+, \mathbb R) $ the space of locally absolutely continuous functions from $\mathbb R_+$ to $\mathbb R$. Let us also define by $L^1(\mathbb R_+, e^{- z_0 t } dt )$ the space of measurable functions $f$ from $\mathbb R_+$ to $\mathbb R$ such that $t \mapsto f(t) e^{ z_0 t } $ is integrable. 

In addition, we write $\hat{f} (z) $ for the Laplace transform of $f$, i.e.
\[
\hat{f} (\lambda):= \int_0^\infty e^{- \lambda t } f(t) dt. 
\]

Finally, for $A\in \R^{n\times m}$ we write $\|A \|$ to indicate the matrix norm induced by the standard euclidean norm on $\R^{n}$ and $\R^{m}$. As no ambiguity should arise we abuse notation and write $\|\cdot \|$ for the standard euclidean norm for vectors.
\bigskip 

The paper is organized as follows.  
In Section \ref{sec:ReformODEasRFE} we discuss the reformulation of a linear systems of ODEs \eqref{eq:IntrodModelODE} in a system of RFEs.  In Section \ref{sec:ApproximationIntegralKernel} we characterize the response functions corresponding to general linear systems of ODEs as well as systems with detailed balance. Section \ref{subsec:CharMarkov} is devoted to the characterization of response functions $\Phi_{\alpha \beta}$ in \eqref{eq:IntroConcentrationOneEntrance},  \eqref{eq:IntroRE one entrance}, \eqref{eq:IntroConcentrationOneOutgoing} corresponding to a Markovian evolution. In particular, we prove that such response functions are exactly exponentials. Then, we analyze the long-time behaviour of solutions to \eqref{eq:IntroConcentrationOneEntrance}, \eqref{eq:IntroRE one entrance},\eqref{eq:IntroConcentrationOneOutgoing} in Section \ref{sec:long-time}. Furthermore, the equivalence of RFEs and SPEs is discussed in Section \ref{sec:PDE}. In Section \ref{sec:Examples} and \ref{sec:nonlinear} we present both linear and non-linear examples of biochemical systems whose response can be described using RFEs. Finally, we include some concluding remarks in Section \ref{sec:conclusion}.

\section{Reformulation of linear ODEs in terms of RFEs}\label{sec:ReformODEasRFE}

  In this section we study how to rewrite systems of ODEs of the form \eqref{eq:IntrodModelODE} using the RFEs formalism. 
  More precisely, we will show that the concentrations of elements with different states in the compartment $\alpha $, $\{ n_\alpha \} $ with $n_\alpha \in \mathbb R_+^{|\alpha |}$ satisfy a system of ODEs (cf. \eqref{eq:PrelDecompODE}). 
  Then we define the concentration $ N_\alpha$ of individuals with states in compartment $ \alpha\in X $ by
\begin{align} \label{eq:DecompDefNumberComp} 
	N_\alpha(t):=\sum_{j \in \alpha} (n_\alpha (t) )_j= \textbf{e}^\top_{|\alpha|} n_\alpha(t).
\end{align} 
Here, we use the notation

 \begin{equation} \label{e bar} 
 \textbf{e}_n =(1,\ldots,1)^\top\in\R^{n}.
 \end{equation} 
 for $ \textbf{e}_\alpha $. 
 As we will see, the evolution of $ \{N_\alpha\} $ is given by a generalization of the RFEs described in the introduction, cf. \eqref{eq:IntroConcentrationOneEntrance}-\eqref{eq:IntroRE one entrance}-\eqref{eq:IntroConcentrationOneOutgoing}, which are valid only when the compartments have at most one entrance point.
In Section \ref{sec:DecompGeneralSyst} we introduce the decomposition of the system \eqref{eq:IntrodModelODE} into compartments. 
In Section \ref{subsec:generlized RFE} we formulate and study the generalized RFEs.
In Section \ref{sec:ODE and RFE}, instead, we explain the relation  between the system of ODEs \eqref{eq:PrelDecompODE} and the RFEs.

\subsection{Decomposition into compartments}
\label{sec:DecompGeneralSyst}
In this section we introduce the decomposition of the system \eqref{eq:IntrodModelODE} into compartments. We also give a definition of a special class of decompositions into compartments that have at most one entrance point. 

Recall that the set of the states, $\Omega$, is finite. The dynamics of concentrations $ n=(n_i)_{i\in \Omega}\in \R_+^{|\Omega|} $ are given by
\begin{align}\label{eq:DecompModelODE}
	\dfrac{dn(t) }{dt}= An(t), \quad t >0.
\end{align}
We assume the matrix $ A\in \R^{|\Omega|\times |\Omega|} $ to be of the form
\begin{align*}
		A_{ij} = \lambda_{ji}\geq0 \text{ for } i,\, j\in \Omega, \, i\neq j, \quad A_{ii} = - \sum_{k\in \Omega\backslash\{i\}} \lambda_{ik}.
\end{align*}
Hence, $\textbf{e}_{|\Omega |}^\top A =0$ where we are using the notation \eqref{e bar}. 

Let us note that the system of ODEs \eqref{eq:DecompModelODE} is related to a pure Markov jump process with jump rates $ \lambda_{ij}\geq0 $. Furthermore, these rates induce a graph structure on the state space $\Omega$. More precisely, we can define the set of the (directed) edges $ \mathcal{E}\subset \Omega^2 $ by $ \mathcal{E}=\setdef{(i,j)\in\Omega^2 }{\lambda_{ij}>0} $.
Then $(\Omega , \mathcal E)$ is a directed graph. See Figure \ref{fig:compartments} for a graphical representation of the division into compartments. 
	\begin{figure}[ht]
		\centering
		\includegraphics[width=0.4\linewidth]{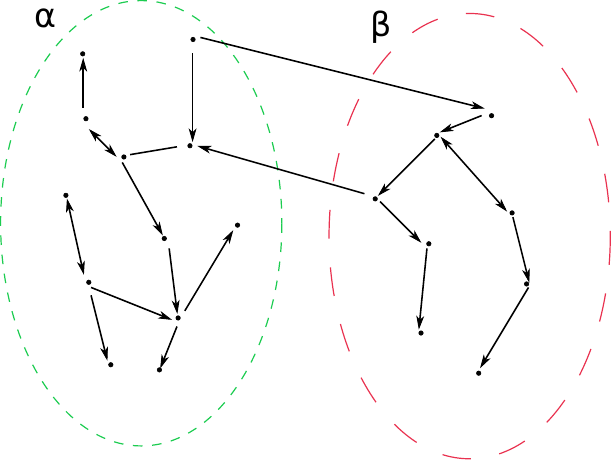}
		\caption{Partition of a graph in two compartments $\alpha $ and $\beta $. The straight lines connecting the elements are Markovian reactions. }
		\label{fig:compartments}
	\end{figure}
 
 For our reduction in compartments we choose a partition $ X\subset 2^\Omega $. Accordingly, the solution of \eqref{eq:DecompModelODE} can be decomposed in $ n(t)=(n_\alpha(t))_{\alpha\in X } $ where  $ n_\alpha(t)\in \R_+^{|\alpha|} $ is the evolution in time of the density of elements in the compartment $\alpha \in X$. In order to specify the equations solved by $ \{n_\alpha \}_{ \alpha \in X} $, we decompose the matrix $A$ as 
\[ A = \left(\begin{matrix}
			A_{\alpha\beta}
		\end{matrix}\right)_{\alpha,\beta\in X}.
  \] 
	Here, the matrices $ A_{\alpha\beta}\in \R^{|\alpha|  \times  |\beta|  } $ are defined by
	$A_{\alpha\beta} = (\lambda_{ji})_{i\in\alpha,j\in\beta}$
	for every $\alpha , \beta \in X$ such that $\alpha\neq \beta$. 
 Instead, when $\alpha=\beta$ we define
	\begin{align*}
		A_{\alpha \alpha } = E_{\alpha \alpha}  - C_{\alpha}. 
	\end{align*}
	where the matrices $E_{\alpha \alpha} \in \mathbb R_+^{|\alpha | \times |\alpha |}$ are given by 
	$
		(E_{\alpha \alpha})_{ij} = A_{ij}=\lambda_{ji}
	$
	when $i,j\in \alpha$, $i\neq j$. While for $ i=j $ we have
	$
		(E_{\alpha \alpha})_{ii} = -\sum_{k\in \alpha}\lambda_{ik}.
	$
	The matrix $C_{\alpha}  \in \mathbb R_+^{|\alpha | \times |\alpha |} $ contains the loss terms due to the jumps from $\alpha $ to other compartments.
 In particular, $C_\alpha$ is a diagonal matrix of the form
	$
		(C_{\alpha})_{ii} = \sum_{j\in \Omega \setminus \{\alpha\}} \lambda_{ij} \text{ where } i \in \alpha.$
  
	As a consequence of \eqref{eq:DecompModelODE} we have for all $n_\alpha $
	\begin{align}\label{eq:PrelDecompODE}
		\frac{d n_\alpha}{dt}(t)= E_{\alpha \alpha}n_{\alpha } (t) - C_\alpha n_\alpha (t)+ \sum_{\beta \in X \setminus \{\alpha\}} A_{\alpha \beta } n_\beta (t).
	\end{align}
	Accordingly, the initial condition is $n_\alpha(0)=n_\alpha^0$ for a given $n_\alpha^0  \in \mathbb R_+^{|\alpha |}$. Finally, the graph structure $(\Omega, \mathcal{E}) $ suggests the following definition.
\begin{definition}[Entrance point]\label{def:EntrancePointGraph}
	Consider a directed graph $ (\Omega,\mathcal{E}) $ and a partition $ X\subset 2^{\Omega} $. Let $ \alpha\in X $. A state $ i\in\alpha $ is an entrance point if there exists $ \beta \in X\setminus\{\alpha\} $ and $ j\in \beta $ such that $ (j,i)\in \mathcal{E} $.
\end{definition}
As mentioned in the introduction we sometimes restrict our attention to decompositions into compartments that have at most one entrance point. This single entrance point (if it exists) is then denoted by $ i_\alpha\in \alpha $.

\subsection{Generalized RFEs} \label{subsec:generlized RFE} 
The generalization, mentioned in the introduction, of the RFE that we are interested in  is given by the following system of equations for $ \alpha\in X $
\begin{align}
	\frac{d N_\alpha }{dt }(t) &= \textbf{e}_\alpha^\top S_\alpha(t) - \textbf{e}_\alpha^\top  J_\alpha(t), \quad N_\alpha(0) = N_\alpha^0\geq0, \label{eq:Concentration}
	\\
	S_\alpha(t) &=  S_\alpha^0(t) + \sum_{\beta \in X \setminus \{\alpha\} } \int_0^t  G_{\beta \alpha} (t-s) S_\beta (s)  ds, \label{eq:REInflux}
	\\
	J_\alpha(t) &= J^0_\alpha(t) + \int_0^t K_{\alpha } (t-s) S_\alpha(s) ds. \label{eq:REOutflux}
\end{align}
The fluxes $ \{S_\alpha(t) \} , \{\, J_\alpha(t)\} $ with $S_\alpha, J_\alpha \in \R_+^{|\alpha|} $ as well as the concentrations $\{ N_\alpha \} $ with $N_\alpha \in \R_+^{|\alpha|} $ are the unknowns.
Instead, the kernels $G_{\alpha\beta}(t)  \in \mathbb R_+^{|\beta | \times |\alpha |} $, $K_\alpha (t) \in  \mathbb R_+^{|\alpha | \times |\alpha |} $ and the forcing functions $S_\alpha^0(t) \in \mathbb R_+^{|\alpha | } $, $J^0_\alpha(t) \in \mathbb R_+^{|\alpha | } $ are given data.

Let us mention that \eqref{eq:Concentration}, \eqref{eq:REInflux}, \eqref{eq:REOutflux} is a closed system of equations, so once $ \{S_\alpha\} $ is known, we can deduce $ \{J_\alpha\} $ and $\{ N_\alpha \}$. 
While the influxes and the outfluxes $\{ B_\alpha \}$ and $\{ D_\alpha \} $ in \eqref{eq:IntroConcentrationOneEntrance}-\eqref{eq:IntroRE one entrance}-\eqref{eq:IntroConcentrationOneOutgoing} are real valued functions of time, the fluxes $\{S_\alpha\} $ and $\{ J_\alpha \} $ are now vector-valued. Here, the $i$-th component of $S_\alpha $ can be interpreted as the flux of elements, coming from any other compartment, to the state $ i \in \alpha$. Analogously, the $i$-th component of the vector $J_\alpha $ is the out-flux from the state $i \in \alpha $ to any state $j$ in some compartment $\beta \neq \alpha $. 

Concerning the well-posedness of the system \eqref{eq:Concentration}-\eqref{eq:REInflux}-\eqref{eq:REOutflux} we rely on the following result.

\begin{lemma}\label{lem:Wellposedness RFE}
	Assume that for all $ \alpha, \beta \in X $
	\begin{align}\label{eq:RegularityCondGenRFE}
		G_{\alpha\beta} \in L^1_{\operatorname{loc}}(\R_+; \R_+^{|\beta|\times |\alpha|}), \quad K_\alpha\in L^1_{\operatorname{loc}}(\R_+; \R_+^{|\alpha|\times |\alpha|}), \quad 
		S^0_\alpha,\,  J^0_\alpha \in L^1_{\operatorname{loc}}(\R_+;\R_+^{|\alpha|}). 
	\end{align}
	Then, the system \eqref{eq:Concentration}-\eqref{eq:REInflux}-\eqref{eq:REOutflux} has a unique solution with $ S_\alpha, \, J_\alpha \in L^1_{\operatorname{loc}}(\R_+;\R_+^{|\alpha|}) $ and $ N_\alpha\in W^{1,1}_{\operatorname{loc}}(\R_+;\R) $ for all $ \alpha\in X$.
\end{lemma}
\begin{proof}
We have to prove existence and uniqueness of a solution of the renewal equation \eqref{eq:REInflux} for every $\alpha \in X$. 
	Existence and uniqueness in $ L^1_{\operatorname{loc}} $ of a solution for \eqref{eq:REInflux} follows from the existence of a unique resolvent $ R \in L^1_{\operatorname{loc}} $, see e.g. \cite[Section 2.3]{Gripenberg1990VolterraIntEq} for the definition of resolvent and for the proof of its existence and uniqueness. 
  The resolvent is given by an infinite series of convolutions involving the kernels, hence it is positive when the kernels are positive. This, together with the fact that the forcing functions have non-negative entries, implies that also the solution $ \{S_\alpha\}, \, \{J_\alpha\} $ have non-negative entries. 
\end{proof}
The solutions to \eqref{eq:REInflux}-\eqref{eq:REOutflux} do not yield, in general, non-negative concentrations $ \{N_\alpha\} $. However, as we see in the following lemma, this holds under additional assumptions on the kernels and the forcing functions.

\begin{lemma}\label{lem:PositivityGenRFE}
	Under the assumptions in Lemma \ref{lem:Wellposedness RFE} consider the unique solution of \eqref{eq:Concentration}-\eqref{eq:REInflux}-\eqref{eq:REOutflux}. Assume in addition that for all $ \alpha \in X $, $ j\in \alpha $
	\begin{align}\label{eq:LemmaPositivityAssumption}
		\int_0^\infty \textbf{e}_\alpha^\top J_\alpha^0(r) \, dr \leq N_\alpha^0, \quad \int_0^\infty 	[\textbf{e}_\alpha^\top K_\alpha(r)]_j\, dr \leq 1.
	\end{align}
	Then, we have $ N_\alpha(t)\geq 0 $ for all $ t\geq0 $, $ \alpha\in X $.
\end{lemma}
\begin{remark}
	Let us mention that condition \eqref{eq:LemmaPositivityAssumption} appears very naturally. The first inequality in \eqref{eq:LemmaPositivityAssumption} ensures that the total number of elements, that where in $\alpha$ already at time zero, removed from compartment $ \alpha $ does not exceed $ N^0_\alpha $. The second inequality can be interpreted by viewing \eqref{eq:Concentration}-\eqref{eq:REInflux}-\eqref{eq:REOutflux} as a semi-Markov process. The integral kernels are related to the jump probabilities. Hence, \eqref{eq:LemmaPositivityAssumption} ensures that the probability to leave $ \alpha $ via state $ j\in \alpha $ is bounded by one.
\end{remark}
\begin{proof}[Proof of Lemma \ref{lem:PositivityGenRFE}]
	According to \eqref{eq:Concentration}, we write
	\begin{align*}
		N_\alpha(t) &= N_\alpha(0) + \int_0^t \left( \textbf{e}_\alpha^\top S_\alpha(s) - \textbf{e}_\alpha^\top  J_\alpha(s) \right) \, ds 
		\\
		&=  N_\alpha(0) + \int_0^t \left( \textbf{e}_\alpha^\top S_\alpha(s) - \textbf{e}_\alpha^\top  J_\alpha^0(s) - \int_0^s \textbf{e}_\alpha^\top K_{\alpha } (s-r) S_\alpha(r)\, dr \right) \, ds.
	\end{align*}
	Here, we used \eqref{eq:REOutflux}. We then obtain from \eqref{eq:LemmaPositivityAssumption}
	\begin{align*}
		N_\alpha(t) &\geq \int_0^t \left( \textbf{e}_\alpha^\top S_\alpha(s) - \int_0^s \textbf{e}_\alpha^\top K_{\alpha } (s-r) S_\alpha(r)\, dr \right) \, ds
		\\
		&=\int_0^t \textbf{e}_\alpha^\top\left( I_{|\alpha|\times |\alpha|} -  \int_s^t K_\alpha(r-s)\, dr \right)S_\alpha(s) \, ds \geq 0,
	\end{align*}
	since $ (S_\alpha)_j\geq0 $ for all $ \alpha\in X $, $ j\in \alpha $. This concludes the proof.
\end{proof}
We now provide some conditions ensuring the conservation of mass.
\begin{lemma}\label{lem:ConservationMassGenRFE}
	Under the assumption in Lemma \ref{lem:Wellposedness RFE} consider the unique solution to \eqref{eq:Concentration}-\eqref{eq:REInflux}-\eqref{eq:REOutflux}. Assume that
	\begin{align}\label{eq:LemmaConservMassCondition}
		\sum_{\alpha \in X} \textbf{e}_\alpha^\top\left( S_\alpha^0(t)-J^0_\alpha(t) \right) = 0,
		\quad
		\sum_{\beta\in X\setminus\{\alpha\}} \textbf{e}_\beta^\top G_{\alpha\beta}(t)-\textbf{e}_\alpha^\top K_\alpha(t) =0
	\end{align}
	for all $\alpha\in X$. Then, the total mass is conserved $ \sum_{\alpha\in X}N_\alpha(t)=\sum_{\alpha\in X}N_\alpha^0 $.
\end{lemma}
\begin{proof}
	We obtain from \eqref{eq:Concentration}-\eqref{eq:REInflux}-\eqref{eq:REOutflux}
	\begin{align*}
		\dfrac{d}{dt}\sum_{\alpha \in X} N_\alpha(t) &=  \sum_{\alpha \in X} \textbf{e}_\alpha^\top\left( S_\alpha^0(t)-J^0_\alpha(t) \right) 
		\\
		&\quad + \sum_{\alpha,\, \beta\in X, \, \alpha\neq \beta} \int_0^t \textbf{e}_\alpha^\top G_{\beta\alpha}(t-s) S_\beta(s)\, ds - \sum_{\alpha \in X} \int_0^t \textbf{e}_\alpha^\top K_\alpha(t-s)S_\alpha(s)\, ds
		\\
		&=\sum_{\alpha\in X}\int_0^t S_\alpha(s)\left\lbrace \sum_{\beta\in X\setminus\{\alpha\}} \textbf{e}_\beta^\top G_{\alpha\beta}(t-s) - \textbf{e}_\alpha^\top K_\alpha(t-s) \right\rbrace \, ds =0.
	\end{align*}
\end{proof}

\subsection{Generalized RFEs corresponding to the ODEs model} \label{sec:ODE and RFE}
The main result of this section yields that generalized RFE, cf. \eqref{eq:Concentration}-\eqref{eq:REInflux}-\eqref{eq:REOutflux}, appear as effective equations for $\{ n_\alpha \} $ solving \eqref{eq:PrelDecompODE}.

\begin{theorem}[ODE to RFE]\label{thm:ReductionODEtoRE}
	Consider $n(t) \in \mathbb R_+^{|\Omega|}$ solution to \eqref{eq:DecompModelODE} and $N_\alpha $ defined in \eqref{eq:DecompDefNumberComp}. Then, $\{N_\alpha\} $ satisfies \eqref{eq:Concentration}. The corresponding fluxes $S_\alpha $ and $J_\alpha $ solve the system \eqref{eq:REInflux}-\eqref{eq:REOutflux} with kernels
	\begin{align}
		G_{\beta\alpha} (t) &= A_{\alpha\beta  } e^{t A_{\beta\beta}  }, \quad t \geq 0, \label{eq:ThmIntegralKernelInflux}
		\\
		\left(K_\alpha(t)\right)_{ij} &= \sum_{\beta \in X \setminus \{\alpha\}} \left( \textbf{e}_{\beta}^\top A_{\beta\alpha} \right)_i\left( e^{t A_{\alpha \alpha }}  \right)_{ij}, \quad i,\, j\in \alpha,\,  t \geq 0, \label{eq:ThmIntegralKernelOutflux}
	\end{align}
	and forcing functions
	\begin{align}\label{eq:ThmInOutFlux}
		S_\alpha^0(t) = \sum_{\beta \in X \setminus \{\alpha\}} G_{\beta \alpha} (t) n^0_\beta; \quad 
		J^0_\alpha(t)= K_\alpha (t) n_\alpha^0, \quad t \geq 0.
	\end{align}
	Furthermore, the fluxes are given in terms of $ n(t) $ by
	\begin{align}\label{eq:ReductionODEtoREFluxes}
		S_\alpha(t) = \sum_{\beta \in X  \setminus \{\alpha\} }A_{\alpha \beta}n_\beta(t), \quad (J_\alpha(t))_i = (n_\alpha(t))_i\sum_{\beta \in X  \setminus \{\alpha\} } \left( \textbf{e}_{\beta}^\top A_{\beta\alpha}\right)_i, \quad i\in \alpha
	\end{align}
	for all $ \alpha\in X $ .
\end{theorem}
In order to prove Theorem \ref{thm:ReductionODEtoRE} we first show that the ODE \eqref{eq:PrelDecompODE} can be reformulated using the response function formalism by a system involving fluxes of the form $I^{i j }_{\alpha \beta }$. This term is the flux from the state $i\in \alpha$ to the state $j\in \beta$. In the proof of Theorem \ref{thm:ReductionODEtoRE} below we will obtain \eqref{eq:Concentration}-\eqref{eq:REInflux}-\eqref{eq:REOutflux} from the system solved by $ \{N_\alpha\} $ and the fluxes $\{I^{i j}_{\alpha \beta}\}$.

Let us first give an informal derivation of the equations solved by $ \{N_\alpha\} $ and the fluxes $\{I^{i j}_{\alpha \beta}\}$. First of all, the change of the concentration of the compartments is just due to the individual fluxes summed appropriately, thus yielding
\begin{equation} \label{eq:N_alpha first level}
	\frac{d N_\alpha (t) }{dt } = \sum_{i \in \alpha} \sum_{\beta \in X \setminus \{\alpha\} }  \sum_{j \in \beta}  I_{\beta \alpha }^{ji} (t) - \sum_{i \in \alpha}\sum_{\beta \in X  \setminus \{\alpha\} } \sum_{j \in \beta } I_{ \alpha \beta }^{ ij } (t).
\end{equation}
Here, the gain term on the right hand side involves the summation over each individual state $ j\in \beta $ for some compartment $ \beta \neq \alpha $ into some state $ i\in \alpha $. The loss term has a similar form. Concerning the flux $ I_{\alpha\beta}^{ij} $ from state $i \in \alpha $ to state $j \in \beta $ there are two contributions.
\begin{enumerate}[(i)]
	\item One contribution is due to elements that are  in $ \alpha $, in some state $ r\in \alpha $, already at time zero. When they evolve in time they reach state $ i\in \alpha $ and jump to state $j \in \beta $ at time $ t $.
	\item The other contribution is due to the elements that jump from some state $ m\in \gamma $, $ \gamma\neq \alpha $ to some state $r \in \alpha $ at time $s \in (0,t)$. Then, they evolve in the compartment $\alpha$ to reach state $i \in \alpha $ and jump to state $ j \in \beta  $ at time $ t $.
\end{enumerate}
The system for the fluxes is therefore given by 
\begin{align} \label{eq:fluxes j first level}
	I_{  \alpha \beta }^{ij} (t)=\overline{I}_{\alpha \beta }^{ij} (t) + \sum_{ r \in \alpha }  \sum_{\gamma \in X \setminus \alpha} \sum_{m\in \gamma}  \int_0^t \Psi_{\alpha \beta }(t-s; r, i, j ) I^{ m r }_{\gamma \alpha } (s) ds
\end{align}
for all $ i\in \alpha $, $ j\in \beta $ and where 
\begin{align}\label{eq:InitialFluxes}
	\overline{I}_{\alpha \beta }^{ij} (t)= \sum_{ r \in \alpha} \Psi_{\alpha \beta}(t; r, i, j)  (n^0_\alpha)_r. 
\end{align}
The above system has again the form of RFEs. 
The precise form of the kernels are given in the following lemma.

Let us mention that equations \eqref{eq:N_alpha first level}, \eqref{eq:fluxes j first level} have been formulated in the supplementary materials of \cite{thurley2011derivation}.

\begin{lemma} \label{lem:from ODE to RE for compartment} 
	Let $n(t) \in \mathbb R_+^{|\Omega|}$ be the solution to \eqref{eq:DecompModelODE} and $\{N_\alpha\}$ defined in \eqref{eq:DecompDefNumberComp}. Then $\{N_\alpha\}$ satisfies \eqref{eq:N_alpha first level} and the corresponding fluxes $ I_{\alpha \beta }^{ij} $ are given in terms of $ n(t) $ by
	\begin{align}\label{eq:LemSolIndividualFluxes}
		I_{\alpha \beta}^{ij}(t) =(A_{\beta\alpha})_{ji} (n_\alpha(t))_i=\lambda_{ij} (n_\alpha(t))_i.
	\end{align} 
    Furthermore, the fluxes $ I_{\alpha \beta }^{ij} $ solve \eqref{eq:fluxes j first level} with kernels
	\begin{align}\label{eq:response function first level}
		\Psi_{\alpha \beta} (t; r , i,j)= (A_{\beta\alpha})_{ji} \left( e^{t A_{\alpha \alpha} } \right)_{ir} = \lambda_{i j }  \left( e^{t A_{\alpha \alpha} } \right)_{ir}
	\end{align}
	for $r,\, i \in \alpha$, $j \in \beta$ and forcing functions $ \overline{I}_{\alpha \beta }^{ij} $ defined by \eqref{eq:InitialFluxes}. 
\end{lemma}
\begin{proof} 
	By definition of $N_\alpha$ in \eqref{eq:DecompDefNumberComp}, we have
	\begin{align*}
		\frac{d}{dt} N_\alpha(t) = \textbf{e}^\top_\alpha \frac{d}{dt} n_\alpha(t) = \textbf{e}^\top_\alpha \left(  E_{\alpha \alpha} n_\alpha (t) - C_\alpha n_\alpha (t) + \sum_{\beta \in X \setminus \{\alpha\} } A_{ \alpha \beta } n_\beta (t) \right),
	\end{align*}
	where we used \eqref{eq:PrelDecompODE}. By definition of $E_{\alpha \alpha} $ we infer $ \textbf{e}^\top_\alpha  E_{\alpha \alpha} =0 $. Similarly, by definition of $C_\alpha $ and $A_{\beta \alpha } $ we have that 
	\begin{align*}
		\textbf{e}_\alpha^\top C_\alpha =  \sum_{\beta \in X \setminus \{\alpha\}}  \textbf{e}^\top_\beta  A_{\beta \alpha }.
	\end{align*}
	As a consequence of these observations we deduce that 
	\begin{align*}
		\frac{d}{dt} N_\alpha(t) &=- \sum_{\beta \in X \setminus \{\alpha\}}  \textbf{e}^\top_\beta  A_{ \beta \alpha } n_\alpha (t)   + \sum_{\beta \in X \setminus \{\alpha\} } \textbf{e}^\top_\alpha A_{ \alpha \beta } n_\beta (t) 
		\\
		&= - \sum_{\beta \in X \setminus \{\alpha\}}  \sum_{j \in \beta } \sum_{ i \in \alpha }  (A_{ \beta \alpha })_{ji} (n_\alpha (t))_i   + \sum_{i \in \alpha} \sum_{\beta \in X \setminus \{\alpha\} } \sum_{j \in \beta } (A_{ \alpha \beta })_{ij} (n_\beta (t))_j
		\\
		&= - \sum_{ i \in \alpha }\sum_{\beta \in X \setminus \{\alpha\}}  \sum_{j \in \beta } I^{ij}_{\alpha\beta}(t)   + \sum_{i \in \alpha} \sum_{\beta \in X \setminus \{\alpha\} } \sum_{j \in \beta } I^{ji}_{\beta\alpha}(t),
	\end{align*}
	with $ I_{\alpha \beta }^{ij}(t) $ given by \eqref{eq:LemSolIndividualFluxes}. 
	
	We now show that these fluxes satisfy \eqref{eq:fluxes j first level} with corresponding kernels and forcing functions in \eqref{eq:response function first level} respectively \eqref{eq:InitialFluxes}. To this end, we note that due to \eqref{eq:PrelDecompODE} we have 
	\begin{align*}
		n_\alpha(t)= e^{  t A_{\alpha \alpha  }} n_\alpha^0 + \sum_{\gamma \in X \setminus \{\alpha\} } \int_0^t  e^{ (t-s)A_{\alpha \alpha}  } A_{\alpha \gamma }  n_\gamma (s) ds . 
	\end{align*}
	Writing the $ i $-th component of the above expression and multiplying by $\lambda_{ij} $ yields
	\begin{align*} 
		I_{\alpha \beta }^{ij}(t) &= \lambda_{i j }( n_\alpha(t))_i
		\\
		&= \lambda_{i j } \sum_{r \in \alpha} \left( e^{t A_{\alpha \alpha} } \right)_{ir}  (n_\alpha^0)_r  + \lambda_{i j } \sum_{r \in \alpha} \sum_{\gamma \in X \setminus \{\alpha\} } \sum_{ m \in \gamma} \int_0^t \left( e^{(t-s) A_{\alpha \alpha}  }\right)_{ir}\left( A_{\alpha \gamma } \right)_{rm}   (n_\gamma (s))_m ds 
		\\
		&= \overline{I}_{\alpha \beta }^{ij} (t)  + \sum_{r \in \alpha} \sum_{\gamma \in X \setminus \{\alpha\} } \sum_{ m \in \gamma}  \int_0^t  \Psi_{\alpha \beta}(t-s;r,i,j) \lambda_{ mr } (n_\gamma (s))_m ds
		\\
		&= \overline{I}_{\alpha \beta }^{ij} (t)  + \sum_{r \in \alpha}\sum_{\gamma \in X \setminus \{\alpha\} }  \sum_{ m \in \gamma} \int_0^t  \Psi_{\alpha \beta}(t-s;r,i,j) I^{mr}_{\gamma\alpha}(s) ds
	\end{align*} 
	This concludes the proof. 
\end{proof}
With this we can give the proof of Theorem \ref{thm:ReductionODEtoRE}.
\begin{proof}[Proof of Theorem \ref{thm:ReductionODEtoRE}]
	According to Lemma \ref{lem:from ODE to RE for compartment} the function $ \{N_\alpha\} $ solves \eqref{eq:N_alpha first level} with fluxes satisfying \eqref{eq:ProofThmRenewalEq}. Thus defining the fluxes $ \{S_\alpha\} $, $ \{J_\alpha\} $ by
	\begin{align}\label{eq:IndividualFluxesToGeneral}
		(S_\alpha)_i = \sum_{\beta \in X \setminus \{\alpha\} } \sum_{j \in \beta }  I_{\beta\alpha}^{ji}, \quad (J_\alpha)_i = \sum_{\beta \in X \setminus \{\alpha\} } \sum_{j \in \beta }  I_{\alpha\beta}^{ij}.
	\end{align}
 	we obtain \eqref{eq:Concentration}. Note that the formula in \eqref{eq:ReductionODEtoREFluxes} is a consequence of \eqref{eq:LemSolIndividualFluxes} and \eqref{eq:IndividualFluxesToGeneral}. Moreover, from \eqref{eq:fluxes j first level} we obtain  
 	\begin{align}\label{eq:ProofThmRenewalEq}
 		I_{  \alpha \beta }^{ij} (t)=\overline{I}_{\alpha \beta }^{ij} (t) + \sum_{ r \in \alpha } \int_0^t \Psi_{\alpha \beta }(t-s; r, i, j ) (S_\alpha(s))_r ds.
 	\end{align}
 	Consequently, summing over $ \beta \in X \setminus \{\alpha\}  $ and over $ j \in \beta $ yields that $\{ J_\alpha \} $ satisfies equation \eqref{eq:REOutflux} with kernels and forcing functions given by 
 	\begin{align*}
 		(K_{\alpha}(t))_{ir} = \sum_{\beta \in X \setminus \{\alpha\} } \sum_{j \in \beta } \Psi_{\alpha \beta }(t; r, i, j ),
 		\quad
 		(J_\alpha^0(t))_i = \sum_{\beta \in X \setminus \{\alpha\} } \sum_{j \in \beta } \overline{I}_{\alpha \beta }^{ij} (t).
 	\end{align*}
 	Thus, \eqref{eq:ThmIntegralKernelOutflux} and \eqref{eq:ThmInOutFlux} follow from the formula for the functions $\Psi_{\alpha \beta } $, cf. \eqref{eq:response function first level}. 
  We can argue similarly to deduce that $ \{ S_\beta \} $ satisfy equation \eqref{eq:REInflux} by summing over $ \alpha\in X $ and $i \in \alpha $ in equation \eqref{eq:fluxes j first level} and defining 
 	\begin{align*}
 		(G_{\alpha\beta}(t))_{jr} = \sum_{ i \in \alpha } \Psi_{\alpha\beta}(t; r, i, j  ),
 		\quad
 		(S_\beta^0(t))_i = \sum_{\alpha \in X \setminus \{\beta\} } \sum_{i \in \alpha } \overline{I}_{\alpha \beta }^{ij} (t)
 	\end{align*}
 	The two equations \eqref{eq:ThmIntegralKernelInflux} and \eqref{eq:ThmInOutFlux} are, therefore, a consequence of \eqref{eq:response function first level}.
\end{proof}
\begin{remark}
The kernels given by \eqref{eq:ThmIntegralKernelInflux} and \eqref{eq:ThmIntegralKernelOutflux} satisfy both the conditions in Lemma \ref{lem:PositivityGenRFE} and \ref{lem:ConservationMassGenRFE}. 
 Indeed, we have with for any $ T\geq0 $ 
	\begin{align*}
		\int_0^T [\textbf{e}_\alpha^\top K_\alpha(r)]_j\, dr &= \sum_{i \in \alpha}\sum_{\beta\in X\setminus\{\alpha\}}\sum_{k\in \beta} \int_0^T (A_{\beta\alpha})_{ki}\left( e^{tA_{\alpha\alpha}} \right)_{ij} \, dr
		\\
		&= \sum_{i \in \alpha}\int_0^T (C_{\alpha})_{ii}\left( e^{tA_{\alpha\alpha}} \right)_{ij} \, dr.
	\end{align*}
	Recall the definition of $ C_\alpha $ in Section \ref{sec:DecompGeneralSyst}. Using $ A_{\alpha\alpha}= E_{\alpha\alpha}-C_\alpha $ we obtain
	\begin{align*}
		\int_0^T \textbf{e}_\alpha^\top K_\alpha(r)\, dr = \textbf{e}_\alpha^\top \left( I - e^{TA_{\alpha\alpha}} \right) + \textbf{e}_\alpha^\top\left( \int_0^TE_{\alpha\alpha}e^{tA_{\alpha\alpha}}\, dt \right) = \textbf{e}_\alpha^\top \left( I - e^{TA_{\alpha\alpha}} \right),
	\end{align*}
	since $ \textbf{e}_\alpha^\top E_{\alpha\alpha}=0 $. Since the semigroup $e^{tA_{\alpha\alpha}}$ induced by $ A_{\alpha\alpha} $ is positivity preserving we obtain the second inequality in \eqref{eq:LemmaPositivityAssumption}. For the first in \eqref{eq:LemmaPositivityAssumption} we use \eqref{eq:ThmInOutFlux} and the previous reasoning
	\begin{align*}
		\int_0^T\textbf{e}_\alpha^\top K_\alpha(t)n^0_\alpha \, dt = \textbf{e}_\alpha^\top \left( I - e^{TA_{\alpha\alpha}} \right)n^0_\alpha \leq \textbf{e}_\alpha^\top n^0_\alpha = N^0_\alpha.
	\end{align*}
	Finally, both conditions in \eqref{eq:LemmaConservMassCondition} follow from \eqref{eq:ThmIntegralKernelInflux} and \eqref{eq:ThmIntegralKernelOutflux}. More precisely, we have for all $\alpha\in X$ and $ k\in \alpha $
	\begin{align*}
		\sum_{\beta\in X\setminus\{\alpha\}} [\textbf{e}_\beta^\top G_{\alpha\beta}(t) ]_k- [\textbf{e}_\alpha^\top K_\alpha(t)]_k = 
		\\
		\sum_{\beta\in X\setminus\{\alpha\}} \sum_{j \in \beta}\sum_{i\in\alpha} (A_{\beta\alpha})_{ji} \left( e^{tA_{\alpha\alpha}} \right)_{ik} - \sum_{\beta\in X\setminus\{\alpha\}} \sum_{j \in \beta}\sum_{i\in\alpha} (A_{\beta\alpha})_{ji} \left( e^{tA_{\alpha\alpha}} \right)_{ik} =0.
	\end{align*}
\end{remark}

We conclude this section with the special case of system of ODEs \eqref{eq:IntrodModelODE} with graph structure $ (\Omega,\mathcal{E})  $ that is decomposed into compartments with at most one entrance point (see Definition \ref{def:EntrancePointGraph}). Then the system of equations \eqref{eq:Concentration}-\eqref{eq:REInflux}-\eqref{eq:REOutflux} can be reduced to \eqref{eq:IntroConcentrationOneEntrance}-\eqref{eq:IntroRE one entrance}-\eqref{eq:IntroConcentrationOneOutgoing}.

\begin{lemma}\label{lem: RE one entrance and general} 
	Consider $n(t) \in \mathbb R_+^{|\Omega|}$ solution to \eqref{eq:DecompModelODE} and $N_\alpha $ defined in \eqref{eq:DecompDefNumberComp}. Furthermore, assume that for all $\alpha \in X $ there is at most one entrance point in $ (\Omega,\mathcal{E}) $, denoted by $i_\alpha$ if it exists. Then the following holds.
	\begin{enumerate}[(i)]
		\item For all $ \alpha \in X $ we have $ (S_\alpha^0)_j=(S_\alpha)_j=0 $ for all $ j\neq i_\alpha $.
		\item The functions $ \{B_\alpha\}=\{(S_\alpha)_{i_\alpha}\} $, $ \{D_\alpha\} = \{\textbf{e}_\alpha^\top J_\alpha\} $ satisfy \eqref{eq:IntroRE one entrance}-\eqref{eq:IntroConcentrationOneOutgoing} with kernels and forcing functions 
		\begin{align*}
			k_\alpha &:= 
			\begin{cases}
				\sum_{j\in \alpha}(K_\alpha  )_{j i_\alpha } & \text{if } i_\alpha\in \alpha \text{ exists,}
				\\
				0 & \text{otherwise,}
			\end{cases} 
			\qquad
			\Phi_{\beta\alpha} := 
			\begin{cases}
				(G_{ \beta\alpha } )_{i_\alpha i_\beta }  & \text{if } i_\alpha\in \alpha,\, i_\beta\in \beta \text{ exist,}
				\\
				0 & \text{otherwise,}
			\end{cases}
			\\
			B_\alpha^0 &: = \begin{cases}
				( S^0_\alpha )_{i _\alpha} & \text{if } i_\alpha \in \alpha \text{ exists,}
				\\
				0 & \text{otherwise,}
			\end{cases}
			\qquad
			D_\alpha^0 := \textbf{e}_\alpha^\top J^0_\alpha.
		\end{align*}
		\item Finally, $ \{N_\alpha\} $ satisfies \eqref{eq:IntroConcentrationOneEntrance} with fluxes $ \{B_\alpha,D_\alpha\} $.
	\end{enumerate}
\end{lemma}
\begin{proof}
    We prove each claim separately.
    
    \textit{Claim (i).} Let us note that for $\alpha,\, \beta\in X$ and $j\neq i_\alpha$, $j\in \beta$ we have $(A_{\alpha\beta})_{j\ell}= \lambda_{\ell j} =0$ for all $\ell\in\beta$. In particular, by \eqref{eq:ThmIntegralKernelInflux} we have $ (G_{\beta\alpha})_{jk} \equiv 0 $ for all $k\in \beta$. Hence, \eqref{eq:ThmInOutFlux} yields (i).
    
    \textit{Claim (ii).} By Theorem \ref{thm:ReductionODEtoRE} the fluxes ${S_\alpha}$ solve the system \eqref{eq:REInflux}. The $ i_\alpha $-th component, if it exists, yields \eqref{eq:IntroRE one entrance} for $\Phi_{\beta\alpha}$ as in (ii) due to (i).

    Furthermore, $\{J_\alpha\}$ solve \eqref{eq:REOutflux}. Thus, (i) and summing over $j\in \alpha$ yields \eqref{eq:IntroConcentrationOneOutgoing} with kernels $\Phi_{\alpha}$ as in (ii).

    \textit{Claim (iii).} Finally, we conclude (iii) from \eqref{eq:Concentration} and $ B_\alpha = \textbf{e}_\alpha^\top S_\alpha = (S_\alpha)_{i_\alpha} $ if $ i_\alpha $ exists, $ B_\alpha =0 $ otherwise.
\end{proof}
\begin{remark}
    Let us recall that the kernels $ \{\Phi_{ \alpha\beta }\}_{\alpha,\beta\in X} $ in \eqref{eq:IntroConcentrationOneEntrance}-\eqref{eq:IntroRE one entrance}-\eqref{eq:IntroConcentrationOneOutgoing} are assumed to satisfy 
    \begin{align}\label{eq:KernelsOneEntranceConservation}
		\sum_{\beta\in X\setminus \{\alpha\}} \int_0^\infty \Phi_{ \alpha\beta }(t)\, dt =1.
	\end{align} 
    if the compartment $\alpha$ is not a sink, i.e. there are no connections to any other compartment $\beta\neq\alpha$. In the latter case the kernels are zero $\Phi_{\alpha\beta}\equiv0$ for all $\beta\neq\alpha$.
    
    However, even if $\Phi_{\alpha\beta}\neq0$ a further assumption on the compartment $\alpha$ is necessary to obtain \eqref{eq:KernelsOneEntranceConservation}. Observe that from (ii) in Lemma \ref{lem: RE one entrance and general} and form \eqref{eq:ThmIntegralKernelInflux} it follows that
	\begin{align*}
		\Phi_{\alpha\beta }(t) =\left( A_{\beta\alpha}e^{t A_{\alpha\alpha}} \right)_{i_\beta, i_\alpha}.
	\end{align*}
    Thus, $\Phi_{\alpha\beta}\neq0$ implies that there is a path in $(\Omega,\mathcal{E})$ from $i_\alpha$ to $i_\beta$. However, in $\alpha$ there could be a state $\Delta\in \alpha$ that has no path to any entrance point $i_\beta$, $\beta\neq\alpha$, but there is a path from $i_\alpha$ to $\Delta$. Hence, there is a sink inside the compartment $\alpha$. In this case, \eqref{eq:KernelsOneEntranceConservation} fails. The reason is that $\lim_{t\to\infty} e^{tA_{\alpha\alpha}}e_{i_\alpha}\neq0$. In fact, the left hand side in \eqref{eq:KernelsOneEntranceConservation} can be reformulated as follows. Note that $ (A_{\beta\alpha})_{j,i}=0 $ for $ j\neq i_\beta $ and any $ i\in \alpha $, since $ i_\beta $ is the only entrance point, cf. Definition \ref{def:EntrancePointGraph}. We conclude from this and the conservation property $ \sum_{ \beta \in X \setminus \{\alpha\} } \textbf{e}_{|\beta|}^\top A_{\beta\alpha} = -\textbf{e}_{|\alpha|}^\top A_{\alpha\alpha} $ that
	\begin{align*}
		\sum_{\beta\in X\setminus \{\alpha\}} \int_0^\infty \Phi_{ \alpha\beta }(t)\, dt =  -\textbf{e}_{|\alpha|}^\top \int_0^\infty A_{\alpha\alpha}e^{tA_{\alpha\alpha}} e_{i_\alpha} \, dt = 1-\lim_{t\to \infty}\textbf{e}_{|\alpha|}^\top e^{tA_{\alpha\alpha}} e_{i_\alpha}.
	\end{align*}
    Thus, we need necessarily $\lim_{t\to \infty}\textbf{e}_{|\alpha|}^\top e^{tA_{\alpha\alpha}} e_{i_\alpha}=0$ to ensure \eqref{eq:KernelsOneEntranceConservation}. This is satisfied if any state $j\in\alpha$ that has a path from $i_\alpha$ to $j$ also has a path to some entrance point $i_\beta$, $\beta\neq\alpha$.
\end{remark}

\section{Approximation of measures by means of response functions} 
\label{sec:ApproximationIntegralKernel}

In this section, we are concerned with approximation results for the kernels $ \{\Phi_{ \alpha\beta }\}_{\alpha,\beta\in X} $ appearing in \eqref{eq:IntroRE one entrance}. We restrict ourselves to the case that each compartment has at most one entrance point. Accordingly, the RFEs have the form \eqref{eq:IntroConcentrationOneEntrance}-\eqref{eq:IntroRE one entrance}-\eqref{eq:IntroConcentrationOneOutgoing}.
In Section \ref{sec:density} we prove that the  set of the response functions obtained form ODEs of the form \eqref{eq:IntrodModelODE} is dense in the space of probability measures.
In Section \ref{sec:detailed balance}, we prove that if we assume that \eqref{eq:IntrodModelODE} satisfies the detailed balance condition, then the response functions must be completely monotone functions.

\subsection{Density of the response functions obtained from ODEs in the space of probability measures} \label{sec:density}
Recall that in Section \ref{sec:ReformODEasRFE} we proved that a system of ODEs of the form \eqref{eq:IntrodModelODE} can be reduced to the system \eqref{eq:IntroConcentrationOneEntrance}-\eqref{eq:IntroRE one entrance}-\eqref{eq:IntroConcentrationOneOutgoing} by decomposing the state space into compartments with at most one entrance point, cf. Theorem \ref{thm:ReductionODEtoRE} and Lemma \ref{lem: RE one entrance and general}. Furthermore, the corresponding kernels $ \{\Psi_{ \alpha\beta }\}_{\alpha,\beta\in X} $ have the form
\begin{align}\label{eq:ApproxKernelsOneEntranceFormula}
	\Psi_{\alpha\beta }(t) =\left( A_{\beta\alpha}e^{t A_{\alpha\alpha}} \right)_{i_\beta, i_\alpha} = \sum_{j \in \alpha } \left(A_{\beta\alpha} \right)_{i_\beta, j}\left( e^{t A_{\alpha\alpha}} \right)_{j,i_\alpha}= \sum_{j \in \alpha } \lambda_{j, i_\beta}\left( e^{t A_{\alpha\alpha}} \right)_{j,i_\alpha} ,
\end{align}
where $ i_\alpha\in \alpha,\, i_\beta\in\beta $ are the entrance points in $\alpha $ and $\beta$, if they exist, $ \Psi_{\alpha\beta }\equiv0 $ otherwise.
	
We now show that for any set of integral kernels $ \{\Phi_{ \beta \alpha }\}_{\beta,\alpha\in X} $ satisfying \eqref{eq:IntodKernelConservation} it is possible to find a sequence of
 kernels $\{\Psi_{\alpha\beta}\}$ as in \eqref{eq:ApproxKernelsOneEntranceFormula} that is arbitrary close (in a suitable sense) to the kernel $ \{\Phi_{ \beta \alpha }\}_{\beta,\alpha\in X} $. 
In other words, we prove that for any system of the form \eqref{eq:IntroConcentrationOneEntrance}-\eqref{eq:IntroRE one entrance}-\eqref{eq:IntroConcentrationOneOutgoing} we can find an approximating sequence of systems of ODEs  of the form \eqref{eq:IntrodModelODE}. 
	
In order to prove this we need to specify 
\begin{enumerate}[(i)]
	\item a state space $ \Omega' $, on which the approximating ODEs will be defined;
	\item the jump rates between the states, i.e. a matrix $ A\in \R^{|\Omega'|\times|\Omega'|} $ yielding an ODE system \eqref{eq:IntrodModelODE}. This defines a directed graph $ (\Omega',\mathcal{E}') $;
	\item a decomposition $ X'\subset 2^{\Omega'} $ into compartments. 
\end{enumerate}
Since the kernels $ \{\Phi_{ \alpha \beta }\}_{\beta,\alpha\in X} $ are labeled by the set $ X $, the decomposition $ X' $ needs to be identified with $ X $. More precisely, we require the existence of a relabeling, i.e. there is a bijective mapping $ \iota:X\to X' $ which associates to each compartment $ \alpha\in X $ a compartment $ \iota(\alpha)=\alpha'\in X' $. Then, we apply the procedure in Section \ref{sec:ReformODEasRFE} to $ (\Omega',\mathcal{E}') $ and the decomposition $ X' $ which yields kernels $ \{\Psi_{\beta'\alpha'}\}_{\beta',\alpha'\in X'} $. Choosing $ (\Omega',\mathcal{E}') $ and $ X' $ appropriately, these kernels yield an approximation to $ \{\Phi_{ \alpha \beta }\}_{\beta,\alpha\in X} $. 

The approximation result holds in terms of the weak topology of measures. Recall that we denote by $ d_w $ a metric inducing weak convergence, see Subsection \ref{sec:notation}. Let us note that in the precise statement below we assume that the response functions are merely finite, non-negative measures $\Phi_{\alpha\beta}\in \mathcal{M}_{+,b}(\R_+)$ rather than measurable functions.
\begin{theorem}\label{thm:ApproximationIntegralKernel}
	Consider a family of finite, non-negative measure kernels $ \{\Phi_{ \alpha \beta }\}_{\beta,\alpha\in X} $ with $ |X|<\infty $ satisfying \eqref{eq:IntodKernelConservation} or \eqref{eq:ApproxKernelsOneEntranceConservationzeroIntro}. Let $ \varepsilon>0 $ be arbitrary. Then, we can find a finite state space $ \Omega' $, a decomposition $ X'\subset 2^{\Omega'} $ and a matrix $ A\in \R^{|\Omega'|\times |\Omega'|} $ with the following properties.
	\begin{enumerate}[(i)]
		\item The matrix $ A $ satisfies $ A_{ij}\geq0 $ for all $ i\neq j $, $ i,\, j\in \Omega' $ and $ \textbf{e}_{|\Omega'|}^\top A=0 $.
		\item There is a bijection $ \iota : X\to X' $.
		\item The partition $ X' $ decomposes $ (\Omega';\mathcal{E}') $ into compartments with at most one entrance point, denoted by $ i_{\alpha}\in \alpha $, $ \alpha\in X' $ if it exists, see Definition \ref{def:EntrancePointGraph}.
		\item The kernels given by 
        \begin{align}\label{eq:ThmApproxKernelsOneEntranceFormula}
			\Psi_{ \alpha\beta }(t) = \sum_{j \in \alpha } \left(A_{\beta\alpha} \right)_{i_\beta, j}\left( e^{t A_{\alpha\alpha}} \right)_{j,i_\alpha}
		\end{align}
		if $ i_{\alpha}, \, \omega_{\alpha,\beta}\in \alpha $ exist, $\Psi_{ \alpha\beta }\equiv 0 $ otherwise are such that 
  	     \begin{align*}
			\int_0^\infty \Psi_{ \alpha\beta }(t)\, dt = \Phi_{ \alpha \beta }(\R_+) =: p_{\alpha\beta},
            \quad
            \sum_{\alpha,\beta\in X, \, p_{\alpha\beta}>0} d_w \left( \frac{1}{p_{\alpha\beta}}\Psi_{ \iota(\alpha),\iota(\beta) }(t)\, dt, \, \frac{1}{p_{\alpha\beta}}\Phi_{ \alpha \beta }\right)  \leq \varepsilon.
		\end{align*}
	\end{enumerate}
\end{theorem}
Furthermore, let us mention that the above result is reminiscent of the so-called \textit{phase method} in queuing theory, see e.g. \cite[Section III.4]{Asmussen2003Queues}, allowing to approximate any probability measure on $ (0,\infty) $ by a combination of exponentially distributed times. For the proof of Theorem \ref{thm:ApproximationIntegralKernel} we make use of this construction which contains probabilistic arguments.
\begin{proof}[Proof of Theorem \ref{thm:ApproximationIntegralKernel}]
	We split the proof into two steps. In the first step we recall a standard result from queuing theory which allows to approximate any probability distribution on $\R_+$. Furthermore, we show how to obtain such approximations from ODE models of the form \eqref{eq:IntrodModelODE} for a certain choice of matrices $A$. In the context of the above theorem, this allows to approximate one given response function by an ODE model. This serves as a basic building block for the next step. In Step 2 we then show how to construct the whole state space $\Omega'$ with a partition $X'$ as well as the matrix $A$ that allows to find an approximation for all response functions $\{\Phi_{\alpha\beta}\}$ simultaneously.
	
	\textit{Step 1.} Let us recall the following result from \cite[Theorem 4.2]{Asmussen2003Queues}: any probability measure $ \Phi $ on $ (0,\infty) $ can be approximated in the weak topology by distributions $ F $ with density
	\begin{align}\label{eq:ProofErlangDist}
		f(t):=\sum_j^L q_j \gamma_{M,m_j}(t), \quad \gamma_{M,m}(t) = \dfrac{M^{m+1}}{m!} t^m e^{-M t}, \quad t>0
	\end{align}
	for some $ L\in \N $, $ q_j>0$, $ \sum_j q_j=1 $ and some $ M\in \N $, $ m_j\in \N $. Here, $ \gamma_{M,m} $ is the density of the Gamma distribution appearing as an $ m $-fold convolution of the exponential distribution with rate $ M\in \N $. Thus, for $ \varepsilon>0 $ we can find corresponding parameters such that $ d_w(F,\Phi)<\varepsilon $ with the measure $ F= f(t)\, dt $.

 We now construct an ODE system of the form \eqref{eq:IntrodModelODE} and relate $f(t)$ in \eqref{eq:ProofErlangDist} with the response function corresponding to the ODEs system.
   We start by describing the structure of the graph. Consider one starting node $ \zeta $ and $ L $ chains with $ m_j-1 $ nodes (that we denote with $x_{n j } $ with $1 \leq j \leq L $ and $1 \leq n \leq m_j -1 $) bifurcating from $ \zeta $. Furthermore, all chains have $ \omega $ as final node, see Figure \ref{fig:proofapprox1}. 
	\begin{figure}[ht]
		\centering
		\includegraphics[width=0.5\linewidth]{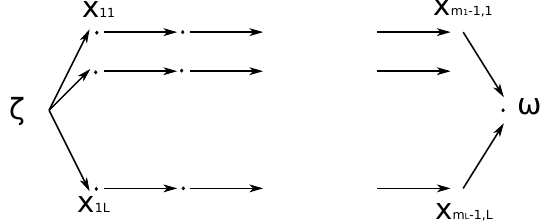}
		\caption{Graph structure corresponding to approximation in Step 1.}
		\label{fig:proofapprox1}
	\end{figure}

We assume that the rate to jump from $\zeta $ to $x_{1 j } $ is equal to $q_j M$ while the rate to jump from $x_{n-1, j}$ to $x_{n, j}$ is equal to $M$.
Finally the rate to go from $x_{m_{j}-1 }$ to $\omega $ is also equal to $M$. 
This defines the ODE system 
\begin{align} \label{ODE approx}
   \begin{split}
        \frac{d \zeta }{dt} &= - M \zeta,  \quad \zeta(0)=1 \\
        \frac{d x_{1,j }}{dt } &= q_j M \zeta - M x_{1j},  \quad  x_{1, j }(0)=0, \quad \text{ for } 1 \leq j \leq L  \\
        \frac{d x_{n, j }}{dt } &= M x_{n-1, j} - M x_{n, j}, \quad x_{n, j }(0)=0,  \quad \text{ for } 1 \leq j \leq L \text{ and } 2 \leq n \leq m_j-1
   \end{split}
\end{align} 
which can be written also as $\frac{dx}{dt}= A x $ where $x=(\zeta, x_{1, 1}, \dots x_{m_L-1 , L }) $ and where the matrix $ A $ is such that $ A_{ij}\geq0 $ for $ i\neq j $ and $ \textbf{e}_n^\top A=0 $.
	
We now prove that $ M\sum_{j=1}^L x_{m_{j}-1, j } (t) = f(t) $. Indeed 
	\begin{align}\label{eq:ProofApproxKernelsOneEntranceFormula}
		\hat{f}(z)= \sum_{j =1 }^L q_j \widehat{\gamma_{M , m_j} }(z) = \sum_{j =1 }^L \frac{ q_j M^{m_j +1}}{(1+ M z )^{m_j+1}}  = M \sum_{j=1 }^L \widehat{x_{m_j-1, j }}(z) 	\end{align}
  where the last inequality has been obtained solving \eqref{ODE approx} via Laplace transforms. 
	Equality \eqref{eq:ProofApproxKernelsOneEntranceFormula} implies that $ M\sum_{j \in L} x_{m_{j}-1, j } (t) = f(t) $. 
 Notice that if $\alpha=\{\zeta, x_{1,1}, \dots x_{m_{L}-1, L } \}$ and  $\beta =\{ \omega \}$ are two compartments, 
then $f(t) $ is the response function $\Phi_{\alpha \beta } $ in Lemma \ref{lem: RE one entrance and general}. 
	
	\textit{Step 2.} We now use Step 1 to construct the state space $ \Omega' $, the matrix $ A $ and the partition $ X' $. To this end, we construct as in Step 1  probability measures $ F_{\alpha\beta} $,  for each pair of compartments, with densities $ f_{\alpha\beta} $ of the form \eqref{eq:ProofErlangDist} that satisfy
	\begin{align}\label{eq:ProofApproxKernels}
		\sum_{\alpha, \beta \in X, \, p_{\alpha\beta}>0} d_w \left( F_{\alpha\beta}, \dfrac{1}{p_{\alpha\beta}}\Phi_{ \alpha \beta }\right) \leq \varepsilon.
	\end{align}
	Observe that $ \frac{1}{p_{\alpha\beta}}\Phi_{ \alpha \beta } $ is a probability measure. 
	
	We now construct the ODE system and its corresponding graph. First, let $ (i_\alpha)_{\alpha\in X} $ be states representing the entrance points of the compartments to be constructed. Next, we fix one compartment $ \alpha\in X $ and specify the states and rates inside the compartment $ \alpha'=\iota(\alpha)\subset \Omega' $. For any $ \beta\neq \alpha $ with $ p_{\alpha\beta}>0 $ introduce a state $ \zeta_{\alpha\beta} $. The rate to jump from $ i_\alpha $ to $ \zeta_{\alpha\beta} $ is given by $Mp_{\alpha\beta}$. 
	
	Now we add to the state $ \zeta_{\alpha\beta} $ the graph in Step 1, see Figure \ref{fig:proofapprox1}, with end point $ \omega:=i_\beta $. 

	\begin{figure}[ht]
		\centering
		\includegraphics[width=0.5\linewidth]{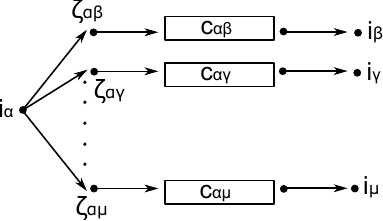}
		\caption{Graph structure in compartments in Step 2.
  Here with $\{C_{\alpha \beta } \}_{\alpha \neq\beta } $ we denote the graphs that connects $\zeta=\zeta_{\alpha \beta } $ to $\omega =i_\beta$ as in Figure \ref{fig:proofapprox1}.}
		\label{fig:proofapprox2}
	\end{figure}
	
	We can repeat this construction for any compartment and we obtain the state space $ \Omega' $, the partition $ X' $, the relabeling $ \iota: X\to X' $ and the matrix $ A $. By definition $ A $ satisfies statement (i) in the theorem. Also statements (ii) and (iii) are satisfied. 
	
	Finally, concerning (iv) we use Step 1 to deduce that for every $\alpha \neq  \beta$ the approximation \eqref{eq:ProofApproxKernels} holds for the measures $F_{\alpha \beta} $ with densities given by 
	\begin{align*}
		f_{\alpha\beta}(t)  = \sum_{j \in \alpha } \left(A_{\beta\alpha} \right)_{i_\beta, j}\left( e^{t A_{\alpha\alpha}} \right)_{j,i_\alpha}=\Psi_{\alpha\beta}(t). 
	\end{align*}
\end{proof}

\subsection{Response functions corresponding to ODEs satisfying the detailed balance condition}
\label{sec:detailed balance}
In this section we demonstrate that a particular structure of the biochemistry of the system imposes some restriction on the form of the response functions. 
Specifically we consider here systems satisfying the detailed balance condition.

We say that the ODE model \eqref{eq:IntrodModelODE} satisfies the detailed balance condition if there is $ \mu\in\R_+^{|\Omega|} $, $ \mu_j>0 $, $\sum_{i\in\Omega}\mu_i=1$ such that $ A\mu=\mu $ and $ A_{ij}\mu_j=A_{ji}\mu_i $ for all $ i,j\in \Omega $. The latter condition is equivalent to the fact that $ M^{-1}AM $ is a symmetric matrix, where
\begin{align*}
	M:=\operatorname{diag}\left( \{\sqrt{\mu_j}\}_{j\in \Omega} \right).
\end{align*}
Let us mention that the vector $\mu$ is usually referred to as the equilibrium distribution of the ODE system. We will prove, see Theorem \ref{thm:DetailedBalanceKernel}, that the corresponding response functions are completely monotone. 
	We recall that a completely monotone function $ f:\R_+\to \R $ is continuous on $ \R_+ $ and satisfies $ (-1)^nf^{(n)}(t)\geq0 $ for all $ n\in \N_0 $, $ t>0 $. An example of completely monotone function is the exponential function $e^{- \lambda t}$ with $\lambda \geq 0$ or the function $1/t $ for $t>0$.
 It is well-known that completely monotone functions are exactly the Laplace transforms of non-negative finite Borel measures on $ \R_+ $. See for instance \cite[Chapter XIII]{feller1967introduction}.
 
We have the following result.
\begin{theorem}\label{thm:DetailedBalanceKernel}
	Let \eqref{eq:IntrodModelODE} with $ A\in \R^{|\Omega|\times|\Omega|} $ satisfy detailed balance for $ \mu\in \R_+^{|\Omega|} $. Let $ X $ be a decomposition such that any compartment has at most one entrance point in $ (\Omega,\mathcal{E}) $. Then, we have for all $ \alpha,\, \beta\in X $, $ \alpha\neq \beta $ and some $ \kappa_j,\, \nu_j\geq0 $
	\begin{align*}
		\Phi_{ \beta \alpha }(t) = \lambda_{i_\beta,i_\alpha} \sum_{j \in \beta} \kappa_j^2 e^{-\nu_j t}, \quad \sum_{j \in \beta } \kappa_j^2=1,
	\end{align*}
	whenever the entrance points $ i_\alpha\in \alpha $, $ i_\beta\in \beta $ exists, otherwise $ \Phi_{ \beta \alpha }\equiv0 $.
	
	In particular, $ \Phi_{ \beta \alpha } $ is completely monotone.
\end{theorem}

	Theorem \ref{thm:DetailedBalanceKernel} implies that not all the sets of kernels $ \{\Phi_{ \beta \alpha }\} $ can be approximated by those appearing in an ODE model \eqref{eq:IntrodModelODE} with detailed balance. This is in contrast with Theorem \ref{thm:ApproximationIntegralKernel} and implies a very strong restriction on the integral kernels that can be approximated by models satisfying the detailed balance condition.
\begin{proof}[Proof of Theorem \ref{thm:DetailedBalanceKernel}]
	Let $ \alpha,\, \beta\in X $ with entrance points $ i_\alpha,\, i_\beta $, respectively. By Theorem \ref{thm:ReductionODEtoRE} and Lemma \ref{lem: RE one entrance and general} we have 
	\begin{align*}
		\Phi_{\beta\alpha}(t)= \left( A_{\alpha\beta} e^{A_{\beta\beta}t}\right)_{i_\alpha,i_\beta} = \lambda_{i_\beta,i_\alpha} \left( e^{A_{\beta\beta}t}\right)_{i_\beta,i_\beta}.
	\end{align*}
	Here, we used that $ A_{\beta\alpha} $ has only one non-zero element due to the detailed balance condition and the assumption that there is at most one entrance point in $\alpha, \, \beta$. Let us define $ M^\beta:= \operatorname{diag}(\{\sqrt{\mu_j}\}_{j\in\beta}) $. Using the decomposition $ A_{\beta\beta}=E_{\beta\beta}-C_{\beta} $ as in Section \ref{sec:DecompGeneralSyst} we obtain that $ D_{\beta\beta}:=(M^\beta)^{-1}A_{\beta\beta}M^\beta $ is symmetric. Recall that $E_{\beta\beta}$ defines a conservative ODE system, while $-C_{\beta}$ is an additional loss term in $ A_{\beta\beta}=E_{\beta\beta}-C_{\beta} $. Thus, the eigenvalues $ \{\nu_j\}_{j\in\beta} $ of $D_{\beta\beta}$ are non-positive. Using an orthonormal basis of eigenvectors we deduce that there is an orthogonal matrix $ Q $ such that
	\begin{align*}
		\Phi_{\beta\alpha}(t) = \lambda_{i_\beta,i_\alpha} \left( Q\operatorname{diag}\left( \{e^{-\nu_jt}\}_{j\in\beta} \right)  Q^\top\right)_{i_\beta,i_\beta} = \lambda_{i_\beta,i_\alpha} \sum_{j \in \beta} Q_{j,i_\beta}^2e^{-\nu_jt}.
	\end{align*}	
	Since $ Q $ is orthogonal we have $ \sum_{j \in \beta }Q_{j,i_\beta}^2 =1 $. Defining $ \kappa_j:=Q_{j,i_\beta} $ concludes the proof.
\end{proof}

\section{Characterization of response functions yielding Markovian dynamics}\label{subsec:CharMarkov}
As anticipated in the introduction, one important property that is lost when reformulating the system of ODEs \eqref{eq:DecompModelODE} using the response function formalism \eqref{eq:Concentration}-\eqref{eq:REInflux}-\eqref{eq:REOutflux}, is the Markovianity of the evolution of the number of individuals $ \{N_\alpha \} $.

In this section, we give a precise definition of Markovianity and characterize the integral kernels yielding a Markovian dynamics. To this end, we restrict ourselves to decompositions $X$ of the state $\Omega $ into compartments with at most one entrance point. In other words, we restrict our attention to RFEs of the form \eqref{eq:IntroConcentrationOneEntrance}-\eqref{eq:IntroRE one entrance}-\eqref{eq:IntroConcentrationOneOutgoing} with forcing function satisfying particular conditions (see \eqref{forcing function markovianity}).

As we will see, kernels inducing a Markovian dynamics are exactly given by exponentials, and they correspond to Markov jump processes involving exponentially distributed waiting times, see Theorem \ref{thm:CharMarkovRE} below. Similar results, for some type of non-linear structured population models can be found in  \cite{diekmann2020finite}. Here we present a simple proof for the linear renewal equations considered in this paper.

\begin{definition}[Markovianity]\label{def:markov}
We say that the evolution induced by \eqref{eq:IntroConcentrationOneEntrance}-\eqref{eq:IntroRE one entrance}-\eqref{eq:IntroConcentrationOneOutgoing} is Markovian if and only if any solution $N$ of equation \eqref{eq:IntroConcentrationOneEntrance} can be written in the form
\begin{equation}
\label{Markov ODE} 
 N (t)= e^{t A} N^0 
\end{equation} 
for all $N^0\in \R^{|X|}$. Furthermore, the matrix $A \in \mathbb R^{|X|\times |X|} $ is independent of $N^0 \in \mathbb R_+^{|X|}$, satisfying $A_{ii} = - \sum_{i \neq k  } A_{ik} $ for every $i \in \{1, \dots, |X|\} $. 
\end{definition}

Since we apply the Laplace transform in the proof of Theorem \ref{thm:CharMarkovRE} below, the following lemma is useful. 

\begin{lemma}\label{lem:existence uniqueness and expo decay} 
    Assume that each compartment in $\alpha \in X$ has at most one entrance point. Let $z_0> 0 $ such that $ \{ k_\alpha\} ,\,  \{ \Phi_{\alpha \beta } \} , \, \{ D_\alpha^0 \} ,\{ B_\alpha^0  \} \in L^1(\mathbb R_+, e^{- z_0 t} dt )$ for every $\alpha ,\,  \beta \in X$.
    
    Then, there exists a unique solution $ \{N_\alpha, B_\alpha, D_\alpha\} $ to \eqref{eq:IntroConcentrationOneEntrance}-\eqref{eq:IntroRE one entrance}-\eqref{eq:IntroConcentrationOneOutgoing} such that for every $\alpha \in X$ it holds that $B_\alpha,\,  D_\alpha \in L^1(\mathbb R_+, e^{- x_0 t} dt)$ for some $x_0 > z_0$ and $N_\alpha \in W^{1,1}(\mathbb R_+, \mathbb R_+ )$ satisfies
    \begin{equation} \label{bound N} 
        \int_0^\infty N_\alpha(t) e^{- x_0 t } dt < \infty. 
    \end{equation} 
\end{lemma}
\begin{proof}
We refer to \cite[Section 3.4]{Gripenberg1990VolterraIntEq} for the existence of $B_\alpha  \in L^1(\mathbb R_+, e^{- x_0 t} dt)$ for some $x_0 > z_0 $. Thanks to the assumptions on $k_\alpha $ and $D^0_\alpha $ we deduce that $D_\alpha  \in L^1(\mathbb R_+, e^{- x_0 t} dt)$. Hence \eqref{bound N} follows by Gronwall's Theorem. 
\end{proof}

\begin{theorem}[Markovianity for RFEs with at most one entrance point] \label{thm:CharMarkovRE}
    Under the conditions of Lemma \ref{lem:existence uniqueness and expo decay}, assume moreover that $ \{ k_\alpha \}  $, $\{ \Phi_{\alpha \beta } \}  $ satisfy \eqref{kernel condition for conservation of numbers} and that for every $\alpha \in X $ either equality \eqref{eq:IntodKernelConservation} or \eqref{eq:ApproxKernelsOneEntranceConservationzeroIntro} holds.
    Let $\{N_\alpha, B_\alpha , D_\alpha\}$ be the solution to \eqref{eq:IntroConcentrationOneEntrance}-\eqref{eq:IntroRE one entrance}-\eqref{eq:IntroConcentrationOneOutgoing} with
    \begin{equation} \label{forcing function markovianity}
        D_\alpha^0 (t)=k_\alpha(t) N^0_\alpha, \quad B_\alpha^0(t)= \sum_{\beta \in X \setminus \alpha } \Phi_{\alpha \beta }(t) N_\beta^0
    \end{equation}
    for every $t \geq 0$ and $N^0\in \mathbb R^{|X|}$. 
      Then, the evolution of $ \{N_\alpha\} $ is Markovian, in the sense of Definition \ref{def:markov}, if and only if for all $t \geq 0$
    \begin{equation}\label{Phi Markovian} 
        \Phi_{\alpha \beta } (t)= \lambda_{ \alpha \beta} e^{- t \sum_{\beta \in X \setminus \{ \alpha\} } \lambda_{\alpha \beta }   }
    \end{equation} 
    for some $\lambda_{\alpha \beta } \in \mathbb R_+ $.
\end{theorem} 
\begin{proof}
    \textit{Step 1.} We first assume that the dynamics are Markovian and deduce \eqref{Phi Markovian}. Let $\Delta:=\left\{ z \in \mathbb C :  \Re z > x_0 \right\} $. 
    We apply the Laplace transform to each term in \eqref{eq:IntroRE one entrance}-\eqref{eq:IntroConcentrationOneOutgoing} yielding
    \begin{align}
        \hat{D}(z) &= \hat{D}^0(z) + Q(z) \hat{B} (z)\label{eq:laplace D}
        \\
        \hat{B}(z) &= \hat{B}^0(z) + M(z) \hat{B} (z)\label{eq:laplace B}
    \end{align}
     for $z \in \Delta$. 
    Here, $\hat{D}(z) \in \mathbb R_+^{|X|} $ is the vector whose $\alpha$-th element is the Laplace transform of $D_\alpha $ and we define $\hat{D}^0(z),\,  \hat{B}^0(z),\, \hat{B}(z)  \in \mathbb R_+^{|X|}$ analogously. The matrix $Q(z)$ is diagonal with $(Q(z))_{\alpha \alpha }=\hat{k}_\alpha(z)$ and $(M(z))_{\alpha \beta }= \hat{\Phi}_{\beta \alpha }(z) $. 

    Similarly, we obtain from \eqref{eq:IntroConcentrationOneEntrance} that 
    \begin{equation} \label{eq:laplace N} 
        z\hat{N} (z) = N^0+  \hat{B} (z) - \hat{D} (z), 
    \end{equation} 
     for $z \in \Delta$.
    Notice that $z \mapsto (I - M(z))^{-1}$ is a meromorphic map on $\Delta$. We denote by $P $ the set of poles in $\Delta$. From equation \eqref{eq:laplace B} we deduce that for $z \in \Delta \setminus P $ it holds that
    \begin{equation}\label{eq: Proof Lap B function of B_0}
        \hat{B}(z)=\left( I - M(z) \right)^{-1}\hat{B}^0(z).
    \end{equation}
    Substituting this in \eqref{eq:laplace N} and using \eqref{eq:laplace D} we obtain 
    \begin{align}\label{eq:N hat 1}
        z \hat{N} (z) = N^0+ \left( I-Q(z) \right) \left( I - M(z) \right)^{-1}\hat{B}^0(z)- \hat{D}^0(z)
    \end{align} 
    for $z \in \Delta \setminus P $. 
    By assumption there exists a matrix $A$ with $\frac{d N}{ dt} = AN$. Since $N$ satisfies \eqref{bound N} the spectral radius of the matrix $A$ is smaller or equal than $x_0$. Applying the Laplace transform yields 
    \begin{equation} \label{eq:N hat 2}
        \hat{N}(z)= (z I- A )^{-1} N^0.
    \end{equation}
    for $z\in \Delta$. 
    Thus, combining \eqref{eq:N hat 1}, \eqref{eq:N hat 2} and $\hat{B}^0(z) = M(z) N^0$, $\hat{D}^0(z) = Q(z) N^0$ we obtain
    \begin{align*}
        A (zI - A)^{-1} N^0 = \left(I - Q(z) \right) \left( I - M(z) \right)^{-1}M (z) N^0- Q(z) N^0
    \end{align*}
    for $z\in \Delta \setminus P$. Since this holds for any $N^0\in \R^{|X|} $ this implies
    \begin{align*}
        A (z I - A)^{-1} = \left(I - Q(z) \right) \left( I - M(z) \right)^{-1}M (z) - Q(z) = - I +  \left(I - Q(z) \right) \left( I - M(z) \right)^{-1}.
    \end{align*} 
    After rearranging we obtain 
    \begin{align}\label{eq: M Q}
        z M(z) = z Q(z)+ A(I - Q(z))
    \end{align}
    for $z \in \Delta \setminus P$. 
    Now recall that $Q(z) $ is diagonal with $Q_{\alpha \alpha}(z)= \hat{k}_\alpha(z)$ and $(M(z))_{\alpha \alpha } =0$. Thus, \eqref{eq: M Q} yields for $z\in \Delta \setminus P$
    \begin{equation} \label{lap transf Q}
        z\hat{k}_{ \alpha }(z) + \lambda_{\alpha \alpha }(1- \hat{k}_\alpha(z)) =0
    \end{equation}
    where $\lambda_{\alpha \alpha}:= A_{\alpha \alpha} $. When instead $\alpha \neq \beta $ then
    \begin{equation} \label{lap transf M}
        z \hat{\Phi}_{\alpha \beta }(z)= \lambda_{ \alpha \beta} (1- \hat{k}_\alpha (z) )
    \end{equation} 
    for $ z\in \Delta \setminus P$, where $\lambda_{\alpha \beta } := A_{\beta \alpha }$. 
    Thus, \eqref{lap transf Q} and \eqref{lap transf M} yield 
    \begin{equation} \label{laplace Phialphabeta} 
        z \hat{\Phi}_{\alpha \beta} (z)= \lambda_{ \alpha \beta} \left( 1+ \frac{\lambda_{\alpha \alpha }}{z- \lambda_{\alpha \alpha}} \right)= z\frac{\lambda_{\alpha \beta }}{z-\lambda_{\alpha \alpha }}
    \end{equation}
    for $z \in \Delta \setminus P$. 
    We deduce that $ \Phi_{\alpha \beta }(t)= \lambda_{ \alpha \beta} e^{t  \lambda_{\alpha \alpha}}$ and that $ k_{\alpha } (t)= - \lambda_{\alpha \alpha} e^{t \lambda_{\alpha \alpha}} $. Then, \eqref{Phi Markovian} follows from $\lambda_{\alpha \alpha } = - \sum_{\beta \in X \setminus \{\alpha\} } \lambda_{\alpha \beta }$. 

    \textit{Step 2.} Assume now that $\Phi $ is given by \eqref{Phi Markovian}. Then we have $\hat{\Phi}_{\alpha \beta } (z)= \lambda_{\alpha \beta }/(z-\lambda_{\alpha \alpha })$ and $\hat{k}_\alpha (z)=-\lambda_{\alpha \alpha }/(z-\lambda_{\alpha \alpha })$ for $z>0$, where $\lambda_{\alpha \alpha } = - \sum_{\beta \in X \setminus \{\alpha\} } \lambda_{\beta \alpha } $. Let us define the matrix $ M(z) $ via $M_{\alpha \alpha}(z)=0$, $M_{\alpha \beta}(z)=\hat{\Phi}_{\beta \alpha } (z)= \lambda_{\beta\alpha }/(z-\lambda_{\beta \beta })$ and the diagonal matrix $Q(z)$ via $Q_{\alpha \alpha } (z) =\hat{k}_\alpha (z)= - \lambda_{\alpha \alpha }/(z-\lambda_{\alpha \alpha })$. As in Step 1, we apply the Laplace transform to \eqref{eq:IntroConcentrationOneEntrance}-\eqref{eq:IntroRE one entrance}-\eqref{eq:IntroConcentrationOneOutgoing} and use \eqref{eq: Proof Lap B function of B_0}, yielding for $z>0$
    \begin{align*}
        z \hat{N}(z) = N^0+(I-Q(z)) (I-M(z))^{-1} M(z) N^0 - Q(z) N^0 = (I-Q(z))(1-M(z))^{-1}N^0.
    \end{align*}
    Thus, the matrix $A(z):= z-z(I-M(z))(I-Q(z))^{-1} $ satisfies $\hat{N}(z) =( z I -A)^{-1} N^0$. To conclude, notice that $A_{\alpha \beta  } (z)= \lambda_{\beta \alpha} $ when $\alpha \neq \beta $ while $A_{\alpha \alpha } (z)= \lambda_{\alpha \alpha} $. In particular, $A$ is independent of $z$. Thus, the requirements of Definition \ref{def:markov} are satisfied.
\end{proof}

\section{Long-time behaviour of the solution of the classical RFEs}\label{sec:long-time}
In this section we study the long-time behaviour of equations of the form  \eqref{eq:IntroConcentrationOneEntrance}, \eqref{eq:IntroRE one entrance},\eqref{eq:IntroConcentrationOneOutgoing}. 
Once the asymptotic behaviour of $B(t) $ is known, the one of the vector $D(t)$ follows from the equality \eqref{eq:IntroConcentrationOneOutgoing}. 
Similarly the long-time behaviour of $N$, satisfying \eqref{eq:IntroConcentrationOneEntrance}, can be obtained starting from the asymptotics of $B$ and $D$.

To study the long-time behaviour of $B(t)$ we will need to make an additional assumption on the kernels $\Phi$.
More precisely we assume that the matrix $\Phi$ such that $(\Phi)_{\alpha \beta } = \Phi_{\beta \alpha}$ is irreducible. 
This corresponds to assuming that every compartment $\alpha $ can be reached by any other compartment $\beta $. In other words the graph having as vertexes the compartments must be strongly connected. 
This assumption guarantees the existence of an invariant measure, which is proven here to be stable. The irreducibility assumption of $\Phi$ corresponds to the irreducibility property which guarantees the existence of and convergence to an invariant measure for Markov chains \cite[Chapter VIII]{feller1967introduction}.

To analyse the long-time behaviour of \eqref{eq:IntroRE one entrance} we will use Laplace transforms. In particular, we follow the approach in \cite{diekmann2012delay}. 
 In order to apply Laplace transforms we make the same assumptions as in Lemma \ref{lem:existence uniqueness and expo decay} on the forcing function $B^0$ and on the kernel $\Phi$. 
We also use the same notation as in the proof of Lemma \ref{lem:existence uniqueness and expo decay} for $M(z), \Delta, P$.  
Namely $\Delta $ is the subset of the complex plain on which the Laplace transform is defined, $(M(z))_{\alpha \beta } = \hat{\Phi}_{\beta \alpha }(z) $ for $z \in \Delta $ and $P $ is the set of the poles of the map $z \mapsto (I - M(z))^{-1}$, which is meromorphic on $\Delta$. 
For $z \in \Delta \setminus P $ we have, cf. \eqref{eq: Proof Lap B function of B_0},
    \begin{equation}\label{eq: Lap B function of B_0}
        \hat{B}(z)=\left( I - M(z) \right)^{-1}\hat{B}^0(z).
    \end{equation}
  The first step for analysing the long-time behaviour of $B$ consist in finding a solution $(z, v) $, with $ r \in \mathbb R$ and $v \in \mathbb R_+^{|X|}$, to the  non-linear eigenproblem
    \begin{equation} \label{eq:non linear eigenproblem}
    v= M(z) v. 
    \end{equation}

    \begin{lemma} \label{lem:eigensolution}
        Let the assumptions of Lemma \ref{lem:existence uniqueness and expo decay} hold. 
        Moreover, assume that for every $z \in \mathbb R$ the matrix $M(z)$ is irreducible, i.e. $(M(z))_{ij} \neq 0$ for every $i \neq j $.
        Then there exists a unique real eigen-couple $(z, v)=(0, v_0 ) $ satisfying equation \eqref{eq:non linear eigenproblem} with $v_0 \in \mathbb R_+^{|X|} $.
    \end{lemma}
\begin{proof}
We recall that by the Perron-Frobenius theorem for irreducible matrices, see \cite{carl2000matrix}, for every $z \in \mathbb R $ the spectral radius $\rho(M(z))$ of the matrix $M(z) $ is a simple eigenvalue of $M(z) $. Moreover, the corresponding eigenvector $v_r $ has positive entries $v_r \in \mathbb R_+^{|X|}$. 
 
By the definition of $M(z) $, $\textbf{e}_{|X|}^T$ is a left-eigenvalue of $M(0)$, hence the matrix $M(0)$ has only one eigenvalue $\lambda=1 $, corresponding to the normalized eigenvector $v_0$. Since on the other hand $\rho(M(0))\leq \norm{M(0)}\leq 1 $, the Perron-Frobenius theorem implies that $\rho(M(0))=1 $. 

Moreover, the function $\mathbb R \rightarrow \mathbb R:z \mapsto \rho(M(z)) $ is strictly decreasing. Indeed, recall that $\rho(M(z)) = \lim_{k \rightarrow \infty} \| M(z)^k\|^{1/k}$, where $\|\cdot \|$ is the matrix norm induced by the standard euclidean norm. Since all the entries of the matrix $ M(z) $ are strictly decreasing as a function of $z\in\R$, which implies the claim.

Thus, we infer $\rho(M(z) ) <1 $ for $z >0 $ while $\rho(M(z)) >1 $ for $z < 0$. Furthermore, by the Perron-Frobenius theorem all eigenvectors with positive entries are associated to the eigenvalue given by the spectral radius $\rho(M(z))$. Consequently, $(0, v_0)$ is the unique real eigen-solution of \eqref{eq:non linear eigenproblem}.
\end{proof}

\begin{theorem}[Long-time behaviour] \label{thm:long time behaviour}
    Let the assumptions of Lemma \ref{lem:eigensolution} hold.
    Let $v_0 $ be as in Lemma \ref{lem:eigensolution}.
    Denote with $B \in L^1(\R_+;\mathbb R_+^{|X|})$ the solution of equation \eqref{eq:IntroRE one entrance}.  
    Then there exist three constants $c_0\geq 0 $, $ \varepsilon,\,  C >0$ such that 
    \[
        \| B(t) - c_0 v_0 \| \leq C e^{- \varepsilon t }
    \]
for all $ t >0 $. Here, $c_0$ depends only on the vector $\int_0^\infty B^0(s)\, ds $.
\end{theorem}
\begin{proof}

By the Laplace inversion formula, see \cite{diekmann2012delay}, applied to \eqref{eq: Lap B function of B_0} we have that 
\begin{equation}\label{Mellin Laplace inversion}
B(t)= \frac{1}{2 \pi i } \lim_{ T \rightarrow \infty } \int_{\gamma-iT}^{\gamma+iT} e^{z t } (I-M(z))^{-1} \hat{B}^0 (z) dz 
\end{equation}
for $\gamma > \sup \{ Re z : \det (I-M(z) )=0\} $. 

We now compute the integral on the right-hand side of equation \eqref{Mellin Laplace inversion}. 
To this end, we first prove that there exists an $\varepsilon>0$ such that, if $z \in \mathbb C$ satisfies $\det (I-M(z) )=0 $, then $\Re z \leq - \varepsilon $ or $z=0$.
Indeed assume that there exists a $\overline z \in  \mathbb R\setminus\{0\} $ such that $\det (I-M( i \overline z ) )=0$. In particular, $M( i \overline z )$ has one as eigenvalue. Note that the absolute value of any element of the matrix $M(i\overline z ) $ is smaller or equal than the corresponding element of $M(0) $. Consequently, by Wielandt's theorem (see \cite[Chapter 8.3]{carl2000matrix}) there is a diagonal unitary matrix $D$ such that $M(i \overline z)=DM(0) D^{-1}$. More precisely, the diagonal elements of $D$ have the form $D_{\alpha\alpha}=e^{i \mu_\alpha }$ for some $\mu_\alpha\in\R$. We infer $M(i \overline z )_{\alpha\beta}= e^{ i (\mu_\alpha- \mu_\beta ) } M(0)_{\alpha \beta } $. As a consequence we have 
 \[
\int_0^\infty \left( 1-  e^{ i ( \mu_\alpha- \mu_\beta -  \overline z t ) } \right) \Phi_{\alpha \beta }(t) dt =0. 
 \]
Since $\overline z\neq 0$ we obtain $\Phi_{\alpha \beta } (t)=0$  a.e., which contradicts the assumption that the matrix $M(z) $ is irreducible.
We deduce that $\det (I- M(z))=0$ implies $\Re z <0$ or $z=0$.

We now prove that there exists an $\varepsilon$ such that if $\det (I - M(z))=0$ then $\Re z \leq - \varepsilon$ or $z=0$. 
Thanks to the Riemann Lebesgue Lemma for every $x <0$ there exists a $\eta_0 >0$ such that the matrix $ (I-M(s+ i \eta )) $ is invertible for every $s \in [ x, 0] $ and $|\eta |> \eta_0 $. 
Moreover, since the function $ z \mapsto (I- M(z))^{-1}$ is meromorphic, we deduce that the number of poles contained in the compact set $\{ \lambda \in \mathbb C : |\Im \lambda | \leq \eta_0 \text{ and }  \Re \lambda \in  [  x , 0 ]\}$ is finite. 
Hence, we can choose $\varepsilon>0$ small enough so that $ z \mapsto (I - M(z))^{-1}$ has only one pole $z=0$ in the set $\{ \lambda \in \mathbb C : \Re \lambda > -\varepsilon  \}$.

Therefore we can consider any $\gamma >0 $ in \eqref{Mellin Laplace inversion}.
We deduce that
\begin{align*}
   \int_{\gamma-iT}^{\gamma+iT} e^{z t } (I-M(z))^{-1} \hat{B}^0 (z) dz = 
     \oint_{\Gamma} e^{z t } (I-M(z))^{-1} \hat{B}^0 (z) dz - \sum_{i=1}^3 \int_{\Gamma_i } e^{z t } (I-M(z))^{-1} \hat{B}^0 (z) dz 
\end{align*}
where $\Gamma:= \cup_{i =1}^4\Gamma_i  $ and where $\Gamma_4$ is the segment in the complex plane connecting $\gamma+i T $ to $\gamma-i T $, $\Gamma_2$ is the segment connecting $\gamma+i T $ to $-\varepsilon + i T $, 
$\Gamma_3$ is the segment connecting $-\varepsilon + i T $ to $-\varepsilon - i T $, finally $\Gamma_1$ connects $- \varepsilon -i T $ to  $\gamma -i T $. 

By the residue theorem we have 
\[
  \frac{1}{2 \pi i }    \oint_{\Gamma} e^{z t } (I-M(z))^{-1} \hat{B}^0 (z) dz = Res_{z=0} (I-M(z))^{-1} \hat{B}^0 (z). 
\]
We now compute $Res_{z=0} (I-M(z))^{-1} \hat{B}^0 (z)$. To this end, we notice that for small $z$ we have the following Laurent series representation of $(I - M (z))^{-1}$ 
\[
(I - M (z))^{-1} = \sum_{n=-p}^\infty z^n \mathcal R_n, 
\]
where $\mathcal R_n\in\R^{|X|\times|X|}$, $p\geq 1$. Since $(I - M (z))^{-1} (I - M (z))=(I - M (z)) (I - M (z))^{-1} = I $ we deduce that 
$ (I - M(z)) \mathcal R_{-1} =\textbf{0}$ which implies that $\mathcal R_{-1} v = \alpha v_0 $ for every vector $v$ and for some $\alpha \in \mathbb C$ that depends on $v$.
Thus $Res_{z=0}(I-M(z))^{-1} \hat{B}^0 (z) = \mathcal R_{-1} \hat{B}^0 (0)= c_0 v_0$ where the constant $c_0\geq0$ depends on $\hat{B}^0(0)= \int_0^\infty B^0(s)\, ds$. 
In addition, by the Riemann Lebesgue Lemma, we have that 
\[
\lim_{T \rightarrow \infty }  \frac{1}{2 \pi i }  \left\| \int_{\Gamma_i} e^{z t } (I-M(z))^{-1} \hat{B}^0 (z) dz \right\|  =0
\]
for $i=1,2$. Moreover 
\begin{align*}
\lim_{T \rightarrow \infty }  \frac{1}{2 \pi i }  \left\| \int_{\Gamma_3 } e^{z t } (I-M(z))^{-1} \hat{B}^0 (z) dz \right\| 
= \lim_{T \rightarrow \infty }  \frac{1}{2 \pi i }  \left\| \int_{- \varepsilon - i T }^{- \varepsilon + i T } e^{z t } (I-M(z))^{-1} \hat{B}^0 (z) dz \right\| 
\leq C e^{-\varepsilon t }. 
\end{align*}
The desired conclusion follows.

\end{proof}

\begin{lemma}
    Make the assumptions of Theorem \ref{thm:long time behaviour}. 
    Let $\mathbf{K} \in \mathbb R_+^{|X| \times |X|} $ be the diagonal matrix s.t. $\mathbf{K}_{\alpha \alpha } = k_\alpha$. 
      Denote with $D \in \mathbb R_+^{|X|}$ the function given by \eqref{eq:IntroConcentrationOneOutgoing}. 
    Assume that there exist $z_0,  C_1, C_2 >0$ such that for every $t>0$    \[
    k_\alpha \leq C_1 e^{- z_0 t } \text{ and } D^0_\alpha \leq C_2 e^{- z_0 t } 
    \]
    for every $\alpha \in X$. 
    Then there exist two constants  $C, \rho >0$ such that
    \begin{align*}
        \left\| D(t) - c_0  v_0 \right\| \leq Ce^{-\rho t }
    \end{align*}
    for every $t >0$, where the constant $c_0 $ is the same as in Theorem \ref{thm:long time behaviour}. 
\end{lemma}
\begin{proof}
    First of all, note that $\int_0^\infty k_\alpha (t ) dt =1$ for every $\alpha \in X$ since $\{\Phi\}$ is irreducible. Then, by definition of $D(t) $ and the bounds on $D^0$, $k_\alpha $ we have 
    \begin{align*}
     &   \left\| D (t) - c_0 v_0   \right\|= \left\| D^0 (t) +  \int_0^t \textbf{K} (s) B(t-s) ds  - c_0 \int_0^t \textbf{K}(s) v_0ds  - c_0  \int_t^\infty  \textbf{K}(s) v_0ds  \right\|\\
        & \leq C  e^{- z_0 t }+  \left\|\int_0^t \textbf{K}(s) \left[ B(t-s) - c_0 v_0\right]  ds  \right\|   
        \leq C e^{- z_0 t }+ \int_0^t  \left\|\textbf{K}(s)  \right\| \left\| B(t-s) - c_0 v_0\right\|  ds \\
        &\leq C  e^{- z_0 t }+ C  e^{- \varepsilon t }\int_0^t  \left\|\textbf{K}(s)  \right\|  e^{ \varepsilon s } ds   \leq C  e^{- z_0 t }+ C  e^{- \varepsilon t }\int_0^t e^{ ( \varepsilon- z_0) s } ds \leq  C  e^{- \rho t }
    \end{align*}
    for some $\rho>0$.
    Notice that for the above computation we have used the long-time behaviour of $B$ computed in Theorem \ref{thm:long time behaviour} and the bounds on $k_\alpha  $ and $D^0_\alpha $. 
\end{proof}

\begin{theorem}
  Make the assumptions of Theorem \ref{thm:long time behaviour}. 
    Let $\mathbf{K} \in \mathbb R_+^{|X|} $ be the diagonal matrix s.t. $\mathbf{K}_{\alpha \alpha } = k_\alpha$. 
    Assume in addition that $N^0= \int_0^\infty D^0(s) ds $. 
    Then the solution $N (t) \in \mathbb R_+^{|X|}$ of \eqref{eq:IntroConcentrationOneEntrance} is such that there exists two positive constants $r$ and $C $ such that for all $t>0$
\begin{equation} \label{asympt N}
 \left\| N(t) - c_0 v_0 \int_0^\infty s \mathbf{K}(s) ds  \right\| \leq C e^{-rt },
\end{equation}
where $c_0$ is the same constant as in Theorem \ref{thm:long time behaviour}. 
\end{theorem}
\begin{proof}
    By definition of $N=(N_\alpha)_{\alpha \in X} $ we have that 
    \begin{align*}
        N(t)=N^0 + \int_0^t \left[ B(s) - D(s) \right] ds. 
    \end{align*}
Then  
\begin{align*}
   & \left\| N(t) - c_0 v_0 \int_0^\infty v \mathbf{K} (v) dv   \right\| = \left\| N^0 +  \int_0^t \left[ B(s) - D(s) \right] ds  - c_0 v_0 \int_0^\infty v \mathbf{K} (v) dv   \right\| \\
    &  =  \left\| N^0 +  \int_0^t \left[ B(s) - D^0(s) - \int_0^s \mathbf{K}(s-v ) B(v) dv  \right] ds  - c_0 v_0 \int_0^\infty v \mathbf{K} (v) dv   \right\| \\
    &
    \leq \left\| N^0 - \int_0^t D^0(s) ds  \right\| +  \left\| \int_0^t \left[ B(s)  - \int_0^s \mathbf{K}(s-v ) B(v) dv  \right] ds  -c_0 v_0 \int_0^\infty v \mathbf{K} (v) dv   \right\|.
\end{align*}
We know that, for some positive constants $C_1, \, C_2, \, C_3 $ and $z_0,\,  \varepsilon>0$
\[
  \left\| N^0 - \int_0^t D^0(s) ds  \right\| \leq C_1 e^{-z_0 t }, \  \left\| B(t) - c_0 v_0 \right\| \leq C_2 e^{-\varepsilon t }, \ \left\| v_0 \int_t^\infty \int_s^\infty \textbf{K} (v) dv ds \right\|  \leq C_3 e^{-z_0 t}. 
\]
Then, we conclude
\begin{align*}
 &    \left\| N^0 - \int_0^t D^0(s) ds  \right\| +  \left\| \int_0^t \left[ B(s)  - \int_0^s \mathbf{K}(s-v ) B(v) dv  \right] ds  - c_0v_0 \int_0^\infty v  \mathbf{K}(v) dv   \right\| \\
     &\leq  (C_1+C_3)e^{-z_0 t } +  \left\| \int_0^t  B(s)\int_{t-s}^\infty  \mathbf{K}(v )  dv ds  -c_0  v_0 \int_0^t \int_{t-s}^\infty  \mathbf{K} (v) dv ds  \right\| \\
     &\leq Ce^{-z_0 t } +  \left\| \int_0^{t} \left[  B(s)  -c_0 v_0 \right]\int_{t-s}^\infty   \mathbf{K}(v )  dv ds  \right\| \leq C e^{-z_0 t } +  \int_0^{t} \left\|  B(s)  -c_0 v_0 \right\| \left\| \int_{t-s}^\infty   \mathbf{K}(v )  dv   \right\| ds  \\
     & \leq  C e^{-z_0 t } + C   \int_0^{t} e^{-\varepsilon s} e^{(s-t) z_0 } ds  \leq C e^{-rt} 
\end{align*}
for some $r,\,  C>0$. 
\end{proof}
  
\section{Response functions and structured population equations} \label{sec:PDE}
In this section, we study the relation between the formalism of response functions introduced in Section \ref{sec:ReformODEasRFE} and the formalism of structured population models, which is often used in biology, see for instance \cite{perthame2006transport}.  
 In Section \ref{sec:PDEclassical} we study the relation between \eqref{eq:IntroConcentrationOneEntrance},  \eqref{eq:IntroRE one entrance}, \eqref{eq:IntroConcentrationOneOutgoing} and \eqref{eq:IntrodModelPDE}, while in Section \ref{sec:PDEgeneralized} we consider the equivalence between \eqref{eq:Concentration}-\eqref{eq:REInflux}-\eqref{eq:REOutflux} and generalizations of \eqref{eq:IntrodModelPDE}. 
 We notice that in order to obtain this equivalence it is crucial to study the properties of the forcing function $\{ B_\alpha ^0 \} $ which describes the transient response of the system, (cf. \ref{PDE RE:forcing functions}). 
Finally, in Section \ref{sec:ODE SPEs} we describe the conditions on the initial data $\{ n_\alpha^0\} $ that allow to rewrite the ODE model \eqref{eq:PrelDecompODE} as SPEs. 

\subsection{Classical RFEs} \label{sec:PDEclassical}
Here, we reformulate  \eqref{eq:IntroConcentrationOneEntrance},  \eqref{eq:IntroRE one entrance}, \eqref{eq:IntroConcentrationOneOutgoing} as a structured population model and we study the equivalence of the two formulations. 
To this end, we introduce the concept of age of an individual in some compartment $\alpha$. 
In this section, we assume that
every element enters the compartment $\alpha $ with the same age $\xi=0$.

Let $f_\alpha(t, \xi) $ be the number of elements in the compartment $\alpha$ with age $\xi$ at time $t$. Since here we assume that all the elements enter the compartment $\alpha $ with age $\xi=0$, the age of an element in the compartment $\alpha$ at a certain time $t \geq0 $ is just $\xi=t-\overline t$, where $\overline t $ is the time at which it entered in $\alpha$.

We then have the following evolution equation for $f_\alpha $
\begin{align}\label{eq:fl}
   & \partial_t f_\alpha (t,\xi) + \partial_\xi f_\alpha (t,\xi) = -   \Lambda_\alpha(\xi) f_\alpha (t,\xi) , \\
    & f_\alpha(t,0) = \sum_{\beta \in X \setminus \{\alpha\} } \int_{\R_+} \lambda_{ \beta \alpha } (\eta) f_\beta(t,\eta) d\eta \label{eq:f2}\\
    & f_\alpha (0,\xi)=m_\alpha (-\xi) e^{- \int_0^{\xi} \Lambda_\alpha(s) ds } \label{eq:f3}
\end{align}
where $\Lambda_\alpha(\xi) = \sum_{\beta \in X \setminus \{\alpha \} }\lambda_{\alpha\beta}(\xi) $.
Indeed, $f_\alpha(t, \xi) $ changes in time due to the aging of the elements, described by the transport term in \eqref{eq:fl}, due to jumps from the compartment $\alpha $ to any other compartment $\beta \in X$, described by the loss term in \eqref{eq:fl}, and, finally, due to the fact that elements enter the compartment $\alpha$ with age zero, from any other compartment $\beta \in X$. 
Accordingly, $\lambda_{\alpha \beta} (\eta)$ in the above equation is the rate at which an individual, which has been in the compartment $\alpha $ for a time interval of length $\eta$, jumps from compartment $\alpha $ to compartment $\beta $. Finally, $m_\alpha (-\xi)$ in \eqref{eq:f3} is the number of elements with state in the compartment $\alpha $ at time $-\xi$. We then obtain $f_\alpha (0, \xi) $, multiplying $m_\alpha (\xi) $ by the probability that these elements stay in the compartment $\alpha $ up to time $0$, namely multiplying by $e^{-\int_0^\xi \Lambda_\alpha (s) ds } $.

\begin{definition}  \label{def:measure solution age=0} 
Let $\lambda_{\alpha \beta } \in C_b(\mathbb R_+ )$ be non-negative for every $\alpha, \beta \in X$.
Assume that $f_\alpha^0 \in \mathcal M_{+, b} (\mathbb R_+)$ for every $\alpha \in X$. 
A family of functions $\{ f_\alpha \}$ with $f_\alpha \in C([0, T] ; \mathcal M_{+,b}(\mathbb R_+) ) $ is a solution of equation \eqref{eq:fl}, with initial condition $f_\alpha^0(\cdot)$, if for every $\varphi \in C^1_c(\mathbb R_+) $  the map $t \mapsto \int_{\mathbb R_+} \varphi(\xi) f_\alpha(t, d \xi)  $ is differentiable and 
\begin{align} \label{eq:weak formPDE}
 \frac{d}{dt}   \int_{\mathbb R_+} \varphi(\xi) f_\alpha(t, d \xi) &=    \int_{\mathbb R_+} \left[ \varphi'(\xi) - \Lambda_\alpha (\xi) \varphi(\xi)  ) \right] f_\alpha(t , d \xi)  + \varphi(0) \sum_{\beta \in X \setminus \{ \alpha\} } \int_{\mathbb R_+} \lambda_{\beta \alpha} (\eta) f_\beta(t, d\eta) 
\end{align}
for every $\alpha \in X$. 
Furthermore $f_\alpha(0, \cdot)=f_{\alpha}^0(\cdot)$. 
\end{definition}

We recall that in the definition of $C([0, T], \mathcal M_{+, b} (\mathbb R_+))$  
we endow the space $\mathcal M_{+, b} (\mathbb R_+)$ with the Wasserstein distance, see Section \ref{sec:notation}.  
\begin{lemma}\label{lem:sol PDE}
Let $\lambda_{\alpha \beta } \in C_b(\mathbb R_+ )$ to be non-negative for every $\alpha , \beta \in X$. 
Let $m_\alpha \in \mathcal M_{+, b} (\mathbb R_+)$ for $\alpha \in X$.  
    Assume that $f_\alpha^0$ is given by \eqref{eq:f3}. 
    Then there exists a unique solution $\{ f_\alpha \} $ with $f_\alpha \in C([0, T], \mathcal M_{+, b}(\mathbb R_+) ) $ for all $\alpha \in X$ to equation \eqref{eq:fl} in the sense of Definition \ref{def:measure solution age=0}. 
    The solution satisfies
    \begin{equation} \label{exponential bound f_alpha}
\int_{\mathbb R_+} e^{ \int_0^{\xi} \Lambda_\alpha (v) d v } f_\alpha (t, d \xi) < \infty, 
    \end{equation}
   for every $\alpha \in X$ and every $t \geq 0$, where $\Lambda_\alpha (t):= \sum_{\beta \in X \setminus \{\alpha\} } \lambda_{\alpha \beta} (t)$. 
\end{lemma}
\begin{proof}
 We can rewrite equation \eqref{eq:weak formPDE} in fixed point form as follows. 
We use the notation $Y:=C([0, T], \mathcal M_{+,b}(\mathbb R_{+} ) )^{|X| } $ where $Y$ is endowed with the metric induced by the distance 
\[
d_T(f,g)= \sum_{\alpha \in X} \sup_{t \in [0, T] } W_1( f_\alpha (t, \cdot) g_\alpha (t, \cdot) ).
\]

Given a $f=( f_\alpha ) \in Y$ we define the operator $\mathcal T [f](t) : C_c(\mathbb R_+) \rightarrow \mathbb R_+^{|X|} $ as 
\begin{align} 
&\langle \mathcal T [f] (t) , \varphi \rangle_\alpha  := \int_0^t  \varphi(t- \xi)  e^{- \int_0^{t- \xi} \Lambda_\alpha (v) dv } \sum_{\beta \in X \setminus \{ \alpha\} } \int_{\mathbb R_+} \lambda_{\beta \alpha } (\eta) f_\beta (\xi, d \eta ) d \xi \nonumber \\
&
+  \int_{\mathbb R_+} f_\alpha^0(d\xi) \varphi(\xi+t) e^{- \int_\xi^{t+\xi} \Lambda_\alpha (v) dv }  \label{eq:fixed point}  
\end{align} 
where we are using the notation $\langle \cdot, \cdot \rangle_\alpha $ to indicate the $\alpha$-th component of $\mathcal T[f] (t) $ applied to $\varphi$. 

We define
$\| f \|_Y :=\sum_{\alpha \in X} \sup_{t \in [0, T]} \| f_\alpha (t, \cdot ) \|_{TV}$
and the set $\mathcal X_T \subset Y $ as 
\[
\mathcal X_T := \left\{  f \in Y : (f(0))_\alpha=f_\alpha^0 \  \text{ for every } \alpha \in X, \ ||f ||_Y \leq 1 + \sum_{\alpha \in X} \| f^0_\alpha \|_{TV} \right\}. 
\]
For every $f \in \mathcal X_T$ each component of the operator $\mathcal T[f](t) $ is a linear, positive and continuous operator from $C_c(\mathbb R_+) $ to $\mathbb R$ and, hence, can be identified with an element of $\mathcal M_{+, b} (\mathbb R_+).$

Using the bound for the parameters $\lambda_{\alpha \beta } $, we deduce that, for every $f \in \mathcal X_T$, we have that $\mathcal T [f] \in \mathcal X_T$ for sufficiently small values of $T$. 
Similarly, for sufficiently small values of $T$, the operator $\mathcal T $ is a contraction and hence Banach fixed point Theorem implies that there exists a unique fixed point $f \in \mathcal X_T$. 

Since the fixed point $f$ of $\mathcal T $ satisfies $\mathcal T [f](t) = f (t) $ the map $ t \mapsto \langle f(t, \cdot) , \varphi \rangle_\alpha =\langle \mathcal T[f] (t), \varphi \rangle_\alpha  $ is differentiable if $\varphi \in C^1_c(\mathbb R_+)$. 
Hence, the fixed point of $\mathcal T $ is a solution to \eqref{eq:weak formPDE} for $T$ small enough. 
By linearity $\| f \|_Y$ stays finite for every $t>0.$
We can then iterate the argument and prove the existence and uniqueness of a solution to \eqref{eq:weak formPDE} for an arbitrary $T>0$. 

To prove \eqref{exponential bound f_alpha} we consider the test function $ \varphi_{\varepsilon, R}(\xi)= \chi_{R, \varepsilon } (\xi)e^{\int_0^\xi \Lambda_\alpha (v) dv}  $ in equation \eqref{eq:fixed point}. 
Here $\chi_{R, \varepsilon}$ is a smooth decreasing function such that $\chi_{R, \varepsilon} (\xi) = 0 $ for $\xi > R+ \varepsilon$ and  $\chi_{R, \varepsilon} (\xi) = 1 $ for $\xi \leq R $.  
We deduce that 
\begin{align*}
 &\int_{\mathbb R_+} \chi_{R,\varepsilon}(\xi)  e^{ \int_0^\xi \Lambda_\alpha (v) dv } f_\alpha(t, d\xi)   \leq  \int_0^t  \chi_{R,\varepsilon}(t- \xi)  \sum_{\beta \in X \setminus \{ \alpha\} } \int_{\mathbb R_+} \lambda_{\beta \alpha } (\eta) f_\beta (\xi, d \eta ) d \xi \\
 &+ c_1 f^0_\alpha (\mathbb R_+)  \leq C \sup_{t \in [0, T] } \sum_{ \beta \in X \setminus \{ \alpha \}} f_\beta (t, \mathbb R_+)   + c_1 f^0_\alpha (\mathbb R_+) 
\end{align*}
where the positive constants $c_1, C$ do not depend on $\varepsilon$ and $R$. 
Sending $\varepsilon$ to zero and $R$ to infinity concludes the proof. 
\end{proof}
Since we are assuming that each of the compartments $\alpha \in X $ has at most one entrance point, we can consider the simplified system for $\{N_\alpha, B_\alpha, D_\alpha\} $, cf. \eqref{eq:IntroConcentrationOneEntrance},  \eqref{eq:IntroRE one entrance}, \eqref{eq:IntroConcentrationOneOutgoing}. 
In the following two theorems, we study the equivalence between \eqref{eq:fl}, \eqref{eq:f2}, \eqref{eq:f3} and \eqref{eq:IntroConcentrationOneEntrance},  \eqref{eq:IntroRE one entrance}, \eqref{eq:IntroConcentrationOneOutgoing}.

\begin{theorem}[RF and SP] \label{thm:RF to SP}
Under the assumptions of Lemma \ref{lem:sol PDE} consider the solution $\{ f_\alpha \} $ with $f_\alpha \in C([0, T] , \mathcal M_{+, b}(\mathbb R_{+}))$ to  \eqref{eq:fl}, \eqref{eq:f2}, \eqref{eq:f3}. 
    Then, the family of functions $\{ N_\alpha \}$, defined by
        \begin{align} \label{n function of f}
   &  N_\alpha (t)= \int_{\mathbb R_+}f_\alpha(t,d \xi), \quad t\geq 0 \quad \forall \alpha \in X,
    \end{align}
    satisfies \eqref{eq:Concentration}  where $\{B_\alpha, D_\alpha \}$ satisfy  \eqref{eq:IntroRE one entrance}, \eqref{eq:IntroConcentrationOneOutgoing} with response functions $\{\Phi_{\alpha \beta } \}$ and $\{k_\alpha \}$ given by
    \begin{align} \label{psi as a function of lambda}
    \Phi_{\alpha \beta } (t)= \lambda_{\alpha \beta}(t)e^{- \int_0^t \Lambda_\alpha(s) ds }   , \quad k_\alpha (t)= \Lambda_\alpha (t)e^{- \int_0^t \Lambda_\alpha(s) ds  } 
    \end{align}
and with 
      \begin{equation} \label{PDE RE:forcing functions}
     B^0_\alpha(t) = \sum_{\beta \in X \setminus \{ \alpha \} }   \int_{\mathbb R_- }  \Phi_{ \beta \alpha }(t-\xi) m_\beta ( d\xi),  \quad   D^0_\alpha(t) =  \int_{\mathbb R_-} k_\alpha (t-\xi)  m_\alpha ( d\xi). 
    \end{equation} 
    \end{theorem} 

%\begin{theorem} 
    % Vice versa, given a kernel $\Phi \in L^1_{loc} (\mathbb R_+, \mathbb R_+^{|\Omega| \times |\Omega |})$ and the family of solutions $\{ N_\alpha \}_{\alpha \in X} $ to \eqref{eq:ConcentrationOneEntrance} with respect to the kernel $\Phi$, then the solutions $\{f_\alpha\}_{\alpha \in X} $ to \eqref{eq:fl}, \eqref{eq:f2}, \eqref{eq:f3} with parameters given by
% \begin{align*}
%        & \lambda_{ij}=\chi_{\{\sum_j \int_{t-\xi}^t \psi_{ij}(\tau,t-\xi)d\tau <1 \}} (t,\xi) \dfrac{\psi_{ij}(t,t-\xi)}{1-\sum_j \int_{t-\xi}^t \psi_{ij}(\tau,t-\xi)d\tau},  \\
    %\end{align*}
% satisfy
  %  \begin{align} \label{n function of f}
%      &  N_\alpha (t)= \int_0^\infty f_\alpha(t,\xi) d\xi, \quad t\geq 0 \quad \forall \alpha \in X. 
%    \end{align}
%\end{theorem}
\begin{proof}
From equation \eqref{eq:weak formPDE} we deduce that 
\begin{align*}
  \frac{d}{dt}  N_\alpha (t)=
  \frac{d}{dt} \int_{\mathbb R_+} f_\alpha (t, d\xi) =   - \int_{\mathbb R_+} \Lambda_\alpha (\xi ) f_\alpha(t, d\xi) + \sum_{\beta \in X \setminus \{\alpha \} } \int_{\mathbb R_+ } \lambda_{\beta \alpha } (\xi) f_\beta ( t, d \xi). 
\end{align*}
Notice that we can use the test function $\varphi =1 $ due to \eqref{exponential bound f_alpha}. 

The fixed point formulation of equation \eqref{eq:weak formPDE}, namely \eqref{eq:fixed point}, implies that
\begin{align} \label{eq:proof PDE then RE}
    &\int_{\mathbb R_+} \lambda_{\beta \alpha } (\xi)  f_\beta (t, d\xi) = \int_0^t \lambda_{\beta \alpha }(t-\xi) e^{- \int_0^{t-\xi} \Lambda_\beta (v) dv } \sum_{\gamma \in X \setminus \{ \beta \} } \int_{\mathbb R_+} \lambda_{\gamma \beta } (\eta) f_\gamma(\xi, d \eta)  d \xi \\
    & + \int_{\mathbb R_-}\lambda_{\beta \alpha } (t-\xi) e^{- \int_{0}^{t-\xi} \Lambda_\beta (v) dv} m_\beta (d \xi) \nonumber. 
\end{align}
Summing in the above equation over $\beta \in X \setminus \{ \alpha \} $ we deduce that 
\[\sum_{\beta \in X \setminus \{\alpha\} } \int_{\mathbb R_+ } \lambda_{\beta \alpha } (\xi) f_\beta ( t, d \xi)=:S_\alpha(t)
\]satisfies \eqref{eq:IntroRE one entrance} with kernel $\Phi_{\alpha \beta } $ as in \eqref{psi as a function of lambda} and with the forcing function $B^0_\alpha$ given by \eqref{PDE RE:forcing functions}.  
Similarly, taking the sum over $\alpha \in X \setminus \{ \beta \} $ we deduce that  
\[
\int_{\mathbb R_+} \Lambda_\alpha (\xi ) f_\alpha (t, d\xi) = \int_0^t k_\alpha ( t-\xi )  \sum_{\gamma \in X \setminus \{\alpha\} } \int_{\mathbb R_+} \lambda_{\gamma \alpha } (\eta) f_\gamma(\xi, d \eta) d\xi + D^0_\alpha (t), 
\]
where $k_\alpha $ is as in \eqref{psi as a function of lambda}. 
Hence, 
\[ 
D_\beta^0 (t) = \int_{\mathbb R_+} \Lambda_\beta (\xi) f_\beta(t, d\xi) 
\]
satisfies \eqref{eq:IntroConcentrationOneOutgoing}. This concludes the proof. 
\end{proof}

%\eu{Assume that all the compartments $\alpha \in X$ have only one exit point.}
%Then the PDE for $f_\alpha $ can be reformulated as follows
%\begin{align}\label{eq:fl one exit}
 %  & \partial_t f_\alpha (t,\xi) + \partial_\xi f_\alpha (t,\xi) = -   \Lambda_\alpha(\xi) f_\alpha (t,\xi) , \\
%    & f_\alpha(t,0) = \sum_{\beta \in X \setminus \alpha } \int_0^\infty  \lambda_{\beta \alpha} (\eta)f_\beta(t,\eta) d\eta \label{eq:f2 one exit}\\
  %  & f_\alpha (0,\xi)=m_\alpha (-\xi) e^{- \int_0^\xi \Lambda_\alpha (s) ds} \label{eq:f3 one exit}
% \end{align}
%where $\Lambda_\alpha (\xi)= \sum_{\beta \in X \setminus \alpha }  \lambda_{\alpha \beta } (\xi) $ and $\lambda_{\alpha \beta }(\xi) = \lambda_{ \overline j_\alpha \overline i_\beta } (\xi) $, where $\overline j_\alpha $ is the exit point from $\alpha $ and  $\overline i_\beta $ is the entrance point in $\beta$. 

\begin{theorem}[RF and SP] \label{thm:SP to RF}
Let $\{N_\alpha , B_\alpha, D_\alpha \} $ satisfy \eqref{eq:IntroConcentrationOneEntrance}, \eqref{eq:IntroRE one entrance}, \eqref{eq:IntroConcentrationOneOutgoing} with response functions  $\Phi_{\alpha \beta}  \in C(\mathbb R_+; \mathbb R_+)$ and $k_\alpha  \in C(\mathbb R_+; \mathbb R_+)$ assume that $\{B_\alpha^0, D_\alpha^0 \}$ satisfy \eqref{PDE RE:forcing functions} where $\{ m_\alpha\} $ is a family of measures $m_\alpha \in \mathcal M_{+, b}(\mathbb R_+)$. 
Assume in addition that for every $\alpha,  \beta \in X $ it holds that the function 
\begin{equation}\label{lambda as a function of Phi}
\lambda_{\alpha \beta }(t):= \frac{\Phi_{\alpha \beta }(t) }{1- \sum_{\gamma \in X \setminus \{ \alpha\} } \int_0^t \Phi_{\alpha \gamma }(s) ds } \quad t \geq 0
\end{equation}
is continuous and bounded for every $\alpha, \beta $. 
Then the family of functions $\{ f_\alpha \} $ solving \eqref{eq:fl}, \eqref{eq:f2}, \eqref{eq:f3} satisfy \eqref{n function of f}.
    \end{theorem} 
    Before proving the theorem we remark that the continuity of $\lambda_{\alpha \beta } $ as in  \eqref{lambda as a function of Phi} holds if we assume that for every $\alpha ,\beta \in X $ the support of the function $\Phi_{\alpha \beta } $ is unbounded. Notice that this assumption is true for all the examples in Section \ref{sec:Examples}. 
    \begin{proof}
    Since $\{f_\alpha\}$ solves \eqref{eq:fl}, \eqref{eq:f2}, \eqref{eq:f3} by Theorem \ref{thm:RF to SP} the functions $\bar{N}_\alpha (t):= \int_{\R_+}f_\alpha(t,d\xi)$ solves \eqref{eq:IntroConcentrationOneEntrance}, \eqref{eq:IntroRE one entrance}, \eqref{eq:IntroConcentrationOneOutgoing} with the same initial conditions as $\{N_\alpha,B_\alpha,D_\alpha\}$ and some response functions $\{\bar{\Phi}_{\alpha\beta}\}$ given by  \eqref{psi as a function of lambda}. Since $\Lambda_\alpha$ is given by 
    \[
        \Lambda_\alpha (t)=  \frac{ \sum_{\beta \in X \setminus \{\alpha\}  }\Phi_{\alpha \beta }(t) }{1- \sum_{\gamma \in X \setminus \{\alpha\} } \int_0^t \Phi_{\alpha \gamma }(s) ds },
    \] 
    then we have that
\begin{align*}
\int_0^t \Lambda_\alpha (v) dv = \int_0^t \frac{ \sum_{\beta \in X \setminus \{\alpha\} }\Phi_{\alpha \beta }(v) }{1- \sum_{\gamma \in X \setminus \{\alpha\}  } \int_0^v \Phi_{\alpha \gamma }(s) ds }  dv= -  \ln \left( 1- \int_0^t \sum_{\beta \in X \setminus \{\alpha\} }\Phi_{\alpha \beta }(s) ds \right). 
\end{align*}
Hence,
\begin{align} \label{formula exp Lambda}
    e^{- \int_0^t \Lambda_\alpha (v) dv }  = 1- \int_0^t \sum_{\beta \in X \setminus \{\alpha\} }\Phi_{\alpha \beta }(s) ds.
\end{align} 
     By \eqref{psi as a function of lambda} we deduce that $\bar \Phi_{\alpha \beta } = \Phi_{\alpha \beta }$. 
    By uniqueness of the solution, we obtain $N_\alpha=\bar{N}_\alpha$. This concludes the proof.
    \end{proof}

\subsection{Generalized RFEs} \label{sec:PDEgeneralized}
Here, under suitable conditions on the forcing functions, we reformulate \eqref{eq:Concentration}, \eqref{eq:REInflux}, \eqref{eq:REOutflux} as a structured population model.
As in Section \ref{sec:PDEclassical}, we assume that the elements have age $\xi=0$ when they enter in the compartment. 
We assume that $ f_{\alpha j} (t, \xi) $ is the density of individuals that entered the compartment $\alpha $ with state $j \in\alpha $, and that have age $\xi $ at time $t$. 

 The evolution in time of $f_{\alpha j } $ is described by the following SPEs 
\begin{align}
  &  \partial_t f_{\alpha j } (t,\xi ) + \partial_\xi f_{\alpha j } (t, \xi )  = - M_{\alpha j}(\xi)  f_{\alpha j }  (t, \xi) \label{eq:gen1} \\
   & f_{\alpha j } ( t, 0)= \sum_{\beta \in X \setminus \{\alpha \} } \sum_{ k \in \beta }  \int_{\R_+} f_{\beta k } (t , \eta ) \mu_{kj}(\eta) d \eta \label{eq:gen2}\\
   & f_{\alpha j } ( 0, \xi ) = m_{\alpha j } (- \xi) e^{- \int_0^\xi M_{\alpha j } (v) dv } \label{eq:gen3}
\end{align}
where $M_{\alpha j} (\xi)= \sum_{\beta \in X \setminus \{\alpha \} } \sum_{\ell \in \beta } \mu_{j \ell } (\xi )$.

The transport term in \eqref{eq:gen1} is due to the aging of the elements in the compartments. Moreover, elements with "state-at-entrance" $j \in \alpha $ and age $\xi$ jump to another compartment with rate $M_{\alpha j } (\xi) $. 
The birth term in \eqref{eq:gen2} is due to elements that entered $\beta $ with any state and at any time in the past, that at time $t$ jump to state $j \in \alpha $. 
Finally $m_\alpha(-\xi)$ in \eqref{eq:gen1} is a vector whose $j$-th element is the density of elements with state $j \in \alpha $ at time $-\xi $.

Equation \eqref{eq:gen1}, \eqref{eq:gen2}, \eqref{eq:gen3} is a special case of \eqref{eq:fl}, \eqref{eq:f2}, \eqref{eq:f3}. 
Indeed, consider in \eqref{eq:fl}, \eqref{eq:f2}, \eqref{eq:f3} a set of compartments $X_1$ such that every compartment in $X_1$ is of the form $\{ j \} $ with $j \in \Omega $. Moreover, we assume that the compartments $\{ j\} $ with $ j \in \alpha$ with $\alpha \in X$, are not connected, which means that $\lambda_{ij} =0$ when $i, j \in \alpha $.
Then \eqref{eq:fl}, \eqref{eq:f2}, \eqref{eq:f3} for the set of compartments $X_1$ reduce to \eqref{eq:gen1}, \eqref{eq:gen2}, \eqref{eq:gen3} for the set of compartments $X$. 

We therefore define a solution of \eqref{eq:gen1}, \eqref{eq:gen2}, \eqref{eq:gen3} in analogy with Definition \ref{def:measure solution age=0}. 
\begin{definition}  \label{def:measure solution multiple entrance points} 
Let $\mu_{ij } \in C_b(\mathbb R_+ )$ be non-negative for every $i \in \alpha $, $ j \in \beta$ with $\alpha, \beta  \in X$.
Assume that $f_{\alpha j}^0 \in \mathcal M_{+, b} (\mathbb R_+)$ for every $j \in \alpha $ and every $\alpha \in X$. 
A family of functions $\{ f_{\alpha j} \}$ with $f_{\alpha j} \in C([0, T] ; \mathcal M_{+,b}(\mathbb R_+) ) $ is a solution of equation \eqref{eq:gen1}, with initial condition $f_{\alpha j } ^0(\cdot)$, if for every $\varphi \in C^1_c(\mathbb R_+) $  the map $t \mapsto \int_{\mathbb R_+} \varphi(\xi) f_{\alpha j} (t, d \xi)  $ is differentiable and 
\begin{align} \label{eq:weak formPDE gen}
 \frac{d}{dt}   \int_{\mathbb R_+} \varphi(\xi) f_{\alpha j } (t, d \xi) =&    \int_{\mathbb R_+} \left[ \varphi'(\xi) - M_{\alpha j }  (\xi) \varphi(\xi)  ) \right] f_{\alpha j } (t , d \xi) \\
 &+ \varphi(0) \sum_{\beta \in X \setminus \{ \alpha\} } \sum_{k \in \beta } \int_{\mathbb R_+} \mu _{k j } (\eta) f_{\beta k } (t, d\eta) \nonumber
\end{align}
for every $j \in \alpha $ with $\alpha \in X$. 
Furthermore $f_{\alpha j } (0, \cdot)=f_{\alpha j }^0(\cdot)$. 
\end{definition}

Moreover, as a consequence of Lemma \ref{lem:sol PDE} we have the existence of a unique solution for \eqref{eq:gen1}, \eqref{eq:gen2}, \eqref{eq:gen3}. 
\begin{lemma}\label{lem:gsol PDE}
Let $\mu_{ij } \in C_b(\mathbb R_+ )$ to be non-negative for every $i, j$. 
Let $m_{\alpha j }  \in \mathcal M_{+, b} (\mathbb R_+)$ for $\alpha \in X$.  
    Assume that $f_{\alpha j}^0$ is given by \eqref{eq:f3}. 
    Then there exists a unique solution $\{ f_{\alpha j } \} $ with $f_{ \alpha j }  \in C([0, T], \mathcal M_{+, b}(\mathbb R_+) ) $ for all $ j \in \alpha $ and $\alpha \in X$ to equation \eqref{eq:fl} in the sense of Definition \ref{def:measure solution multiple entrance points}. 
    The solution satisfies, for every $j \in \alpha $ and $\alpha \in X$ and every $t \geq 0$
    \begin{equation} \label{exponential bound f_alpha general}
\int_{\mathbb R_+} e^{ \int_0^{\xi} M_{ \alpha j }  (v) d v } f_{\alpha j }  (t, d \xi) < \infty, 
    \end{equation}
    where $M_{\alpha j }  (t):= \sum_{\beta \in X \setminus \{\alpha\} } \sum_{i \in \beta }  \mu_{ i j } (t)$. 
\end{lemma}

In the following two theorems, we study the equivalence between \eqref{eq:gen1}, \eqref{eq:gen2}, \eqref{eq:gen3} and \eqref{eq:Concentration}, \eqref{eq:REInflux}, \eqref{eq:REOutflux}. 

\begin{theorem}[RF and SP] \label{thm:RF to SP gen}
Under the assumptions of Lemma \ref{lem:gsol PDE} consider the solution $\{ f_{\alpha j }  \} $ with $f_{\alpha, j }  \in C([0, T] , \mathcal M_{+, b}(\mathbb R_{+}))$ to  \eqref{eq:gen1}, \eqref{eq:gen2}, \eqref{eq:gen3}. 
    Then, the family of functions $\{ N_\alpha \}$, defined by
        \begin{align} \label{n function of f general}
   &  N_\alpha (t)= \sum_{j \in \alpha } \int_{\mathbb R_+}f_{\alpha j } (t,d \xi), \quad t\geq 0, \quad \forall \alpha \in X,
    \end{align}
    satisfies \eqref{eq:Concentration}. The corresponding fluxes $ \{ S_\alpha \}$ and $ \{J_\alpha \} $ satisfy \eqref{eq:REInflux} and \eqref{eq:REOutflux} with response functions given by 
    \begin{equation}\label{response function gen}
   ( G_{\alpha \beta }(t) )_{j k }= \mu_{kj } (t) e^{-\int_0^t M_{\alpha k } (v) dv }
    \end{equation}
    for $k \in \alpha , j \in \beta$ and
    \begin{equation} \label{response function gen out}
   (  K_\alpha (t) )_{ii } = M_{\alpha i } (t) e^{- \int_0^t M_{\alpha i } (s ) ds }, \quad    (  K_\alpha (t) )_{ij } = 0
    \end{equation}
    for $i, \, j \in \alpha $. The corresponding forcing functions have the form
    \begin{equation} \label{forcing function gen}
    S^0_\alpha (t)= \sum_{\beta \in X \setminus \{\alpha  \}} \int_{\mathbb R_-} G_{\beta \alpha } (t-s) m_\beta (d s ), \quad J^0_\alpha (t)=  \int_{\mathbb R_-} K_{ \alpha } (t-s) m_\alpha  (d s ).
    \end{equation}
    Here, we write $m_\alpha\in \R^{|\alpha|}$, $(m_\alpha )_j = m_{\alpha j } $ for every $j \in \alpha $, $\alpha \in X$. 
    \end{theorem} 
\begin{proof}
We have that 
  \begin{align*}
      \frac{d}{dt} N_\alpha (t)= - \sum_{j \in \alpha } \int_{\mathbb R_+} M_{\alpha j} (\xi)  f_{\alpha j }(t, d\xi) + \sum_{j \in \alpha } \sum_{\beta \in X \setminus \{ \alpha \}} \sum_{ k \in \beta } \int_{\mathbb R+} \mu_{k j }(\xi)  f_{\beta k } (t, d \xi).  
  \end{align*}
As in the proof of Theorem \ref{thm:RF to SP} one can show that the function $t \mapsto S_\alpha (t) $ with $S_\alpha (t)\in \mathbb R^{|\alpha |}$ given by
\[
(S_\alpha (t) )_j = \sum_{\beta \in X \setminus \{ \alpha \}} \sum_{ k \in \beta }  \int_{\mathbb R_+} \mu_{kj } (\xi) f_{\beta k } (t, d \xi)
\]
satisfies \eqref{eq:REInflux} with response function given by \eqref{response function gen} and with forcing function given by \eqref{forcing function gen}. 
Similarly, $J_\alpha (t) \in \mathbb R_+^{|\alpha |}$ given by 
\[
(J_\alpha (t))_j = \int_{\mathbb R_+} M_{\alpha j } (\xi) f_{\alpha j } (t, d\xi)
\]
satisfy \eqref{eq:REOutflux} with response function \eqref{response function gen out} and forcing function \eqref{forcing function gen}. 
\end{proof}
\begin{theorem}[RF and SP] \label{thm:SP to RF gen}
Let $\{N_\alpha , S_\alpha, J_\alpha \} $ satisfy \eqref{eq:Concentration}, \eqref{eq:REInflux}, \eqref{eq:REOutflux} with response functions $G_{\alpha \beta}  \in C(\mathbb R_+; \mathbb R_+^{|\beta|\times |\alpha|})$ and $K_\alpha  \in C(\mathbb R_+; \mathbb R_+^{|\alpha|\times |\alpha |})$ assume that $\{S_\alpha^0, J_\alpha^0 \}$ satisfy \eqref{PDE RE:forcing functions} where $\{ m_{\alpha j } \} $ is a family of measures $m_{\alpha j} \in \mathcal M_{+, b}(\mathbb R_+)$. 
Assume in addition that for every $i \in \alpha $ and $j \in \beta $ with $\alpha,  \beta \in X $ the function 
\begin{equation}\label{lambda as a function of g}
\mu_{ij  }(t):= \frac{(G_{\alpha \beta }(t))_{ji} }{1- \sum_{\gamma \in X \setminus \{ \alpha\} } \sum_{k \in \gamma } \int_0^t (G_{\alpha \gamma }(s))_{ki } ds } \quad t \geq 0
\end{equation}
is continuous and bounded. 
Then the family of functions $\{ f_{\alpha i}  \} $ solving \eqref{eq:gen1}, \eqref{eq:gen2}, \eqref{eq:gen3} satisfies \eqref{n function of f general}.
    \end{theorem} 
    \begin{proof}
        Since $\{f_{\alpha j} \}$ solves \eqref{eq:gen1},\eqref{eq:gen2}, \eqref{eq:gen3} with coefficients $\mu_{ij}$ given by \eqref{lambda as a function of g} the functions $\bar{N}_\alpha (t):= \sum_{ j \in \alpha } \int_{\R_+}f_{\alpha j }(t,d\xi)$ satisfy \eqref{eq:Concentration}, \eqref{eq:REInflux}, \eqref{eq:REOutflux} with the same forcing function as $\{N_\alpha,S_\alpha,J_\alpha\}$ and some response functions $\{\bar{G}_{\alpha\beta}\}$ given by \eqref{forcing function gen}. Then
    \begin{align*}
        \int_0^t M_{\alpha i } (v) dv = \int_0^t \frac{ \sum_{\beta \in X \setminus \{\alpha\} } \sum_{j \in \beta } (G _{\alpha \beta }(v))_{ji }}{1- \sum_{\gamma \in X \setminus \{\alpha\}  }  \sum_{k \in \beta } \int_0^v (G_{\alpha \gamma }(s))_{ki }ds }  dv 
        = -  \ln \left( 1- \int_0^t \sum_{\beta \in X \setminus \{\alpha\} } \sum_{k \in \beta } (G _{\alpha \beta }(s))_{ki} ds \right). 
    \end{align*}
    Hence, $ e^{- \int_0^t M_{\alpha i } (v) dv }  = 1- \int_0^t \sum_{\beta \in X \setminus \{\alpha\} } \sum_{ k \in \beta } (G_{\alpha \beta }(s))_{ki } ds$. We deduce that $\bar G_{\alpha \beta } = G_{\alpha \beta }$. By uniqueness of the solution, we obtain $N_\alpha=\bar{N}_\alpha$. This concludes the proof.
    \end{proof}

\subsection{Initial conditions of ODEs compatible with a SPEs reformulation} \label{sec:ODE SPEs}

We now characterize the initial conditions of the ODE \eqref{eq:PrelDecompODE} that guarantee that the corresponding forcing function $S^0_\alpha $ is of the form \eqref{forcing function gen} for some $m_\alpha$. 
This allows to associate with these ODEs a SPE of the form \eqref{eq:gen1}, \eqref{eq:gen2}, \eqref{eq:gen3}. 
Indeed, starting from the ODEs \eqref{eq:PrelDecompODE} we can write a system of RFEs for $N_\alpha $, see Theorem \ref{thm:ReductionODEtoRE}. Then, if $S^0_\alpha $ is of the form \eqref{forcing function gen} we can associate to this RFEs a SPEs. 
With this procedure we can therefore, implicitly, associate to these ODEs an age structure. 

%The reason why we want to characterise these initial conditions is that, when \eqref{forcing function gen} holds for some $m_\alpha$ the corresponding RFEs are equivalent to SPEs of the form \eqref{eq:weak formPDE gen}. 
%More precisely, with equivalent we mean that given the forcing and response functions of a RFEs we can associate, by Theorem \ref{thm:SP to RF gen}, a SPEs of the form \eqref{eq:weak formPDE gen} with coefficients given by \eqref{lambda as a function of g}. 
%Then the solution of equation \eqref{eq:Concentration} and \eqref{eq:gen1} are related by \eqref{n function of f general}. 
%Vice-versa, given the solution $\{ f_{\alpha j } $ of a SPEs of the form \eqref{eq:weak formPDE gen} we can construct the RFEs satisfied by the $\{ N_\alpha \} $ given by \eqref{n function of f general}. The forcing and response functions are given by \eqref{response function gen}, \eqref{response function gen out}, \eqref{forcing function gen}. 

\begin{proposition}[ODEs and SPEs]\label{prop:age to ODEs}
Let $N_\alpha $ be given by \eqref{eq:DecompDefNumberComp} where $\{n_\alpha \}$ is the family of solutions to \eqref{eq:PrelDecompODE}. 
If for every $\beta \in X $ there exists a $m_\beta \in \left( \mathcal M_{+, b} (\mathbb R_+) \right)^{|\beta |}$ such that 
\begin{equation} \label{n_0 PDE}
n^0_\beta = \int_{\mathbb R_- } e^{- \xi A_{\beta \beta }}  m_\beta( d\xi) 
\end{equation} 
then $N_\alpha$ is given by \eqref{n function of f general}, where $\{ f_{\alpha j } \}$ is the solution to a SPE of the form \eqref{eq:gen1}, \eqref{eq:gen2}, \eqref{eq:gen3}.
\end{proposition}
\begin{proof}
 Theorem \ref{thm:ReductionODEtoRE} implies that $N_\alpha $ satisfies equation \eqref{eq:Concentration}. The flux $S_\alpha $ solves \eqref{eq:REInflux} with response functions $G_{\alpha \beta }(t)=A_{\beta \alpha} e^{t A_{\alpha \alpha }} $ and forcing functions 
\begin{align}\label{eq:ForcingFunctionAgeStructure}
  S^0_\alpha (t)= \sum_{\beta \in X \setminus \{ \alpha \}} G_{\beta \alpha } (t) n_\beta^0 =\sum_{\beta \in X \setminus \{ \alpha \}} A_{\alpha \beta } e^{t A_{\beta \beta } } n_\beta^0= \sum_{\beta \in X \setminus \{ \alpha \}} \int_{\mathbb R_-} 
 G_{\beta \alpha} (t-\xi) m_\beta(d\xi), 
\end{align}
where in the last equality we used \eqref{n_0 PDE}. 
As a consequence of \eqref{eq:ForcingFunctionAgeStructure} and Theorem \ref{thm:RF to SP gen} we can associate to the RFEs \eqref{eq:Concentration}, \eqref{eq:REInflux}, \eqref{eq:REOutflux} a system of SPEs of the form \eqref{eq:gen1}, \eqref{eq:gen2}, \eqref{eq:gen3}.
\end{proof}

As explained above, Proposition \ref{prop:age to ODEs} guarantees that if \eqref{n_0 PDE} holds, then we can introduce an age structure in the ODEs system.

\section{Specific examples and applications: linear case }\label{sec:Examples}
We now consider some examples of applications of the response function formalism in biochemistry. We consider first systems leading to linear problems. 
The examples that we study include the classical kinetic proofreading model introduced by Hopfield and Ninio (Section \ref{sec:proofreading}), a model of non-Markovian linear polymerization (Section \ref{sec:linear polimerization}), and a simple linear network inspired by the model of robust adaptation in \cite{barkai1997robustness}, (see Section \ref{sec:adaptation}). 

\subsection{Kinetic proofreading, Hopfield model} \label{sec:proofreading}
Our first application concerns a kinetic proofreading mechanism due to Hopfield \cite{Hopfield} and Ninio \cite{ninio1975kinetic}. 
The Hopfield model, or other mechanisms of kinetic proofreading inspired by that, have been found in many biological processes, including pathogen recognition from the immune system \cite{goldstein2004mathematical}, \cite{mckeithan1995kinetic}, DNA replication, m-RNA translation, DNA recognition and DNA transcription, see the review \cite{Boeger} and also \cite{Murugan,PigolottiSartori}.

The classical Hopfield model is a biochemical network of reactions of the form 
\begin{equation} \label{reaction:CS}
C \leftrightarrows S, \quad C \leftrightarrows S^*
\end{equation} 
\begin{equation} \label{reaction:SS}
S \ \overleftarrow{	\rightsquigarrow}  \ S^*
\end{equation} 
\begin{equation} \label{reaction:empty}
C \rightarrow \emptyset, \quad S^* \rightarrow P 
\end{equation} 
where we indicate with $ \leftrightarrows$ the reactions having detailed balance and with $ \overleftarrow{	\rightsquigarrow}$ the reaction that does not have detailed balance.
For more complex kinetic proofreading networks we refer to \cite{Murugan}, \cite{murugan2014discriminatory}. 

It is important to notice that the state denoted by $S$ is a complex that consists in the combination of the molecule denoted by $C$ with some receptor. 
In the formulation of \eqref{reaction:CS} and \eqref{reaction:SS} we assume that the number of receptors is very large and therefore the number of attachment points for $C$ can be assumed to be constant. 
On the other hand, $S^*$ is the phosphorylated state of $S$ and it synthesize the product $P$.
For example, $C$ could be a codon on the mRNA, $S$ the complex made of $C$ and a tRNA. The product $P$ would then be an amino acid. We refer to Chapter 7 in \cite{alon2019introduction} for more biological details.

We assume that $E(S)- E(C)=E_1 $ while $E(S^*) - E(C) = E_2 $ where $E(S) , E(S^*) , E(C) $ are the free energies of $S$, $S^*$ and $C$.
The affinity $E_1$ to be detected by the receptor can be different for another molecule $\overline C$, producing the complexes $\overline S $ and $\overline S^*$. 
In other words, the affinity $\overline E_1= E(\overline S) - E(\overline C) $ can be slightly different from $E_1 $. Specifically, if the circuit has a preference for the molecule $C$ over $\overline C$ we have $\overline E_1 > E_1$. 
However $E_1-E_2$ can be assumed to be constant, i.e.  $E_1-E_2= \overline E_1 - \overline E_2 $ where $\overline E_2= E(\overline S^*) - E( \overline S) $.
The reason why we can make this assumption is that the phosphorylation process takes place in a part of the molecule that is far from the part of the receptor where the molecule $C$ or $\overline C$ was attached.
The phosphorylation reaction \eqref{reaction:SS} does not have detailed balance due to the consumption of ATP (or energy) in an irreversible manner, see Figure \ref{fig:HopfieldGraph}.

The rate of the reaction $C \rightarrow S $ is $k$. Due to the detailed balance condition the rate of the reaction $S \rightarrow C$ is $k e^{E_1}$. Similarly the rate of the reaction $C \rightarrow S^* $ is $\beta $ and the one of the reaction $S^* \rightarrow C$ is $\beta e^{E_2} $. 
Finally the rate of the reaction $S \rightarrow S^*$ is $\alpha $. On the other hand, due to the lack of detailed balance, we will denote the rate of the reverse reaction $S^* \rightarrow S $ as $\alpha e^{E_2-E_1} /Q $, where $Q $ is a coefficient which measures the lack of detailed balance. If  $Q = 1$ the reaction \eqref{reaction:SS} has detailed balance.
On the contrary, if $Q>1$ energy is spent in an irrerversible manner to transform $S$ into $S^*$. 

%While an energetic discrimination relies merely on energy barriers (i.e. energetic affinities of certain reactions to take place) the kinetic proofreading mechanism involves reactions with an extra energy cost. This allows to increase the discrimination against incorrect products. 
%We consider a substrate $ C $, which results, by means of a chemical reaction, in a substrate $ S $ and eventually to an excited state $ S^* $. 
%The excited state $S^*$ then leads to a correct product $ P $ with rate $ \lambda $. The reactions $ C\to S $, $ S^*\to C $ are assumed to satisfy detailed balance (with relative rates $ k, \, \beta $ and energies $ E_1,\, E_2 $) while the reaction $ S\to S^* $ spends energy (with relative rate $ \alpha $ and defect parameter $ Q $), see Figure \ref{fig:HopfieldGraph}. 

\begin{figure}[ht]
	\centering
	\includegraphics[width=0.3\linewidth]{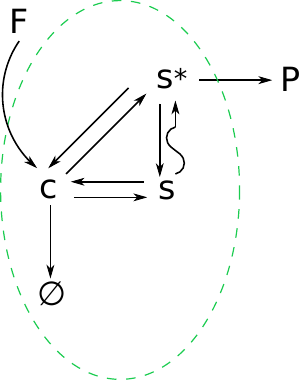}
	\caption{Reactions in the proofreading mechanism due to Hopfield.
 In the green dashed line we have the compartment $\{C, S, S^*, \emptyset \}$.}
	\label{fig:HopfieldGraph}
\end{figure}

More precisely, we consider the ODEs
\begin{align}\label{eq:FullHopfieldModel}
	\begin{split}
		\dfrac{dC}{dt} &= -(k+\beta+\mu)C+ke^{E_1} S +\beta e^{E_2} S^*+F(t)
		\\
		\dfrac{dS}{dt} &= kC-(ke^{E_1}+\alpha) S+\dfrac{\alpha}{Q}e^{E_2-E_1}S^*
		\\
		\dfrac{dS^*}{dt} &= \beta C+ \alpha S-\left(\dfrac{\alpha}{Q}e^{E_2-E_1}+\beta e^{E_2} + \lambda \right) S^*
		\\
		\dfrac{dP}{dt} &= \lambda S^*.
	\end{split}
\end{align}	
The equations for $\overline C, \overline S , \overline S^*$ are identical, except for the energies  $E_1$ and $E_2$ which are $\overline E_1$ and $\overline E_2 $. 
The term $F$ in \eqref{eq:FullHopfieldModel} is an external source of substance $C$.

%Similarly, we can write a system for substrates $ \bar{C} $, $ \bar{S} $, $ \bar{S}^* $ leading to an incorrect product $ \bar{P^*} $. We assume that the rates $ \alpha, \beta, \lambda, \mu, Q $ are the same as before. However, we assume that the energetic affinities are different, i.e. $ E_1<\bar{E}_1 $ and $ E_2<\bar{E}_2 $ with $ \bar{E}_2-\bar{E}_1=E_2-E_1 $.

Changing the time unit, we can assume without loss of generality that $ k=1 $ in \eqref{eq:FullHopfieldModel}.
The effective behaviour of the reactions $ C\to P $, $ \bar{C}\to P $ can be described via the corresponding response functions $ \Phi,\, \bar{\Phi} $.
The response functions $\Phi $ and $\overline \Phi$ can be obtained taking $F(t)=\delta_0(t)$.

%Using the terminology introduced in Section \ref{subsec:ReductionODEtoSPM} we have that  $ \{C,S,S^*, \emptyset\}, $ and  $ \{\bar{C},\bar{S},\bar{S}^*, \emptyset\} $ are two compartments. 

Using the linearity and invariance under time translations of \eqref{eq:FullHopfieldModel} we can rewrite the total production until the time $t$ in the form 
\[
\frac{d}{dt}P(t)= \int_{-\infty }^t \Phi(t-s) F(s) ds 
\]
where $\Phi(t) = \lambda S^*(t) $, where $S^*$ is computed solving \eqref{eq:FullHopfieldModel} with initial value $C(-\infty ) = S(-\infty ) = S^*(-\infty )=0$ and $F(t)=\delta_0(t)$. 
%Theorem \ref{thm:ReductionODEtoRE} and Lemma \ref{lem: RE one entrance and general} implies that
%\begin{align*}
%	\Phi(t) = \lambda (e^{At})_{S^*,C} = \lambda S^*(t), \quad \bar{\Phi}(t) = \lambda (e^{\bar{A}t})_{\bar{S}^*,\bar{C}} = \lambda \bar{S}^*(t),
%\end{align*}
%where the matrices $ A,\, \bar{A}\in \R^{3\times 3} $ describe the reactions in \eqref{eq:FullHopfieldModel}. Here, $ S^*(t) $, $ \bar{S}^*(t) $ solve the corresponding systems 
This is equivalent to solving \eqref{eq:FullHopfieldModel} for $t>0$ with initial conditions $ \bar{C}(0)=1 $, $ \bar{S}(0)=\bar{S}^*(0)=0 $, $P(0)=0$ respectively. 

We will denote with $P_\infty  $ the total quantity of product generated upon excitation by a signal $F(t)=\delta_0(t)$ of the molecule $C$ and with $\overline P_\infty $ the same quantity produced by a signal of the molecule $\overline C$. 
Then, the total production ratio is given by
\begin{align*}
\frac{\overline P_\infty }{P_\infty }= \dfrac{\int_0^\infty\bar{\Phi}(t)\, dt}{\int_0^\infty\Phi(t)\, dt} = \dfrac{\hat{\bar{S}}^*(0)}{\hat{S}^*(0)},
\end{align*}
where $ \hat{\bar{S}}^*, \, \hat{S}^* $ are the Laplace transforms of $ \bar{S}^*, \, S^* $. 
Taking the Laplace transform of \eqref{eq:FullHopfieldModel}, evaluating $ z=0 $ and solving the linear system leads to
\begin{align*} 
\frac{\overline P_\infty }{P_\infty }=&\theta^2 \, \dfrac{1+\beta+\beta \xi/\alpha\theta}{1+\beta+\beta\xi/\alpha} \, \dfrac{\xi^2+\xi((\beta+\mu)\lambda/\mu\beta\eta +\alpha/\beta Q+\alpha)+(1+\beta+\mu)\alpha\lambda/\mu\beta\eta}{\xi^2+\theta\xi((\beta+\mu)\lambda/\mu\beta\eta +\alpha/\beta Q+\alpha)+\theta^2(1+\beta+\mu)\alpha\lambda/\mu\beta\eta}
	\\
	\theta :=& e^{-(\bar{E}_1-E_1)}, \quad \xi := e^{E_1}, \quad \eta:= e^{E_2-E_1}= e^{\bar{E}_2-\bar{E}_1}.
\end{align*}
Note that $ \theta \in (0,1) $. Consequently, the above expression is greater than $ \theta^2=e^{-2(\bar{E}_1-E_1)}  $.
Let us mention that the optimal discrimination $ \theta^2 $ agrees with the one originally obtained by  Hopfield \cite{Hopfield} for constant flux solutions (i.e. assuming that $C$ or $\overline C$ are constant in time, hence ignoring the first equation in \eqref{eq:FullHopfieldModel}).  

The value $ e^{-2(\bar{E}_1-E_1)}  $ is achieved if
\begin{align}\label{eq:HopfieldLimits}
	\alpha\to 0, \quad \dfrac{\beta}{\alpha} \to 0, \quad \dfrac{\alpha}{\beta Q}\to 0, \quad \dfrac{\lambda}{\mu \eta}\to 0, \quad \dfrac{\lambda}{\beta\eta}\to 0, \quad \zeta:=\dfrac{\alpha\lambda}{\mu\beta\eta}\to 0
\end{align}
assuming that $\xi$ is of order $1$. 
Notice that the first three limit expressions in \eqref{eq:HopfieldLimits} imply that $ \alpha, \, \beta\to 0 $ as well as $ Q\to \infty $. The first two limits expression above imply that the reaction transforming $S$ in $S^*$ is slow, but faster than the reaction transforming $C$ in $S^*$. Furthermore, $ Q\to \infty $ yields that detail balance is strongly violated in the reaction \eqref{reaction:SS}.

\begin{figure}
	\centering
	\begin{minipage}[c]{0.99\linewidth}
		\includegraphics[width=\linewidth]{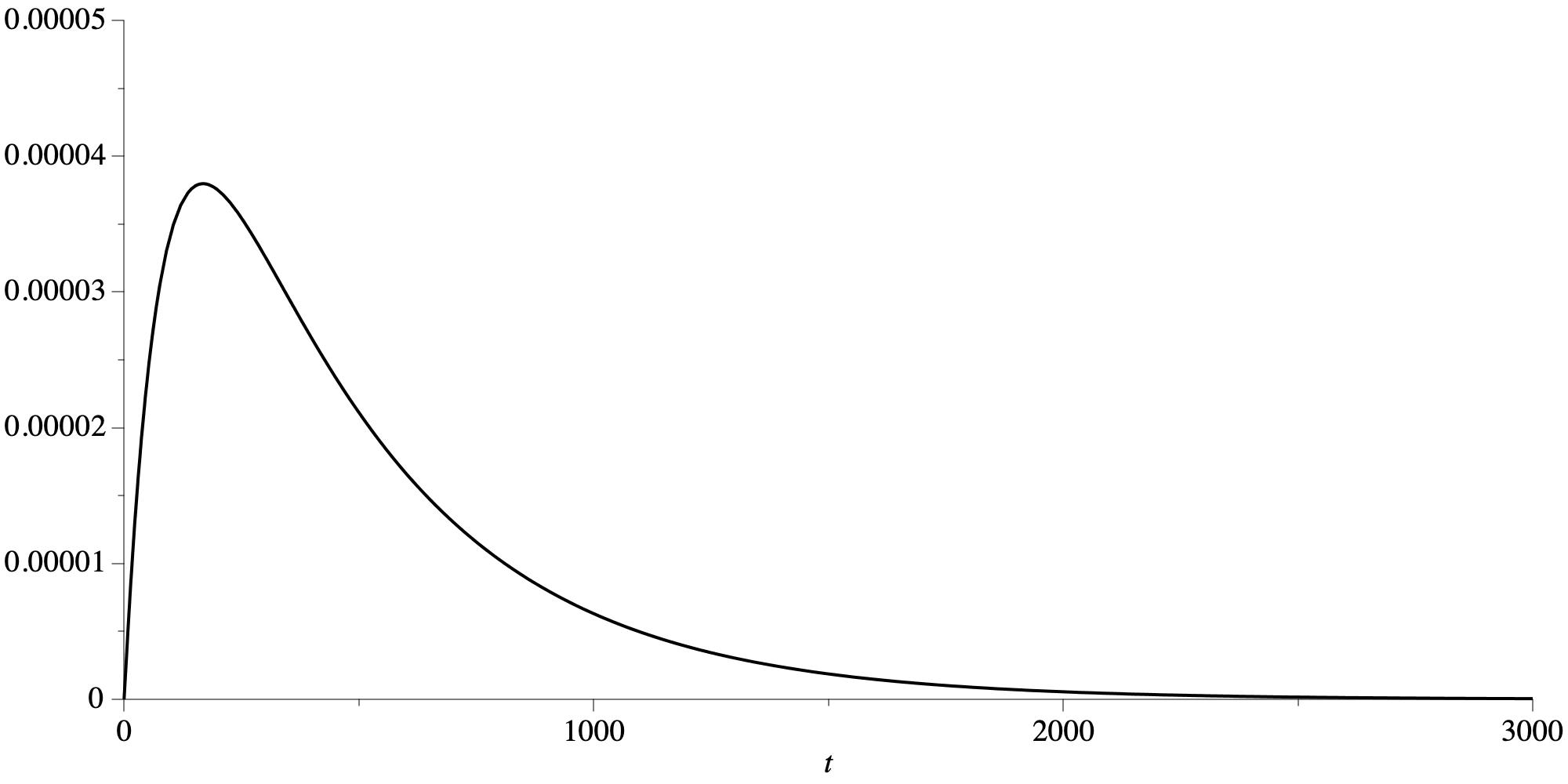}
	\end{minipage}
	\caption{Plot of response function $ \Phi(t) $ for $ \alpha=\mu=\varepsilon $, $ \beta=\varepsilon^2 $, $ Q=1/\varepsilon^2 $, $ \lambda=2\varepsilon^2 $, $ e^{E_1}=2 $, $ e^{E_2}=4/\varepsilon^s $, $ e^{\bar{E}_1}=8 $, $ e^{\bar{E}_2}=16/\varepsilon^s $ where $ \varepsilon=0.01 $, $s=1/2$.}
	\label{fig:HopfieldExample}
\end{figure}

In order to ensure that the last three limits expressions in \eqref{eq:HopfieldLimits} hold we can either assume $ \frac{\lambda}{\mu} \ll \frac{\beta}{\alpha }$ and $\lambda \ll \beta $ and $\eta $ of order $1$.  
Alternatively, it is possible to obtain all the formulas in \eqref{eq:HopfieldLimits} also with $ \eta\to \infty $, we refer to the examples later for this case.
Note that $ \mu\to0 $ is not required. However, if $ \mu $ is of order one, most of the signal $C$ will be lost in times of order $1$ (or shorter) due to the term $ -\mu C $ in \eqref{eq:FullHopfieldModel}.

In addition, we can also compute the total production $ P_\infty  $ under the assumptions \eqref{eq:HopfieldLimits}. We then obtain
\begin{align*}
P_\infty  &= \lambda \hat{S}^*(0) = \dfrac{1+\beta+\beta\xi /\alpha}{1+\beta+\mu+\xi(\mu\beta\eta/\lambda+\beta/\alpha+\mu/\alpha+\xi\mu\beta\eta/\lambda\alpha+\eta\mu/Q\lambda)} =  \dfrac{\zeta}{\xi^2} \left( 1+o(1) \right)
\end{align*}
where $\zeta $ is defined as in \eqref{eq:HopfieldLimits} and goes to zero and $\xi$ is of order $1$. 
We thus observe that the quadratic discrimination is obtained in the proofreading mechanism at the cost of having a very small fraction of molecules of $C$ generating the product $P$.

We now consider specific forms of the response functions for some particular scaling limits of the chemical coefficients. 
More precisely, we set 
\begin{align}\label{eq:HopfieldExample}
	\alpha=\varepsilon, \quad \beta=\varepsilon^2, \quad Q=\varepsilon^{-2},\quad  \mu=\varepsilon, \quad \dfrac{\lambda}{\eta}= \varepsilon^{2+s},
\end{align}
for some $ s\in (0,1] $ and for $ \varepsilon\to 0 $. In this case, we have as $ \varepsilon\to 0 $
\begin{align*}
\frac{\overline P_\infty }{P_\infty } =\theta^2 \left( 1+\mathcal{O}(\varepsilon^s) \right), \quad \frac{\overline P_\infty }{P_\infty } = \dfrac{\varepsilon^s}{\xi^2}\left( 1+\mathcal{O}(\varepsilon^s) \right).
\end{align*}

In Figure \ref{fig:HopfieldExample} we plot the response function $ \Phi $ in the case $ s=1/2 $, $ \varepsilon=0.01 $ and $ \lambda=2\varepsilon^{2} $, $ \eta=2\varepsilon^{-1/2} $ together with \eqref{eq:HopfieldExample}. Notice that the response function $\Phi$ that we obtain from this model is not an exponential function. This means that the dynamics is non-Markovian in general. However, as time tends to infinity $\Phi $ approaches an exponential function. This suggests that the system could be considered to be Markovian for times large enough.

\begin{figure}[ht]
	\centering
	\includegraphics[width=0.3\linewidth]{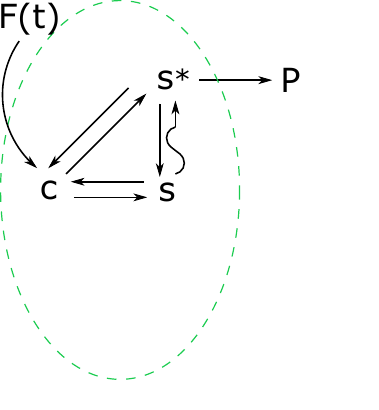}
	\caption{Reactions in the proofreading mechanism due to Hopfield, without the degradation of $C$.}
	\label{fig:HopfieldGraph2}
\end{figure}
The degradation term $ -\mu C $ in the equation for $ dC/dt $ is usually not included in the standard (time-independent) Hopfield  problem as in \cite{Hopfield}, see Figure \ref{fig:HopfieldGraph2}. 
Here, since we consider a time-dependent model, the degradation term is needed to be able to discriminate, in terms of production of $P$, the molecule $C $ and the molecule $\overline C$. 
Indeed, in the absence of degradation, every molecule $C$ or $\overline C$ would synthesize the product $P$. Although, the molecule with more affinity will be much faster than the other one in the production of $P$, as we explain now. 
We compute the average time needed for the molecule $C$ to produce $P$. This is given by 
\[
T=  \frac{\int_0^\infty t \Phi(t) dt }{\int_0^\infty \Phi(t) dt }
\]
where $\Phi(t)= \lambda S^*(t) $ and where $S^*$ is computed from equations \eqref{eq:FullHopfieldModel} with $\mu=0$ and with $C(0)=1$, $S(0 )=S^*(0 )=0$. 
Notice that, since $\frac{d}{dt} (C+S+S^*)= - \lambda S^* $, then 
\[
\int_0^\infty \Phi(t) dt= \lambda \int_0^\infty S^*(t) dt = C(0)=1 , 
\]
hence $T=\int_0^\infty t \Phi(t) dt$. 
Moreover, 
\[
T=\int_0^\infty t \Phi(t) dt=  \lambda \int_0^\infty t  S^*(t) dt =  -  \int_0^\infty t \frac{d}{dt} (C(t)+S(t)+S^*(t) )dt. 
\]
Integrating by parts and using the fact that $C, S, S^*$ decay exponentially as time tends to infinity,  
we obtain that 
\[
T= \int_0^\infty  (C(t)+S(t)+S^*(t) )dt = \hat{C}(0)+\hat{S}(0) + \hat{S^*}(0). 
\]

Performing the Laplace transform in \eqref{eq:FullHopfieldModel} we deduce that 
\[
T= \frac{( \lambda + \beta \eta e^{E_1} ) (k e^{E_1} +\alpha ) + \frac{\alpha}{Q} \eta k e^{E_1} + (k+\beta ) \frac{\alpha \eta}{Q}+k(\lambda +\beta \eta e^{E_1})  }{\lambda (\beta k e^{E_1} +\alpha k +\alpha \beta ) }
\]
We compare $T$ with $\overline T $ under the following assumptions 
\[
\lambda \ll \beta \eta e^{E_1}, \quad \alpha \ll k e^{E_1},\quad  \beta e^{E_1} \ll k,\quad  \frac{\alpha }{Q} \ll \beta e^{E_1}, \quad  e^{E_1} \gg 1 
\]
and deduce that 
\[
\frac{T}{\overline T } \approx e^{-2 (\bar{E}_1-E_1)}. 
\]
This is in agreement with the results in \cite{Hopfield} for solutions with constant fluxes.

\subsection{A linear polymerization model} \label{sec:linear polimerization}
In this section we describe a simple (linear) polymerization model which includes a kinetic proofreading mechanism. 
Polymerization processes are ubiquitous in biology. Some of the most important examples of polymerization in biochemistry are the transcription of DNA in mRNA and for the translation of mRNA into polypeptides. Some of the most widely studied mathematical models of this last process can be found in \cite{macdonald1969concerning,macdonald1968kinetics, pipkin1966kinetics}. 

Most of the models of polymerization used in biology that we are aware of are Markovian.  On the other hand, in all the polymerization models mentioned above kinetic proofreading mechanism and error correcting reactions take place whenever a monomer is incorporated to the polymer (see \cite{PigolottiSartori}).
Hence, as explained in Section \ref{sec:proofreading}, it is natural to formulate a polymerization model for these process in which the addition of monomers takes place in a non-Markovian manner.  
We will describe in this section a very simple polymerization model using the formalism of response functions to describe each polymerization step. 
Notice that we do not try to include in this model the simultaneous reading of a single mRNA strain by several ribosomes, as it has been done in \cite{macdonald1969concerning,macdonald1968kinetics, pipkin1966kinetics}.
Nevertheless we remark that the model considered here is the non-Markovian version of the model in \cite{pipkin1966kinetics}. 

We assume that polymers are characterized by their length $\ell$ and that polymers interact only with monomers.
When a monomer binds to a polymer a sequence of kinetic proofreading reactions starts. 
 After these reactions take place a monomer of size $\ell+1 $ is formed.
Due to the proofreading mechanisms the reaction $(\ell) \rightarrow (\ell +1) $ is non-Markovian and, hence, we will model it using the formalism of response functions. 

Let $n_\ell (t) $ be the number of polymers of length $\ell$ at time $t$. 
Then $n_\ell $ increases in time due to the flux $I_\ell$ of polymers from size $\ell -1$ to size $\ell$ and decreases due to the flux $I_{\ell+1}$ of polymers from size $\ell $ to size $\ell+1$. 
Namely, for $\ell  \geq 1$,  we have
\begin{align} \label{eq:linear polymer}
   \partial_t n_\ell (t) = I_\ell (t)- I_{\ell +1}(t), \quad \ell \geq 1 
   \end{align}
where    
  \begin{align} \label{RE polymerization}
    I_\ell (t)=  \int_{-\infty}^t \Psi(t-s) I_{\ell -1 } (s) ds, \ \ell \geq 2.    
\end{align}
Here, to simplify the analysis,  the response function $\Psi$ is assumed to be independent on $\ell $. 
Moreover, we assume that $\int_0^\infty \Psi(s) ds =1$.

We now study the long-time behaviour of $n_\ell$. 
To this end we assume that there is a constant flux of monomers entering the system, that is $I_1 (t) =1$ for all $t\geq0$. Furthermore, we assume that $n_\ell (0)=0$ for every $\ell \geq 1 $. 
Taking the Laplace transform to all the terms of equation \eqref{RE polymerization} we deduce that 
\[
\hat{I}_\ell(z)= \hat{\Psi} (z) \hat{I}_{\ell -1} (z) \ \text{  for } \  \ell \geq 2  \ \text{ and } \ \hat{I}_{1} (z) = \frac{1}{z}.
\] 
It follows that $\hat{I}_\ell(z)= \hat{\Psi}(z)^{\ell -1 }\hat{I}_1(z) = \frac{1}{z} \hat{\Psi}(z)^{\ell -1 }$. 
Applying the Laplace transform also to all the terms in equation \eqref{eq:linear polymer}, we deduce that $z \hat{n}_\ell (z)= \hat{I}_\ell (z) - \hat{I}_{\ell+1}. $
Hence 
\[ 
\hat{n}_\ell (z)= \frac{1}{z^2} \left( \hat{\Psi}(z)^{\ell-1}  -  \hat{\Psi}(z)^{\ell} \right) = \frac{1}{z^2}  \hat{\Psi}(z)^{\ell-1}  \left(  1 - \hat{\Psi}(z) \right). 
\] 
In order to obtain the long-time asymptotics for $n_\ell $ we need to consider the asymptotics of $\hat{n}_\ell $ for $z$ small. 
For $z$ small the function $\hat{\Psi}(z)$ can be approximated by $ 1- \mu z$
where $\mu= \int_0^\infty s \Psi(s) ds $. 
Then 
\[
\hat{n}_\ell (z)= \frac{1}{z^2}  \hat{\Psi}(z)^{\ell-1} \left(1 -  \hat{\Psi}(z)\right) \approx \frac{1}{z^2}  (1-\mu z)^{\ell-1}  \mu z = \frac{\mu }{z}  (1-\mu z)^{\ell-1} \approx \frac{\mu }{z}  e^{-  \mu z \ell} 
\]
as $z$ goes to zero. 
Inverting the Laplace transform of $\frac{\mu }{z}  e^{-  \mu z \ell} $ we obtain that $n_\ell $ behave like a wave front of the form $n_\ell (t) \approx   \mu  \chi_{[\ell \mu , \infty )} (t) $ for large times. 
Notice that this solution describes a front of concentration in the space of polymer size propagating with speed $1/\mu $. 
A more detailed description of the solution near the edge of the front needs a more precise analysis that we will not pursue here. 

\subsection{A linear chemical model of adaptation} \label{sec:adaptation}
An important concept in Systems Biology is the one of adaptation. 
A system shows adaptation if it reacts to gradients (in time or space) of a chemical, rather than to absolute values of each concentration. 
One of the earliest models of adaptation is the Barkai-Leibler model of bacterial chemotaxis, see
\cite{alon2019introduction, barkai1997robustness}. Other models of adaptation can be found in
\cite{ferrell2016perfect,tang2010defining}. 

In this section we present a very simple model showing adaptation.
This model can be thought as a linear version of the classical Barkai-Leibler model.  This model is suited for a RFE reformulation as we have an input function, a compartment where reactions take place and an output.
Although, in this case the definition of response function has to be slightly different from the one in the previous sections. Indeed, we will prove that the integral of the response function is equal to $0$. This is necessary to induce adaptation as it forces the system to return to its initial state, if the signal remains constant for sufficiently long time.  

The system of ODEs we consider is the following 
\begin{align*}
   & \frac{ d X}{dt} = a Y - b X + s(t) \\
       & \frac{ d Y}{dt} = 1- X
\end{align*}
with $a, b >0$. 
Here $X$ measures the quantity of active receptors. It increases when the signal $s(t)$ starts. Instead, $Y$ is the response regulator protein.
The output of this system will be the quantity of active receptors. 
We remark that the constant source term in the second equation could be the limit value of the Michaelis-Menten law in the saturation regime. 

It is convenient to make the change of variables $\xi = X-1 $ in the above equation, hence 
\begin{align} \label{eq:adaptation}
   & \frac{ d\xi }{dt} = a Y - b \xi - b  + s(t) \\
       & \frac{ d Y}{dt} = - \xi \nonumber
\end{align}
Notice that the steady state of $Y$ is reached when $X=\overline X=1$. We stress that $\overline X$ does not dependent on the signal $s$. 
This is a necessary property to have adaptation. 

We now reformulate equation \eqref{eq:adaptation} using the formalism of response functions. 
We can think about $\{ \xi, Y\}$ as a compartment. 
The output of the chains of reactions taking place in the compartment is $bX$, while the input is $s(t)-b $.
The dynamic inside the compartment $\{ \xi, Y\}$ is driven by the following subsystem of ODEs
\begin{align*}
   & \frac{ d}{dt} \left[ \begin{matrix} \xi \\ Y  \end{matrix} \right]   = A  \left[ \begin{matrix} \xi \\ Y  \end{matrix} \right]   
\end{align*}
where $
A= \left[ \begin{matrix}
     - b & a \\
     -1  & 0
\end{matrix}\right]. $
We compute the response function $\Phi$. Applying Theorem \ref{thm:ReductionODEtoRE} together with Lemma \ref{lem: RE one entrance and general} we deduce that $\Phi(t)=( b, 0) e^{A t } e_1 $.

The matrix $A$ has two eigenvalues $\lambda_{\pm} = \frac{1}{2} \left( - b \pm \sqrt{b^2 - 4 a } \right) $, where $\lambda_- < \lambda_+ < 0$ corresponding to the eigenvectors 
   \[ v_{\pm}=
\left[\begin{matrix}
  1 \\
  - \frac{\lambda_{\mp}}{a}
\end{matrix}\right]. 
\]
As a consequence $\Phi(t)= C_+e^{\lambda_+ t} + C_- e^{\lambda_- t}.$
Using the fact that $\Phi(0)=b $ we deduce that 
\[
C_+= \frac{ b \lambda_+}{\sqrt{b^2-4a}} \ \text{ and } \ C_-= - \frac{ b \lambda_-}{\sqrt{b^2-4a}}. 
\]
Hence, $\Phi(t)=\frac{b}{\sqrt{b^2-4a}} \left(  \lambda_+  e^{\lambda_+ t}  -   \lambda_- e^{\lambda_- t} \right). $
Then we have the following dependence of $X$ on $s(t)$
\[
X(t)= 1+ \dfrac{1}{b}\int_0^t \Phi(v) (s(t-v) - b ) dv. 
\]
Notice that $\int_0^\infty \Phi(v) dv =0$. As a consequence, if $s(t)\rightarrow k $ as $t\rightarrow \infty $ with $k >b $, then $X(t) \rightarrow 1$ as $t \rightarrow \infty $.
This form of the response function is the one that might be expected from a system exhibiting adaptation, i.e. the output converges to a constant value if the signal $s(t)$ approaches a constant value.

\section{Specific examples and applications: non-linear case }
\label{sec:nonlinear}
In this section we present some examples of non-linear RFEs. We will not analyse the mathematics of the models presented here in full detail. The aim of this section is to explain how to apply the formalism of response functions
in order to describe the interactions of different parts of a biochemical system, modelled by non-linear ODEs containing  non-linearities of the form of mass action, or Michaelis-Menten.
In particular, for these types of models it is possible to describe the response of a given compartment to a given set of inputs $I_i$ by means of functionals of the  form
\begin{align} \label{eq:nonlinearresponse}
    & R(t) = \sum_{i} \int_{-\infty }^t  I_i(s)\psi_i(t-s) ds  +   \sum_{i} \sum_{j} \int_{-\infty }^t  \int_{-\infty }^t I_i(s_1) I_j(s_2) \psi_{ij}(t- s_1, t- s_2 ) ds_1 ds_2 \\
    &+ \sum_{i} \sum_{j} \sum_k \int_{-\infty }^t  \int_{-\infty }^t  \int_{-\infty }^t I_i(s_1) I_j(s_2) I_k(s_3)  \psi_{ijk}(t- s_1, t- s_2 , t-s_3) ds_1 ds_2 ds_3+ \dots. \nonumber 
\end{align}
The main new feature of the non-linear models \eqref{eq:nonlinearresponse} compared to the linear ones of the form \eqref{A1} is that the response can contain information about the correlations in time of the inputs $I_i$. 

It is interesting to mention that 
RFEs of the form \eqref{eq:nonlinearresponse} containing only linear and quadratic terms of the inputs can yield a much richer set of responses than the linear RFEs. 
In Section \ref{sec:FF} we show that one of the most common network motifs appearing in metabolic networks, \cite{alon2019introduction}, yield a response function of the form 
\begin{align*}
    & R(t) = \sum_{i} \int_{-\infty }^t  I_i(s)\psi_i(t-s) ds  +   \sum_{i} \sum_{j} \int_{-\infty }^t  \int_{-\infty }^t I_i(s_1) I_j(s_2) \psi_{ij}(t- s_1, t- s_2 ) ds_1 ds_2. 
\end{align*}
Moreover, we will see that these type of response functions allow to describe the most distinguished features associated to this particular network motif. 

Furthermore, in Section \ref{sec:nonlinearPoly} we formulate equations that describe a non-Markovian polymerization model. 

\subsection{Feed Forward network} \label{sec:FF}
We can illustrate the usefulness of the theory of RFEs computing the RFEs for a specific biochemical network, namely the so-called Coherent Type $1$ Feed Forward Loop (C1FFL).
This network motif has been found often in many metabolic networks, 
 we refer to \cite{alon2019introduction} for an extensive description of this motif. 
 
We assume that a signal $S$ activates a protein $X$, which promotes the production of a protein $Y$. 
Then the proteins $X$ and $Y$ jointly produce $Z$. 
This motif can be represented as in Figure \ref{fig:FFL}, using the logic formalism of the AND/OR-gates. 

	\begin{figure}[ht]
		\centering
		\includegraphics[width=0.4\linewidth]{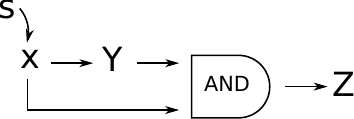}
		\caption{Type $1$ coherent Feed Forward loop.}
		\label{fig:FFL}
	\end{figure}
The main feature of the C1FFL is that it yields a delay in the production of $Z$, upon activation of the signal $S$.
Instead, when the signals $S$ stops, the production of $Z$ stops immediately. Therefore, using the terminology of \cite{alon2019introduction}, this network is a sign-sensitive delay element. 
In this section, we will show a class of response functions exhibiting sign-sensitive delay. 
The biological advantage of this mechanism is that it avoids to produce an immediate response to a fast fluctuating input signal. 
We present now a system of ODEs representing the C1FFL and we examine the REFs type of equations that describe the relation between the input and the output of the system.  

A possible way of modelling the C1FFL is with the following system of ODEs
\begin{align} \label{ODE FFL}
   & \frac{dX}{dt} = S- a X \nonumber \\
    & \frac{dY}{dt}= X-b Y \\
    & \frac{d Z }{dt} = XY- c Z, \nonumber  
\end{align}
where $a,b,c$ are positive constants. 
From this system of equation we deduce that 
\begin{align} \label{eq X}
    X(t)= \int^t_{- \infty } e^{-a (t-s) } S(s) ds, \quad 
  Y(t)= \int_{-\infty }^t e^{- b (t-s)} X(s) ds.
\end{align}
Using \eqref{eq X} we obtain 
\begin{align*}
    & Y(t)= \int_{-\infty }^t e^{- b (t-s)} X(s) =  \int_{-\infty }^t e^{- b (t-s)} \int_{-\infty}^s S(v) e^{-a(s-v) } dv ds \\
     &= e^{-bt} \int_{-\infty }^t \int_v^t e^{s(b-a)} ds S(v) e^{ a v }  dv = \frac{1}{b-a} e^{-at} \int_{-\infty }^t \left( 1 -e^{(v-t) (b-a)}  \right)   e^{ a v }S(v)  dv. 
\end{align*}

Finally, from the equation for $Z$ we infer that 
\begin{align*}
    & Z(t)= \int_{-\infty }^t  X(s) Y(s) e^{-c (t-s) } ds \\
    &=   e^{-ct  }  \int_{-\infty }^t e^{ (c-2a)s } \left[ \int_{-\infty }^s e^{ a v } S(v) dv \right] \left[ \int_{-\infty }^s \frac{\left( 1- e^{(b-a)(w-s) }\right)}{b-a} e^{w a  } S(w) dw  \right] ds. 
\end{align*}
Using Fubini, we deduce that 
\begin{equation} \label{response function for Z}
Z(t)= \int_{-\infty }^t \int_{-\infty }^t  K(t- v, t- w) S(w) S(v)  dw dv
\end{equation}
    where 
    \[
    K(\eta , \xi ) =  \int_{- \min\{\eta , \xi \}}^0 e^{(c-2a)s  }  e^{- a(\xi + \eta  )}     \frac{\left( 1-  e^{ -(b-a)(\xi+ s) } \right)}{b-a} ds, \quad \xi , \eta \geq 0.   \]
    Note that
    \[
    \frac{ 1-  e^{ -(b-a)(\xi+ s) }}{b-a} \geq 0 
    \]
in the domain of integration. See Figure \ref{fig:FFL_Kernel} for a plot in the case $a=c=5$, $b=1$.

\begin{figure}[ht]
	\centering
	\begin{minipage}[c]{0.75\linewidth}
		\includegraphics[width=\linewidth]{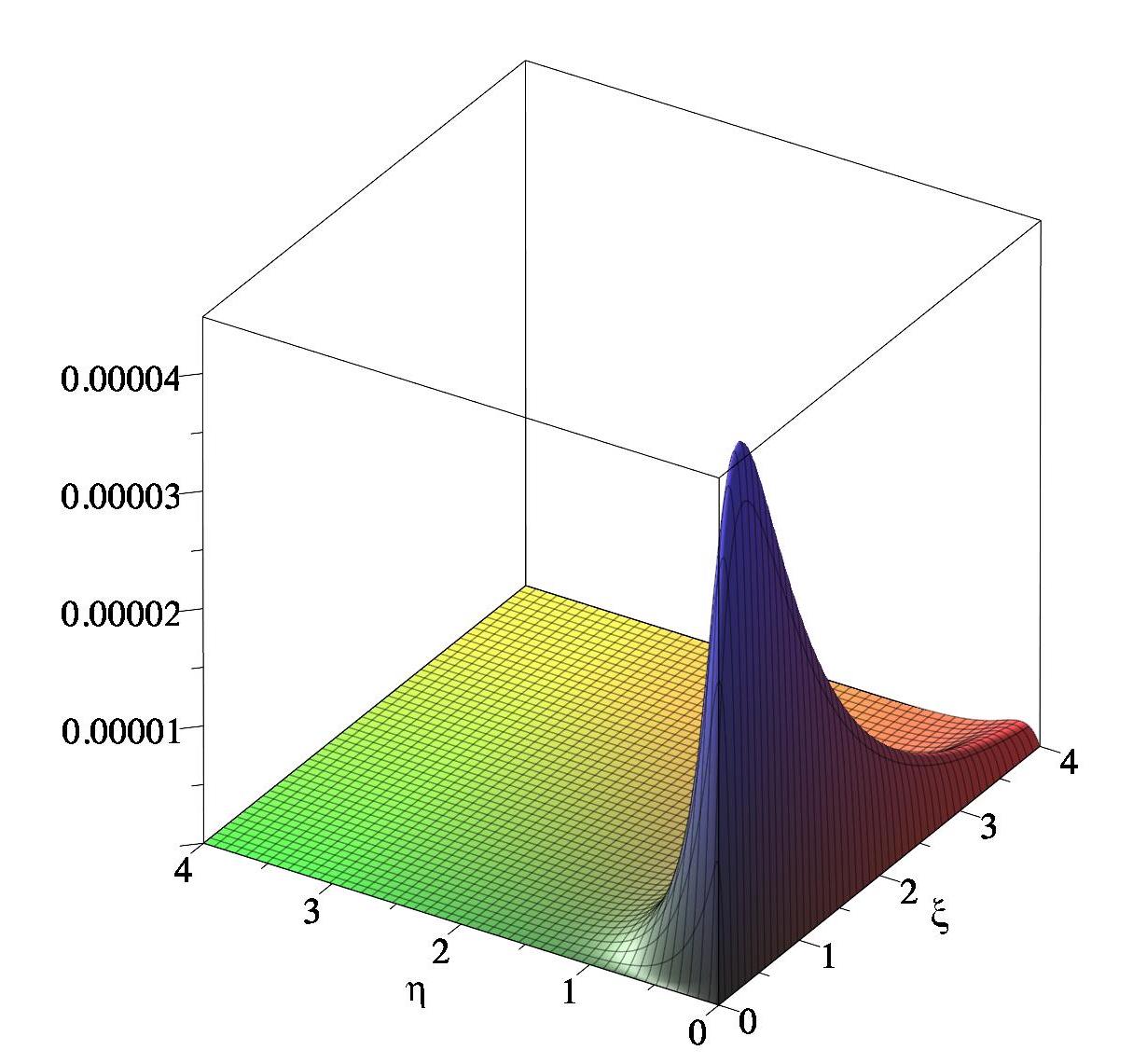}
	\end{minipage}
	\caption{Plot of kernel $ K(\eta,\xi) $ in the case $a=c=5$, $b=1$.}
	\label{fig:FFL_Kernel}
\end{figure}

In order to gain insights on the RFE \eqref{response function for Z} we consider a limit situation.
To this end, we assume that $a \gg 1 $ and that $c \gg 1 $. 
This means that we assume that the time scales $1/a$ and $1/c$ are much smaller than the characteristic time scale in which the signal $S$ changes.
Therefore, \eqref{ODE FFL} implies that for $a \gg 1$ and $c \gg 1 $
\[
X(t) \approx \frac{1}{a} S(t), \  Y(t) \approx \frac{1}{a} \int_{-\infty}^t e^{-b(t-s) } S(s) ds \ \text{ and } \
Z(t)\approx \frac{1}{a^2c} S(t) \int_{-\infty }^t S(s) e^{-b(t-s)} ds. 
\] 
Assume that $S$ has maximum value $S_m$. Then, the function $Z(t) $ reaches values of the order of magnitude of its saturation value $ S_m^2 /a^2 c b $ for times of order $1/b$. This motivates to define the function $\overline S$ such that $S(t) = \overline S (tb)$ and to make the time rescaling $\tau= b t $. We also introduce the new variable $\xi(\tau)=a^2 bc Z(\tau) / S_m^2 $. 
Then for $\tau$ of order $1$ we have
\begin{equation}\label{RE FFL}
\xi(\tau )\approx \overline S(\tau)\int_{-\infty }^\tau \overline S\left( s\right) e^{-(\tau- s)} ds =  \int_{-\infty }^\tau \int_{-\infty }^\tau K_0(\tau-s, \tau- v) \overline S\left( s\right) \overline S(v) ds dv 
\end{equation}
where  
\[
K_0 (s_1, s_2)= e^{-s_1} \delta_0(s_2) \chi_{\mathbb R_*}(s_1) . 
\]
The kernel $K_0 $ is not symmetric, but we can rewrite \eqref{RE FFL} as 
\begin{equation}\label{RFES FFL}
\xi(\tau )\approx  \int_{-\infty }^\tau \int_{-\infty }^\tau \overline{K}_0(\tau-s, \tau- v) \overline S\left( s\right) \overline S(v) ds dv 
\end{equation}
where 
\[
\overline{K}_0(s_1, s_2)=  \frac{1}{2}\left( e^{-s_1} \delta_0(s_2) \chi_{\mathbb R_*}(s_1) +  e^{-s_2} \delta_0(s_1) \chi_{\mathbb R_*}(s_2)\right) . 
\]

We can examine now the response of the system to two signals, namely to $\overline S_1(\tau) =\chi_{[0, \infty)} (\tau ) $ and 
$\overline S_2 (\tau) = \chi_{(-\infty, \overline t]} (\tau)$ with $\overline t >0$.  
Some care is needed to consider signals that are characteristic functions. In practise, the signals $\overline S_1$ and $\overline S_2 $ must be understood as functions that change their values in time scales that are much larger than $1/ba $ and $1/bc $. 
Notice that the response to the signal $\overline S_2 $ yields $ \xi(\tau)=0 $ for every $\tau > \overline t $.
So when the signal $\overline S_2$ suddenly stops, the level of $\xi$ decays instantaneously. On the other hand, in the case of the signal $\overline S_1$ $\xi $ reaches its saturation $1$ after times $\tau$ of order $1$. This is the expected behaviour for the C1FFL system. 
 
We conclude this section by stressing that most of the network motifs in \cite{alon2019introduction} can be formulated and analysed using non-linear RFEs including terms of the form \eqref{nonlinearityintroduction}, as we did in this section for the C1FFL.

\subsection{A nonlinear polymerization model} \label{sec:nonlinearPoly}
We formulate now a non-Markovian, non-linear polymerization model that has some analogies with the classical Becker-D\"oring equations. The  Becker-D\"oring equations are a classical polymerization model that has been extensively studied as a model of phase transitions (see for instance \cite{hingant2017deterministic}). 
The main difference between our model and the classical polymerization models is that in our case the addition of a monomer takes place in  a non-Markovian way.

More precisely, we consider a population of polymers which can have different clusters sizes $\ell $.  
As in Section \ref{sec:linear polimerization}, we assume that polymers grow due to attachment of a monomer and we assume that the reaction $(\ell ) +(1) \rightarrow (\ell+1) $ is non-Markovian. More precisely, we assume that, when a  monomer and a polymer of size $\ell $ bind, they do not form immediately a polymer of size $\ell+1 $. Instead we assume that the addition of a monomer takes place by means of a chain of reactions that we do not try to model in detail, but that we represent with RFEs. Each of the reactions inside the chain can be assumed to be Markovian, but, as we have extensively seen in this paper, the relation between the inputs (in this case a monomer of size $\ell$ and one of size $1$) and the output (which in this case is a monomer of size $\ell+1$) is typically non-Markovian and should be described using the formalism of response functions in this setting. 

The main assumption of this model is that there exists an intermediate state after the attachment of a monomer to a polymer of size $\ell $ that describes the transient state until the polymer can be considered to have size $\ell +1 $.
We denote the density of these intermediate states with $w_\ell $. 
The equations describing the dynamics of the system of polymers is the following
\begin{align*}
&  \frac{d n_\ell}{dt} = - n_1 n_\ell + n_{\ell +1 }-n_\ell + I_\ell \\
  & \frac{d w_\ell }{dt} =  n_1 n_{\ell-1} - I_\ell \\
  & \frac{d n_1}{dt }= 2 n_2 - \sum_{\ell =2 }^\infty n_1 n_\ell - 2 n_1^2+ \sum_{\ell =2}^\infty n_{\ell+1}  + S , 
\end{align*}
where $S$ is a source of monomers. The flux $I_\ell $ is given by 
\[
I_\ell (t)= \int_{-\infty }^t \Psi(t-s) n_1(s) n_{\ell -1}(s) ds \quad \ell \geq 2. 
\]

Notice that $S= \frac{d}{dt} \left( \sum_{\ell=1}^\infty 
 \ell n_\ell + \sum_{\ell=2 } \ell w_\ell   \right) $. 

In contrast with the model described in Section \ref{sec:linear polimerization}, here we have that the influx of polymer of size $\ell $ depends on the history of the number of monomers of size $\ell $ and of size $1$, namely on all the values of $n_\ell $ and $n_1 $ in the time interval $[0, t]$.

\section{Conclusions} \label{sec:conclusion}
In this paper we propose to use the formalism of RFEs to model complex biochemical systems. 
As explained in this paper, the interactions between different parts of biochemical systems can be non-Markovian. 
Since the RFEs can be thought as non-Markovian equations they are well suited to model these interactions. 
The formalism of RFEs has been already extensively used in biology, in the context of population dynamics and epidemiology (see for instance \cite{diekmann1998formulation}).

In this paper we study under which conditions it is possible to reformulate a given model of ODEs as RFEs.
We analyse mainly linear RFEs that are conservative, i.e. they conserve the total number of elements, although we consider also non-linear and non-conservative examples. 
Many applications in biology lead to non-linear models (see  Section \ref{sec:nonlinear}), or non-conservative models (see for instance \cite{zilman2010stochastic}).
It would be interesting to extend the approach presented in this paper to these cases, as it has been done in Population Dynamics and Epidemiology. 

Another possibility would be to consider space dependent response functions. This would allow to consider space dependent models. See for example \cite{galstyan2020proofreading} for an extension of the proofreading model that propose spatial gradients as a way to improve specificity.

As explained in Section \ref{sec:ApproximationIntegralKernel}, the structure of the biochemical reactions of the system impose constraints on the response functions. In particular we proved that systems satisfying the detailed balance condition are associated to completely monotone response functions.
This opens the possibility of deriving the properties of biochemical systems from the   (experimentally measurable) properties of the response functions. 

Writing a model using RFEs allows to describe complex systems of reactions by means of some operators (linear or non-linear) that are characterised by a set of response functions. 
It is therefore relevant to determine if the behaviour of a biochemical system described by a specific system of ODEs can be captured by a REF with response functions having generic features (for instance that are completely monotone).
This would allow to verify the robustness of the behaviour of the system as done in Population Dynamics and Epidemiology. 
An issue that is not considered in this paper is to describe the interactions between different biochemical circuits. This would require to consider a combination of many RFEs models.  

We conclude by stressing that one of the questions addressed in this paper, on the relation between ODEs and RFEs, have some analogies with the one addressed in \cite{diekmann2020finite} and in \cite{diekmann2023systematic} for models appearing in Population Dynamics and Epidemiology. 
The main difference is that in these papers are studied the conditions on the response functions (or kernels) of the renewal equation that allow a reformulation of the system as ODEs.  
Instead, in Section \ref{sec:ODE and RFE} we start from the ODEs models (well suited for systems of biochemical reactions) and we rewrite them using the RFEs formalism, that includes a renewal equation.
As expected, the class of kernels we obtained with this procedure is the same one for which it is shown in \cite{diekmann2020finite} and \cite{diekmann2023systematic} that a reformulation of the REs as ODEs is possible. In particular, it is the class of kernels that correspond to Markovian interactions between the compartments, see Section \ref{subsec:CharMarkov}.

\bigskip 
\bigskip

\textbf{Acknowledgements}
The authors are grateful to K. Thurley for several stimulating discussions that inspired the topic of the paper and for suggesting many references. 
The authors thank also O. Diekmann for many interesting mathematical discussions and for suggesting many references that are relevant to the contents of this paper. 
The authors gratefully acknowledge the financial support of the Hausdorff Research Institute for Mathematics (Bonn), through the mathematics of emergent effects at the University of Bonn
funded through the German Science Foundation (DFG), of the Bonn International Graduate
School of Mathematics at the Hausdorff Center for Mathematics (EXC 2047/1, Project-ID
390685813) funded through the Deutsche Forschungsgemeinschaft (DFG, German Research
Foundation). 
The funders had no role in study design, analysis, decision to publish,
or preparation of the manuscript.

\bibliographystyle{habbrv}
\bibliography{References}
	
\end{document}